\documentclass[11pt]{amsart}
\usepackage{amsmath,amssymb,txfonts}
\usepackage{amssymb}
\usepackage{amsmath}
\usepackage{mathrsfs}
\usepackage{amsmath,amssymb}
\usepackage{authblk}
\usepackage[english]{babel}
\usepackage{amsmath, amsthm, amssymb}
\usepackage[shortlabels]{enumitem}
\usepackage{color}
\usepackage[foot]{amsaddr}

\usepackage{color}

\providecommand{\keywords}[1]
{
  \small	
  \textbf{\textit{Keywords---}} #1
}

\newtheorem{theorem}{Theorem}[section]
\newtheorem{lemma}[theorem]{Lemma}
\newtheorem{proposition}[theorem]{Proposition}
\newtheorem{corollary}[theorem]{Corollary}
\theoremstyle{definition}
\newtheorem{definition}[theorem]{Definition}

\theoremstyle{remark}
\newtheorem{remark}[theorem]{Remark}

\numberwithin{equation}{section}

\begin{document}

\title[Embeddings  of  Functions Spaces]{Embeddings of  Function Spaces   via  the  Caffarelli-Silvestre Extension, Capacities and Wolff potentials}

\author[P. Li]{Pengtao Li}
\address[P. Li]{School of Mathematics and Statistics, Qingdao University, Qingdao, Shandong 266071, China}
\email{ptli@qdu.edu.cn}

\author[S. Shi]{Shaoguang Shi}
\address[S. Shi]{Department of Mathematics, Linyi University, Linyi, 276005, China}
\email{shishaoguang@lyu.edu.cn}

\author[R. Hu]{Rui Hu}
\address[R. Hu]{Department of Mathematics and Statistics,  MacEwan University   Edmonton, Alberta T5J2P2 Canada}
\email{hur3@macewan.ca}

\author[Z. Zhai]{Zhichun Zhai}
\address[Z. Zhai]{Department of Mathematics and Statistics,  MacEwan University   Edmonton, Alberta T5J2P2 Canada}
\email{zhaiz2@macewan.ca}

\thanks{Project supported:  Pengtao Li was in part supported by the  National Natural Science Foundation of
China (\# 11871293 \ \&\ \# 11571217) \& Shandong Natural Science
Foundation of China (\# ZR2017JL008, \# ZR2016AM05)}

%\date{}
%\dedicatory{}

\begin{abstract}
Let $P_{\alpha} f(x,t)$ be the  Caffarelli-Silvestre extension of a smooth function  $f(x): \mathbb{R}^n \rightarrow \mathbb{R}^{n+1}_+:=\mathbb{R}^n\times (0,\infty).$ The purpose of this article is twofold. Firstly,  we want to characterize a nonnegative measure $\mu$ on   $\mathbb{R}^{n+1}_+$ such that $f(x)\rightarrow P_{\alpha} f(x,t)$ induces  bounded embeddings from the Lebesgue spaces $L^p(\mathbb{R}^n)$ to the  $L^q(\mathbb{R}^{n+1}_+,\mu).$
On one hand, these embeddings will be characterized by using a newly introduced
$L^p-$capacity associated with the Caffarelli-Silvestre extension.
In doing so,  the mixed norm estimates of $P_{\alpha} f(x,t),$  the dual form  of  the $L^p-$capacity,  the $L^p-$capacity of  general balls,  and a capacitary strong type inequality  will be established, respectively. On the other hand, when $p>q>1,$ these embeddings will also be characterized in terms of the Hedberg-Wolff  potential of $\mu.$  Secondly, we characterize a nonnegative measure $\mu$ on $\mathbb{R}^{n+1}_+$ such that $f(x)\rightarrow P_{\alpha} f(x,t)$ induces  bounded embeddings from the  homogeneous Sobolev spaces $\dot{W}^{\beta,p}(\mathbb{R}^n)$ to the $L^q(\mathbb{R}^{n+1}_+,\mu)$ in terms of the fractional perimeter of open sets for endpoint cases  and the fractional  capacity for general cases.
\end{abstract}
\date{}
\dedicatory{}


 \maketitle

\tableofcontents
\pagenumbering{arabic}

\section{Introduction}
The fractional Laplace operator $(-\triangle)^{\frac{\alpha}{2}}$ for $\alpha\in (0,2)$
in $\mathbb{R}^n $ is defined  on the Schwartz class through the Fourier transform as
$$\widehat{((-\triangle)^{\frac{\alpha}{2}} f)} (\xi)=(2\pi |\xi|)^{\alpha}\widehat{f}(\xi),$$
where $\widehat{f}(\xi)=\int_{\mathbb{R}^n}e^{-2\pi ix\cdot \xi } f(x) dx$ is the Fourier transform of $f,$
or via the Riesz potential as
$$
(-\triangle)^{\frac{\alpha}{2}} f(x)=\frac{\alpha2^{\alpha}\Gamma\left({(n+\alpha)}/{2}\right)}{2\Gamma(1-\alpha/2)\pi^{{n}/{2}}}\hbox{P.V.}\int_{\mathbb{R}^n}\frac{f(x)-f(y)}{|x-y|^{n+\alpha}}dy.$$
Here $\Gamma(\cdot) $ is the usual Gamma function and $\hbox{P.V.}$ denotes the Cauchy principal value.
The fractional Laplacian  has been widely applied in probability,  finance, physical systems, and engineering problems.

The fractional Laplacian is a nonlocal operator because the value of $(-\triangle)^{\frac{\alpha}{2}} f$ at $x$ depends on the value of $f$ at infinity. This nonlocal property may cause some issues.
 Caffarelli and  Silvestre  in \cite{Caffarelli} localized the nonlocal operator $(-\triangle)^{\frac{\alpha}{2}} $ by adding another variable. They  provided the characterization  for the
fractional Laplacian $(-\triangle)^{\frac{\alpha}{2}}$ by  solving the harmonic extension problem to the upper half-space as the weighted operator that maps the Dirichlet boundary condition to the Neumann condition.

Let $f$ be a regular function in $\mathbb{R}^n.$ We say that $u(x,t)=P_{\alpha} f(x,t)$ is the Caffarelli-Silvestre extension  of $f$ to the upper half-space $\mathbb{R}^{n+1}_{+}:=\mathbb{R}^n\times (0,\infty),$ if $u$ is a solution to the problem
\begin{equation}\label{1}
\left\{
\begin{aligned}
\hbox{div}(t^{1-\alpha}\nabla u)&= 0,\quad \hbox{in}\quad \mathbb{R}^{n+1}_{+};\\
u &= f, \quad \hbox{on} \quad \mathbb{R}^n\times\{t=0\}.
\end{aligned}
\right.
\end{equation}
The Caffarelli-Silvestre extension is well defined for smooth functions through the Poisson  kernel $$p^{\alpha}_t(x)=\frac{c(n,\alpha)t^\alpha}{(|x|^2+t^2)^{{(n+\alpha)}/{2}}}$$ as
$$P_{\alpha} f(x,t)=p^{\alpha}_t\ast f(x,t)
=c(n,\alpha)\int_{\mathbb{R}^n}\frac{f(y)t^{\alpha}}{(|x-y|^2+t^2)^{{(n+\alpha)}/{2}}}dy.$$
Here  $f\ast g$ means the convolution of $f$ and $g,$ and $c(n,\alpha)=\frac{\Gamma({(n+\alpha)}/{2})}{\pi^{n/2}\Gamma({\alpha}/{2})}$ is the normalized constant such that $\int_{\mathbb{R}^n}p^{\alpha}_t(x)dx=1.$
Caffarelli and  Silvestre  \cite{Caffarelli}  proved that
\begin{equation}\label{2}
(-\triangle)^{\frac{\alpha}{2}} f(x)=-c_{\alpha}\lim_{t\rightarrow{0^+}}t^{1-\alpha}\partial_t u(x,t),\ c_{\alpha}=\frac{\Gamma(\alpha/2)}{2^{1-\alpha}\Gamma(1-\alpha/2)}.
\end{equation}
 This characterization has dramatically  popularized the application of  the fractional Laplacian.

The identity (\ref{2}) can be viewed as the consequence of the equality of the energy functionals
$$%\begin{equation}\label{3}
\int_{\mathbb{R}^n}|2\pi\xi|^{\alpha}|\widehat{f}(\xi)|^2d\xi =D(n,\alpha)\int_{\mathbb{R}^{n+1}_{+}}|\nabla P_{\alpha} f(x,t)|^2t^{1-\alpha}dxdt
$$
which is equivalent to
\begin{equation}\label{3}
\|f\|_{\dot{W}^{\alpha/2,2}(\mathbb{R}^n)} =\|P_{\alpha} f(x,t)\|_{\dot{\mathbb{W}}^{1,2}_{\alpha}(\mathbb{R}^{n+1}_{+})}
\end{equation}
up to a multiplication constant. Here, $\dot{\mathbb{W}}^{1,2}_{\alpha}(\mathbb{R}^{n+1}_{+})$ is the weighted Sobolev space defined as
$$\dot{\mathbb{W}}^{1,2}_{\alpha}(\mathbb{R}^{n+1}_{+})=\left\{u(x,t)\in W^{1,1}_{loc}(\mathbb{R}^{n+1}_{+}):  \int_{\mathbb{R}^{n+1}_{+}}|\nabla u(x,t)|^2t^{1-\alpha}dxdt<\infty\right\}.$$
Let $C^{\infty}_{0}(\mathbb R^{n})$ stand for all infinitely smooth functions with compact support in $\mathbb R^{n}$. The homogeneous Sobolev space $\dot{W}^{\beta,p}(\mathbb{R}^{n})$  is the completion of $C_0^\infty(\mathbb{R}^n)$  with respect to the norm
\begin{equation*}
    \|f\|_{\dot{W}^{\beta,p}(\mathbb{R}^{n})}=\left\{
\begin{aligned}
&\|(-\triangle)^{\beta/2}f\|_{L^p(\mathbb{R}^n)}, \quad p\in (1,n/\beta);\\
&\left(\int_{\mathbb{R}^n}\frac{\|\triangle^k_{h}f\|^p_{L^p(\mathbb{R}^n)}}{|h|^{n+p\beta}}dh\right)^{1/p}, \quad p=1 \text{ or }  p=n/\beta,\ \beta\in (0,n),
\end{aligned}
    \right.
\end{equation*}
where $k=1+[\beta], \beta=[\beta]+\{\beta\}$ with $[\beta]\in \mathbb{Z}_+,$ $\{\beta\}\in(0,1)$ and
\begin{equation*}
\triangle^k_{h}f(x)=
\left\{\begin{aligned}
&\triangle^1_h\triangle^{k-1}_h f(x),\quad k>1;\\
&f(x+h)-f(x), \quad k=1.
\end{aligned}\right.
\end{equation*}

Equality (\ref{3}) allows us to identify  fractional  (logarithmic) Sobolev  inequalities as fractional (logarithmic) Sobolev trace inequalities, see \cite{Brandle, Nguyen}.  It also provides  us a way to view  the  fractional perimeters of a Borel set $E\subset\mathbb{R}^n$   as the norm of the Caffarelli-Silvestre extension of the $1_E$ (the characteristic function of $E$).  See \cite{CaffarelliRoquejoffreSavin, Fusco} and the references therein for more details on the fractional perimeter.

Motivated by (\ref{3}), in this paper, we characterize the following two embedding relations via the Caffarelli-Silvestre extension:

{\bf Embedding I: }  Given $\alpha\in (0,2)$ and a nonnegative Radon measure $\mu$ on $\mathbb{R}^{n+1}_{+}$,
\begin{equation}\label{5}
\|P_{\alpha} f(\cdot,\cdot)\|_{L^{q}(\mathbb{R}^{n+1}_+,\mu)}\lesssim \|f\|_{L^p(\mathbb R^n)}.
\end{equation}

 For $0<p,q<\infty$ and a nonnegative Radon measure $\mu$ on $\mathbb{R}^{n+1}_+,$
$L^{q,p}(\mathbb{R}^{n+1}_+,\mu)$ and $L^q(\mathbb{R}^{n+1}_+,\mu)$ denote the Lorentz space and the Lebesgue space of all functions on $\mathbb{R}^{n+1}_+$, respectively, for which
$$\|g\|_{L^{q,p}(\mathbb{R}^{n+1}_+,\mu)}=\left\{\int_0^\infty\left(\mu\left(\{x\in \mathbb{R}^{n+1}_+:\ |g(x)|>s \}\right)\right)^{p/q}ds^p\right\}^{1/p}<\infty$$
and
$$\|g|_{L^q(\mathbb{R}^{n+1}_+,\mu)}=\left(\int_{\mathbb{R}^{n+1}_+}|g(x)|^qd\mu\right)^{1/q}<\infty,$$
respectively. Moreover, we denote by $L^{q,\infty}(\mathbb{R}^{n+1}_+,\mu)$ the set of all $\mu-$measurable functions $g$ on $\mathbb{R}^{n+1}_+$ with
$$\|g\|_{L^{q,\infty}(\mathbb{R}^{n+1}_+,\mu)}=\sup_{s>0}s\left(\mu\left(\left\{x\in \mathbb{R}^{n+1}_+: |g(x)|>s \right\}\right)\right)^{1/q}<\infty.$$

The embedding (\ref{5}) will be characterized by conditions in terms of capacities and Hedberg-Wolff potentials of $\mu.$
Firstly, we introduce the  $L^p-$capacity associated with the Caffarelli-Silvestre extension.

\begin{definition}
Let $1\leq p<\infty$. For a subset $E$ of $\mathbb R^{n+1}_{+}$, let
$$C^{\alpha,p}_{\mathbb{R}^{n+1}_+}(E):=\inf\left\{\|f\|^{p}_{L^{p}(\mathbb R^{n})}:\ f\geq 0\ \&\ P_{\alpha}f(x,t)\geq 1\text{ for all }(x,t)\in E\right\}.$$
When the set $\Big\{f\in {L^{p}(\mathbb R^{n})}:\ f\geq 0\ \&\ P_{\alpha}f(x,t)\geq 1\text{ for all }(x,t)\in E\Big\}$ is empty, we set $C^{\alpha,p}_{\mathbb{R}^{n+1}_+}(E)=0$.

\end{definition}
Then, we establish the mixed norm estimate of $P_{\alpha}f(x,t),$ the dual form of  $C^{\alpha,p}_{\mathbb{R}^{n+1}_+}(\cdot),$ some  basic properties of   $C^{\alpha,p}_{\mathbb{R}^{n+1}_+}(\cdot)$ and a capacitary strong type inequality.

{\bf Embedding II: } Given $\alpha\in (0,2)$ and a nonnegative Radon measure $\mu$ on $\mathbb{R}^{n+1}_{+}$,
\begin{equation}\label{4}
  \|P_{\alpha} f(\cdot,\cdot)\|_{L^{q}(\mathbb{R}^{n+1}_+,\mu)}\lesssim \|f\|_{\dot{W}^{\beta,p}(\mathbb{R}^{n})}
\end{equation}
for $0<\beta<n, 1\leq p\leq {n}/{\beta}$ and $1<q<\infty.$

We will show that the embedding (\ref{4}) when $p=1$ can be characterized by conditions in terms of the  fractional perimeter of open sets for the endpoint cases and fractional  capacities  for general cases.  The fractional perimeter and the  fractional capacity  are defined, respectively, as follows.

\begin{definition}
Let $s\in(0,1)$. The fractional perimeter  is  defined as
$$Per_s(E)=\int_E\int_{\mathbb{R}^n\backslash{E}}\frac{1}{|x-y|^{n+s}}dxdy$$
for  a given measurable set $E\subseteq \mathbb{R}^n.$
\end{definition}
It follows from the definition of $\|f\|_{\dot{W}^{s,1}(\mathbb{R}^n)}$ that
$Per_{s}(E)=\frac{1}{2}\|1_E\|_{\dot{W}^{s,1}(\mathbb{R}^n)}.$ On the other hand, there holds
$$Per_s(E)=\frac{\Gamma((n+s)/2)}{2\pi^{n/2}\Gamma(s/2)}\int_{\mathbb{R}^{n+1}_+}|\nabla u_E(x,t)|^{2}t^{1-s}dxdt,$$
where $ u_E(\cdot,\cdot)$ is the solution to equation (\ref{1}) with $f=1_E.$

For the fractional perimeter,  Ambrosio-DePhilippis-Martinazzi in  \cite{Ambrosio Philippis Martinazzi} proved the  generalized coarea formula:
\begin{equation} \label{7}
\|f\|_{\dot{W}^{s,1}(\mathbb{R}^n)}=2\int_0^\infty Per_s\Big(\Big\{x:\ f(x)>t\Big\}\Big)dt
\end{equation}
for every nonnegative  $f\in \dot{W}^{s,1}(\mathbb{R}^n).$

 Denote by $T(O)$ the tent based on an open subset $O$ of $\mathbb{R}^n:$
$$T(O)=\left\{(x,r)\in \mathbb{R}^{n+1}_+:\ B(x,r)\subseteq O\right\}
$$
with $B(x,r)$ the open ball centered at $x\in\mathbb{R}^n$ with radius  $r>0$.

\begin{definition}
 Let  $\beta \in (0,n)$ and $p\in[1,n/\beta].$
 \item{(i)} The fractional capacity of an arbitrary set $S\subset \mathbb{R}^n,$
denoted by  $Cap_{\mathbb{R}^n}^{\beta,p}(S),$ is defined as
$$Cap_{\mathbb{R}^n}^{\beta,p}(S):=\inf\left\{\|f\|^p_{\dot{W}^{\beta,p}(\mathbb{R}^n)}: \ f\in C_0^\infty(\mathbb{R}^{n}),\  f\geq 0\text{ on }\mathbb{R}^n\ \&\ f\geq  1_S\right\}.$$

 \item{(ii)} For $t\in (0,\infty),$ the $(p,\beta)-$fractional capacity minimizing function associated with both $\dot{W}^{\beta,p}(\mathbb{R}^{n})$ and a nonnegative measure $\mu$ on $\mathbb{R}^{n+1}_+,$ denoted by $c^\beta_p(\mu,t),$ is set as
$$c^\beta_p(\mu,t):=\inf\left\{Cap_{\mathbb{R}^n}^{\beta,p}(O):\ \text{ bounded open }  O\subseteq \mathbb{R}^n,\ \mu(T(O))>t\right\}.$$
\end{definition}

This article is mainly motivated by the work on embeddings like (\ref{5}) and (\ref{4}) via classical/fractional heat equations.  Xiao in \cite{Xiao} studied  the  embeddings  of  the homogeneous Sobolev space $\dot{W}^{1,p}(\mathbb{R}^{n})$ into the Lebesgue space $L^{q} (\mathbb{R}^{n+1}_{+}
 ,\mu),$ under
$(p, q) \in (1,\infty) \times \mathbb{R}_{+},$   via the Gauss-Weierstrass heat
kernel.  For fractional diffusion equations, motivated by Xiao \cite{Xiao},  Zhai in \cite{Zhai}  explored the embeddings of the homogeneous Sobolev space $\dot{W}^{\beta,p}(\mathbb{R}^{n})$ into the Lebesgue space $L^{q} (\mathbb{R}^{n+1}_{+}
 ,\mu).$  By using the  $L^p-$capacities associated with the fractional heat kernel,  Chang-Xiao in \cite{Chang Xiao} and  Shi-Xiao in \cite{Shi Xiao}  established    embeddings similar to     (\ref{5}). For  applications of Poisson like kernel in function spaces, see \cite{Bui,Lenzmann} and the references therein.  

This article will be organized as follows. In Section 2.1, we investigate the  dual form  and  basic properties of the $L^p-$capacity $C^{\alpha,p}_{\mathbb{R}^{n+1}_+}(\cdot)$, and Section \ref{sec-2.2} is devoted to a capacitary strong type inequality corresponding to $C^{\alpha,p}_{\mathbb{R}^{n+1}_+}(\cdot)$. Section \ref{sec-2.3} describes several technical lemmas on the fractional capacity $Cap_{\mathbb{R}^n}^{\beta,p}(\cdot)$.
In Section \ref{sec-3}, we deduce the embedding (\ref{5}) for two cases $p\leq q$ and $p>q,$ respectively. Section \ref{sec-4} studies  the embedding (\ref{4}) for two cases $p\leq q$ and $p>q$.

{\it Some notations}: Let us agree to more conventions. ${\mathsf U}\approx{\mathsf V}$ represents that
there is a constant $c>0$ such that $c^{-1}{\mathsf V}\le{\mathsf
U}\le c{\mathsf V}$ whose right inequality is also written as
${\mathsf U}\lesssim{\mathsf V}$. Similarly, one writes ${\mathsf V}\gtrsim{\mathsf U}$ for
${\mathsf V}\ge c{\mathsf U}$.

\section{Preliminaries on capacities}\label{sec-2}

\subsection{$L^p-$ capacities associated with the Caffarelli-Silvestre extension}\label{sec-2.1}
In this section, $\alpha\in (0,2).$  We will first establish the dual form of the $L^p-$capacity associated with the Caffarelli-Silvestre extension. Then, we will prove some basic properties of the $L^p-$capacity and the capacitary strong type inequality for $C^{\alpha,p}_{\mathbb{R}^{n+1}_+}(\cdot)$.

To establish the adjoint formulation of the foregoing definition, we need to find out the adjoint operator  of $P_{\alpha}$. Note that
$$\int_{\mathbb R^{n+1}_{+}}P_{\alpha}f(x,t)g(x,t)dtdx=\int_{\mathbb R^{n}}f(x)\left(\int_{\mathbb R^{n+1}_{+}}p^{\alpha}_{t}(x-z)g(z,t)dzdt\right)dx$$
holds for all $f\in C^{\infty}_{0}(\mathbb R^{n})$ and $g\in C^{\infty}_{0}(\mathbb R^{n+1}_{+})$.
 The adjoint operator denoted by $P^{\ast}_{\alpha}$ can be defined as
$$(P^{\ast}_{\alpha}g)(x):=\int_{\mathbb R^{n+1}_{+}}p^{\alpha}_{t}(x-z)g(z,t)dzdt,\quad g\in C^{\infty}_{0}(\mathbb R^{n+1}_{+}).$$
For a Borel measure $\mu$ with compact support in $\mathbb R^{n+1}_{+}$, we define
$$P^{\ast}_{\alpha}\mu(x):=\int_{\mathbb R^{n+1}_{+}}p^{\alpha}_{t}(x-z)d\mu(z,t).$$

\begin{proposition}\label{prop-4.1}
Given $p\in (1,\infty)$ and a compact subset $K$ of $\mathbb R^{n+1}_{+}$, let $p'=p/(p-1)$ and $\mathcal M_{+}(K)$ be the class of nonnegative Radon measures supported on $K$. Then
\item{\rm (i)}
$$C^{\alpha,p}_{\mathbb{R}^{n+1}_+}(K)=\sup\left\{\|\mu\|^{p}_{1}:\ \mu\in \mathcal M_{+}(K)\ \&\ \|P^{\ast}_{\alpha}\mu\|_{L^{p'}(\mathbb R^{n})}\leq 1\right\}.$$
\item{\rm (ii)} There exists a measure $\mu_{K}\in\mathcal M_{+}(K)$ such that
$$
\mu_{K}(K)=\int_{\mathbb R^{n}}(P^{\ast}_{\alpha}\mu_{K}(x))^{p'}dx=\int_{\mathbb R^{n+1}_{+}}P_{\alpha}\left(P^{\ast}_{\alpha}\mu_{K}(x)\right)^{p'-1}d\mu_{K}=C^{\alpha,p}_{\mathbb{R}^{n+1}_+}(K).
$$
\end{proposition}

\begin{proof}
(i) We set
$$\widetilde{C}^{\alpha,p}_{\mathbb{R}^{n+1}_+}:=\sup\left\{\|\mu\|^{p}_{1}:\ \mu\in \mathcal M_{+}(K)\ \&\ \|P^{\ast}_{\alpha}\mu\|_{L^{p'}(\mathbb R^{n})}\leq 1\right\}.$$
Let $\mu\in \mathcal M_{+}(K)$ satisfy $\|P^{\ast}_{\alpha}\mu\|_{L^{p'}(\mathbb R^{n})}\leq 1$.
Since $\|\mu\|_{1}=\mu(K)$, for any $f\geq 0$ and $P_{\alpha}f\geq 1_{K}$,
\begin{eqnarray*}
\|\mu\|_{1}&\leq&\int_{\mathbb R^{n+1}_{+}}P_{\alpha}fd\mu=\int_{\mathbb R^{n}}f(x)P^{\ast}_{\alpha}\mu(x)dx
\leq\|f\|_{L^{p}(\mathbb R^{n})}\|P^{\ast}_{\alpha}\mu\|_{L^{p'}(\mathbb R^{n})}\leq\|f\|_{L^{p}(\mathbb R^{n})},
\end{eqnarray*}
which implies  $\widetilde{C}^{\alpha,p}_{\mathbb{R}^{n+1}_+}(K)\leq C^{\alpha,p}_{\mathbb{R}^{n+1}_+}(K)$.

Conversely, define
\begin{equation*}
\left\{\begin{array}{ll}
&\mathcal{X}:=\left\{\mu:\ \mu\in\mathcal M_{+}(K)\quad \&\quad  \mu(K)=1\right\};\\
&\mathcal{Y}:=\left\{f:\ 0< f\in L^{p}(\mathbb R^{n})\quad  \&\quad \|f\|_{L^{p}(\mathbb R^{n})}\leq 1\right\};\\
&\mathcal Z:=\left\{f:\ 0\leq f\in L^{p}(\mathbb R^{n})\quad  \&\quad  P_{\alpha}f\geq 1_{K}\right\};\\
&E(\mu, f):=\int_{\mathbb R^{n}}(P^{\ast}_{\alpha}\mu)(x)f(x)dx=\int_{\mathbb R^{n+1}_{+}} P_{\alpha}f(x,t)d\mu(x,t).
\end{array}
\right.
\end{equation*}

By \cite[Thorem 2.4.1]{AH}, $\min_{\mu\in\mathcal X}\sup_{f\in\mathcal Y}E(\mu, f)
=\sup_{f\in\mathcal Y}\min_{\mu\in\mathcal X}E(\mu, f)$. We can get

\begin{eqnarray*}
\min_{\mu\in\mathcal M_{+}(K)}\frac{\|P^{\ast}_{\alpha}\mu\|_{L^{p'}(\mathbb R^{n})}}{\mu(K)}
&=&\min_{\mu\in\mathcal M_{+}(K)}\frac{\sup_{0\leq f\in L^{p}(\mathbb R^{n}) \& \|f\|_{L^{p}(\mathbb R^{n})}\leq 1}\int_{\mathbb{R}^{n}}f(x) P^{\ast}_{\alpha}\mu (x) dx}{\mu(K)}\\
&=&\min_{\mu\in\mathcal M_{+}(K)}\frac{\sup_{ f\in L^{p}(\mathbb R^{n}) \& \|f\|_{L^{p}(\mathbb R^{n})}=1}\int_{\mathbb{R}^{n}}f(x) P^{\ast}_{\alpha}\mu (x) dx}{\mu(K)}\\
&\leq &\min_{\mu\in\mathcal M_{+}(K)}\sup_{f\in\mathcal Y}\frac{\int_{\mathbb{R}^{n}}f(x) P^{\ast}_{\alpha}\mu (x) dx}{\|f\|_{L^{p}(\mathbb R^{n})}\mu(K)}\\
%&=&\sup_{f\in\mathcal Y}\min_{\mu\in\mathcal X}\left\{\int_{K}\frac{P_{\alpha}f(x,t)}{\|f\|_{L^{p}(\mathbb R^{n})}}d\mu(x,t)\right\}\\
&=&\sup_{f\in\mathcal Y}\min_{\mu\in\mathcal X}\frac{1}{\|f\|_{L^{p}(\mathbb R^{n})}}\left\{\int_{K}P_{\alpha}f(x,t)d\mu(x,t)\right\}\\
&=&\sup_{f\in\mathcal Y}\frac{1}{\|f\|_{L^{p}(\mathbb R^{n})}}\left(\min_{(x,t)\in K}P_{\alpha}f(x,t)\right)\min_{\mu\in\mathcal X}\mu(K)\\
&\leq&\sup_{0< f\in L^{p}(\mathbb R^{n}) }\frac{\left(\min_{(x,t)\in K}P_{\alpha}f(x,t)\right)}{\|f\|_{L^{p}(\mathbb R^{n})}}\\
&\leq & \left(\inf_{0\leq f \in L^{p}(\mathbb R^{n}) \ \&\  \min P_\alpha f=1}\|f\|_{L^{p}(\mathbb R^{n})}\right)^{-1}\\
&=&\left(C^{\alpha,p}_{\mathbb{R}^{n+1}_+}(K)\right)^{-1/p}.
\end{eqnarray*}

For any $\mu\in\mathcal M_{+}(K)$, take $\mu_{1}:=\|P^{\ast}_{\alpha}\mu\|^{-1}_{L^{p'}(\mathbb R^{n})}\mu.$
It is obvious that $\|P^{\ast}_{\alpha}\mu_{1}\|_{L^{p'}(\mathbb R^{n})}=1$, and consequently,
$$\Big(\widetilde{C}^{\alpha,p}_{\mathbb{R}^{n+1}_+}(K)\Big)^{1/p}\geq\sup\left\{\frac{\|\mu\|_{1}}{\|P^{\ast}_{\alpha}\mu\|_{L^{p'}(\mathbb R^{n})}}:\ \mu\in\mathcal M_{+}(K)\right\}=\sup\Big\{\left\|\mu_{1}\right\|_{1},\ \mu_{1}\in\mathcal M_{+}(K)\Big\},$$
which implies
$$\min_{\mu\in\mathcal M_{+}(K)}\left\{\frac{\|P^{\ast}_{\alpha}\mu\|_{L^{p'}(\mathbb R^{n})}}{\mu(K)}\right\}=\min_{\mu\in\mathcal M_{+}(K)}\left\{\frac{\|P^{\ast}_{\alpha}\mu\|_{L^{p'}(\mathbb R^{n})}}{\|\mu\|_{1}}\right\}\geq \left(\widetilde{C}^{\alpha,p}_{\mathbb{R}^{n+1}_+}(K)\right)^{-1/p}.$$
This gives
$\left({C}^{\alpha,p}_{\mathbb{R}^{n+1}_+}(K)\right)^{-1/p}\geq\left(\widetilde{C}^{\alpha,p}_{\mathbb{R}^{n+1}_+}(K)\right)^{-1/p}$. The proof of (i) is completed.

Next let us verify (ii). According to (i),
we select a sequence $\{\mu_{j}\}\subset \mathcal M_{+}(K)$ such that
$$\begin{cases}
&\lim\limits_{j\rightarrow\infty}(\mu_{j}(K))^{p}=C^{\alpha,p}_{\mathbb{R}^{n+1}_+}(K);\\
&\|P^{\ast}_{\alpha}\mu_{j}\|_{L^{p'}(\mathbb R^{n})}\leq 1.
\end{cases}$$
A direct computation implies that
$$C^{\alpha,p}_{\mathbb{R}^{n+1}_+}(K)=\sup\left\{\|\mu\|_{1}^{p}:\ \mu\in\mathcal M_{+}(K)\ \&\ \|P^{\ast}_{\alpha}\mu\|_{L^{p'}(\mathbb R^{n})}=1\right\}.$$
Then, using the fact $\|P^{\ast}_{\alpha}\mu_{j}\|_{L^{p'}(\mathbb R^{n})}=1$, we get
\begin{eqnarray*}
\left|\int_{\mathbb R^{n+1}_{+}}P_{\alpha}f(x,t)d\mu_{j}(x,t)\right|=\left|\int_{\mathbb R^{n}}f(x)P^{\ast}_{\alpha}\mu_{j}(x)dx\right|\leq \|P^{\ast}_{\alpha}\mu_{j}\|_{L^{p'}(\mathbb R^{n})}\|f\|_{L^{p}(\mathbb R^{n})}
\leq\|f\|_{L^{p}(\mathbb R^{n})}.
\end{eqnarray*}
There exists $\mu\in\mathcal M_{+}(K)$ such that $\mu_{j}$ weak $\ast$ convergence to $\mu$. Hence $\mu^{p}(K)=C^{\alpha,p}_{\mathbb{R}^{n+1}_+}(K)$. Since the Lebesgue space $L^{p}(\mathbb R^{n})$ is uniformly convex for $1<p<\infty$,  the set
$$\Big\{f\in C^{\infty}_{0}(\mathbb R^{n}):\ f\geq0\text{ on }\mathbb R^{n}\text{ and } f\geq1\Big\}$$
is a convex set. Following the procedure of \cite[ Theorem 2.3.10]{AH}, we can prove that there exists a unique function denoted by $f_{K}$ such that
\begin{equation}\label{eq-2.1}
\left\{\begin{aligned}
&f_{K}\in L^{p}(\mathbb R^{n});\\
&C_{\mathbb R^{n+1}_{+}}^{\alpha,p}(\{(x,t)\in K:\ P_{\alpha}f_{K}< 1\})=0; \\
&\|f_{K}\|^{p}_{L^{p}(\mathbb R^{n})}=C_{\mathbb R^{n+1}_{+}}^{\alpha,p}(K).
\end{aligned}\right.
\end{equation}
Then we have
\begin{eqnarray}\label{eq-2.2}
(C_{\mathbb R^{n+1}_{+}}^{\alpha,p}(K))^{1/p}=\mu(K)
&\leq&\int_{\mathbb R^{n+1}_{+}}P_{\alpha}f(x,t)d\mu(x,t)\\
&\leq&\int_{\mathbb R^{n}}f_{K}(x)P^{\ast}_{\alpha}\mu(x)dx\nonumber\\
&\leq&\|f_{K}\|_{L^{p}(\mathbb R^{n})}\|P^{\ast}_{\alpha}\mu\|_{L^{p'}(\mathbb R^{n})}\nonumber\\
&=&(C_{\mathbb R^{n+1}_{+}}^{\alpha,p}(K))^{1/p}\|P^{\ast}_{\alpha}\mu\|_{L^{p'}(\mathbb R^{n})},\nonumber
\end{eqnarray}
which implies that $\|P^{\ast}_{\alpha}\mu\|_{L^{p'}(\mathbb R^{n})}\geq 1$. On the other hand, following the procedure of \cite[Proposition 2.3.2, (b)]{AH}, we can prove that $\mu\longmapsto P^{\ast}_{\alpha}\mu(x)$ is lower semi-continuous on $ M_{+}(\mathbb R^{n+1}_{+})$ in the weak* topology, i.e.,
$$P^{\ast}_{\alpha}\mu(x)\leq \liminf_{j\rightarrow\infty}P^{\ast}_{\alpha}\mu_{j}(x),$$
which, together with $\|P^{\ast}_{\alpha}\mu_{j}\|_{L^{p'}(\mathbb R^{n})}\leq 1$, gives $\|P^{\ast}_{\alpha}\mu\|_{L^{p'}(\mathbb R^{n})}\leq 1$. Hence
 $\|P^{\ast}_{\alpha}\mu\|_{L^{p'}(\mathbb R^{n})}=1$. Taking $\mu_{K}=(C^{\alpha,p}_{\mathbb{R}^{n+1}_+}(K))^{1/p'}\mu$ yields
\begin{eqnarray*}
\mu_{K}(K)&=&\int_{K}(C^{\alpha,p}_{\mathbb{R}^{n+1}_+}(K))^{1/p'}d\mu
=(C^{\alpha,p}_{\mathbb{R}^{n+1}_+}(K))^{1/p'}\mu(K)\\
&=&(C^{\alpha,p}_{\mathbb{R}^{n+1}_+}(K))^{1/p'}(C^{\alpha,p}_{\mathbb{R}^{n+1}_+}(K))^{1/p}=C^{\alpha,p}_{\mathbb{R}^{n+1}_+}(K).
\end{eqnarray*}
On the other hand,
\begin{eqnarray*}
\int_{\mathbb R^{n}}(P^{\ast}_{\alpha}\mu_{K}(x))^{p'}dx=\|P^{\ast}_{\alpha}\mu_{K}\|^{p'}_{L^{p'}(\mathbb R^{n})}=(C^{\alpha,p}_{\mathbb{R}^{n+1}_+}(K))\|P^{\ast}_{\alpha}\mu\|^{p'}_{L^{p'}(\mathbb R^{n})}=C^{\alpha,p}_{\mathbb{R}^{n+1}_+}(K).
\end{eqnarray*}
This indicates that
\begin{equation}\label{eq-4.1}
\mu_{K}(K)=\int_{\mathbb R^{n}}(P^{\ast}_{\alpha}\mu_{K}(x))^{p'}dx=C^{\alpha,p}_{\mathbb{R}^{n+1}_+}(K).
\end{equation}
Let $f_{K}$ be the function mentioned above. Then 
 $\mu_{K}\left(\left\{(x,t)\in K:\ P_{\alpha}f_{K}(x,t)\leq 1\right\}\right)=0$.
 
By H\"older's inequality, we can get
\begin{eqnarray}\label{eq-4.2}
C^{\alpha,p}_{\mathbb{R}^{n+1}_+}(K)&=&\mu_{K}(K)\leq\int_{K}P_{\alpha}f_{K}d\mu_{K}\\
&\leq&\int_{\mathbb R^{n}}f_{K}(x)P^{\ast}_{\alpha}\mu_{K}(x)dx\nonumber\\
&\leq&\|f_{K}\|_{L^{p}(\mathbb R^{n})}\|P^{\ast}_{\alpha}\mu_{K}\|_{L^{p'}(\mathbb
 R^{n})}\nonumber\\
&=&(C^{\alpha,p}_{\mathbb{R}^{n+1}_+}(K))^{1/p}(C^{\alpha,p}_{\mathbb{R}^{n+1}_+}(K))^{1/p'}=C^{\alpha,p}_{\mathbb{R}^{n+1}_+}(K).\nonumber
\end{eqnarray}
It follows from (\ref{eq-4.2}) that
$$C^{\alpha,p}_{\mathbb{R}^{n+1}_+}(K)=\int_{\mathbb R^{n}}f_{K}P^{\ast}_{\alpha}\mu_{K}dx.$$
Hence
$$\int_{\mathbb R^{n}}f_{K}P^{\ast}_{\alpha}\mu_{K}dx=\int_{\mathbb R^{n}}(P^{\ast}_{\alpha}\mu_{K})^{p'}dx=C^{\alpha,p}_{\mathbb{R}^{n+1}_+}(K).$$
The above identity implies that
$$C^{\alpha,p}_{\mathbb{R}^{n+1}_+}(K)=\int_{\mathbb R^{n+1}_{+}}P_{\alpha}f_{K}d\mu_{K}=\int_{\mathbb R^{n+1}_{+}}P_{\alpha}(P^{\ast}_{\alpha}\mu_{K})^{p'-1}d\mu_{K},$$
which completes the proof of Proposition \ref{prop-4.1}.
\end{proof}

\begin{remark}
The inequality (\ref{eq-4.2}) implies that
$$\int_{\mathbb R^{n}}f_{K}(x)P^{\ast}_{\alpha}\mu_{K}(x)dx=\|f_{K}\|_{L^{p}(\mathbb R^{n})}\|P^{\ast}_{\alpha}\mu_{K}\|_{L^{p'}(\mathbb R^{n})},$$
i.e., we have equality in H\"older inequality. This means that, due to the normalization chosen, $(f_{K})^{p}=(P^{\ast}_{\alpha}\mu_{K})^{p'}$.
\end{remark}

Below we investigate some basic properties of $C^{\alpha,p}_{\mathbb{R}^{n+1}_+}(\cdot)$.
\begin{proposition}\label{prop-4.2}
\item{\rm (i)} $C^{\alpha,p}_{\mathbb{R}^{n+1}_+}(\varnothing)=0$;
\item{\rm (ii)} If $E_{1}\subseteq E_{2}\subset\mathbb R^{n+1}_{+}$, then $C^{\alpha,p}_{\mathbb{R}^{n+1}_+}(E_{1})\leq C^{\alpha,p}_{\mathbb{R}^{n+1}_+}(E_{2})$;
\item{\rm (iii)} For any sequence $\{E_{j}\}^{\infty}_{j=1}$ of subsets of $\mathbb R^{n+1}_{+}$,
$$C^{\alpha,p}_{\mathbb{R}^{n+1}_+}\left(\bigcup^{\infty}_{j=1}E_{j}\right)\leq \sum^{\infty}_{j=1}C^{\alpha,p}_{\mathbb{R}^{n+1}_+}(E_{j});$$
\item{\rm (iv)} For any $E\subset \mathbb R^{n+1}_{+}$ and any $x_{0}\in \mathbb R^{n}$, $C^{\alpha,p}_{\mathbb{R}^{n+1}_+}(E+(x_{0},0))=C^{\alpha,p}_{\mathbb{R}^{n+1}_+}(E)$.
\end{proposition}

\begin{proof}
The statements (i)\ \&\ (ii) can be deduced from the definition of $C^{\alpha,p}_{\mathbb{R}^{n+1}_+}(\cdot)$ immediately.
For (iii), let $\epsilon>0$. Take $f_{j}\geq0$ such that  $P_{\alpha}f_{j}\geq 1$ on $E_{j}$ and $\int_{\mathbb R^{n}}|f_{j}(x)|^{p}dx\leq C^{\alpha,p}_{\mathbb{R}^{n+1}_+}(E_{j})+\epsilon/2^{j}$. Let $f=\sup\limits_{j\in\mathbb N_{+}}f_{j}$. For any $(x,t)\in \cup^{\infty}_{j=1}E_{j}$, there exists a $j_{0}$ such that $(x,t)\in E_{j_{0}}$ and $P_{\alpha}f_{j_{0}}(x,t)\geq 1$. Hence $P_{\alpha}f(x,t)\geq 1$ on $\cup^{\infty}_{j=1}E_{j}$. On the other hand,
$$\|f\|^{p}_{L^{p}(\mathbb R^{n})}=\int_{\mathbb R^{n}}|f(x)|^{p}dx\leq \sum^{\infty}_{j=1}\int_{\mathbb R^{n}}|f_{j}(x)|^{p}dx=\sum^{\infty}_{j=1}C^{\alpha,p}_{\mathbb{R}^{n+1}_+}(E_{j})+\epsilon,$$
which indicates $C^{\alpha,p}_{\mathbb{R}^{n+1}_+}(\cup^{\infty}_{j=1}E_{j})\leq \sum\limits^{\infty}_{j=1}C^{\alpha,p}_{\mathbb{R}^{n+1}_+}(E_{j}).$

Now we verify (iv). Define $f_{x_{0}}(x)=f(x+x_{0})$. Then $\|f\|_{L^{p}(\mathbb R^{n})}=\|f_{x_{0}}\|_{L^{p}(\mathbb R^{n})}$. If $(x,t)\in E+(x_{0}, 0)$, then $(x-x_{0}, t)\in E$ and vice versa. Take $f\geq0$ such that $P_{\alpha}f\geq 1_{E}$.  Changing of variables reaches
\begin{eqnarray*}
P_{\alpha}f_{x_{0}}(x,t)&=&\int_{\mathbb R^{n}}p^{\alpha}_{t}(x-y)f_{x_{0}}(y)dy=P_{\alpha}f(x-x_{0},t),
\end{eqnarray*}
which implies that
$P_{\alpha}f_{x_{0}}(x,t)\geq 1_{E+(x_{0},0)}$ is equivalent to $P_{\alpha}f(x-x_{0},t)\geq 1_{E}.$
This gives $C^{\alpha,p}_{\mathbb{R}^{n+1}_+}(E+(x_{0},0))=C^{\alpha,p}_{\mathbb{R}^{n+1}_+}(E)$ and the proof of Proposition \ref{prop-4.2} is completed.
\end{proof}
We then prove that the capacity $C^{\alpha,p}_{\mathbb{R}^{n+1}_+}(\cdot)$ is an outer capacity, i.e.,
\begin{proposition}\label{prop-2.1}
For any $E\subset\mathbb R^{n+1}_{+}$,
$$C^{\alpha,p}_{\mathbb{R}^{n+1}_+}(E)=\inf\Big\{C^{\alpha,p}_{\mathbb{R}^{n+1}_+}(O):\ O\supset E, O\text{ open}\Big\}.$$
\end{proposition}
\begin{proof}
Without loss of generality, we assume that $C^{\alpha,p}_{\mathbb{R}^{n+1}_+}(E)<\infty$. By (ii) of Proposition \ref{prop-4.2},
 $$C^{\alpha,p}_{\mathbb{R}^{n+1}_+}(E)\leq\inf\Big\{C^{\alpha,p}_{\mathbb{R}^{n+1}_+}(O):\ O\supset E, O\text{ open}\Big\}.$$
 For $\epsilon\in(0,1)$, there exists a measurable, nonnegative function $f$ such that $P_{\alpha}f\geq 1$ on $E$ and
 $$\int_{\mathbb R^{n}}|f(x)|^{p}dx\leq C^{\alpha,p}_{\mathbb{R}^{n+1}_+}(E)+\epsilon.$$
Since $P_{\alpha}f$ is lower semi-continuous, then the set $O_{\epsilon}:=\{(x,t)\in\mathbb R^{n+1}_{+}:\ P_{\alpha}f(x,t)>1-\epsilon\}$ is an open set.
On the other hand, $E\subset O_{\epsilon}$. This implies that
$$C^{\alpha,p}_{\mathbb{R}^{n+1}_+}(O_{\epsilon})\leq \frac{1}{(1-\epsilon)^{p}}\int_{\mathbb R^{n}}|f(x)|^{p}dx<\frac{1}{(1-\epsilon)^{p}}(C^{\alpha,p}_{\mathbb{R}^{n+1}_+}(E)+\epsilon).$$
The arbitrariness of $\epsilon$ indicates that
$$C^{\alpha,p}_{\mathbb{R}^{n+1}_+}(E)\geq\inf\Big\{C^{\alpha,p}_{\mathbb{R}^{n+1}_+}(O):\ O\supset E, O\text{ open}\Big\}.$$
\end{proof}
An immediate corollary of Proposition \ref{prop-2.1} is the follow result.
\begin{corollary}
If $\{K_{j}\}^{\infty}_{j=1}$  is a decreasing sequence of compact sets, then
$$C^{\alpha,p}_{\mathbb{R}^{n+1}_+}(\cap^{\infty}_{j=1}K_{j})=\lim_{j\rightarrow\infty} C^{\alpha,p}_{\mathbb{R}^{n+1}_+}(K_{j}).$$
\end{corollary}
\begin{proof}
Let $O$ be any open set satisfying $\cap^{\infty}_{j=1}K_{j}\subset O$. Then for some $j$, $K_{j}\subset O$, which, together with Proposition \ref{prop-2.1}, gives
$$C^{\alpha,p}_{\mathbb{R}^{n+1}_+}(\cap^{\infty}_{j=1}K_{j})\leq \lim_{j\rightarrow\infty} C^{\alpha,p}_{\mathbb{R}^{n+1}_+}(K_{j})\leq
\inf_{\cap^{\infty}_{j=1}K_{j}\subset O}C^{\alpha,p}_{\mathbb{R}^{n+1}_+}(O)=C^{\alpha,p}_{\mathbb{R}^{n+1}_+}(\cap^{\infty}_{j=1}K_{j}).$$
\end{proof}
\begin{proposition}\label{prop-2.2}
Let $1<p<\infty$. If $\{E_{j}\}^{\infty}_{j=1}$ is an increasing sequence of arbitrary subsets of $\mathbb R^{n}$, then
$$C^{\alpha,p}_{\mathbb{R}^{n+1}_+}(\cup^{\infty}_{j=1}E_{j})=\lim_{j\rightarrow\infty}C^{\alpha,p}_{\mathbb{R}^{n+1}_+}(E_{j}).$$
\end{proposition}
\begin{proof}
Since $\{E_{j}\}^{\infty}_{j=1}$ is increasing, then
$$C^{\alpha,p}_{\mathbb{R}^{n+1}_+}(\cup^{\infty}_{j=1}E_{j})\geq\lim_{j\rightarrow\infty}C^{\alpha,p}_{\mathbb{R}^{n+1}_+}(E_{j}).$$
Conversely, without loss generality, we assume that $\lim_{j\rightarrow\infty}C^{\alpha,p}_{\mathbb{R}^{n+1}_+}(E_{j})$ is finite. For each $j$, let $f_{E_{j}}$ be the unique function such that $f_{E_{j}}\geq1$ on $E_{j}$ and $\|f_{E_{j}}\|^{p}_{L^{p}}=C^{\alpha,p}_{\mathbb{R}^{n+1}_+}(E_{j})$. Then for $i<j$, it holds that $P_{\alpha}f_{E_{j}}\geq 1$ on $E_{i}$ and further, $P_{\alpha}((f_{E_{i}}+f_{E_{j}})/2)\geq 1$ on $E_{i}$, which means that
$$\int_{\mathbb R^{n}}((f_{E_{i}}+f_{E_{j}})/2)^{p}dx\geq C^{\alpha,p}_{\mathbb{R}^{n+1}_+}(E_{i}).$$
By \cite[Corollary 1.3.3]{AH}, the sequence $\{f_{E_{j}}\}^{\infty}_{j=1}$ converges strongly to a function $f$ satisfying
$$\|f\|_{L^{p}(\mathbb{R}^n)}^{p}=\lim_{j\rightarrow\infty}C^{\alpha,p}_{\mathbb{R}^{n+1}_+}(E_{j}).$$
Similar to \cite[Proposition 2.3.12]{AH}, we can prove that $P_{\alpha}f\geq 1$ on $\cup^{\infty}_{j=1}E_{j}$, except possibly on a countable union of sets with $C^{\alpha,p}_{\mathbb{R}^{n+1}_+}(\cdot)$ zero. Hence
$$\lim_{j\rightarrow\infty}C^{\alpha,p}_{\mathbb{R}^{n+1}_+}(E_{j})\geq\int_{\mathbb R^{n}}|f(x)|^{p}dx\geq C^{\alpha,p}_{\mathbb{R}^{n+1}_+}(\cup^{\infty}_{j=1}E_{j}).$$
\end{proof}
As a corollary of Proposition \ref{prop-2.2}, we can get
\begin{corollary}
Let $O$ be an open subset of $\mathbb R^{n+1}_{+}$. Then
$$C^{\alpha,p}_{\mathbb{R}^{n+1}_+}(O)=\sup\left\{C^{\alpha,p}_{\mathbb{R}^{n+1}_+}(K):\ \text{ compact }K\subset O\right\}.$$
\end{corollary}
\begin{proof}
It is obvious that
$$C^{\alpha,p}_{\mathbb{R}^{n+1}_+}(O)\geq\sup\left\{C^{\alpha,p}_{\mathbb{R}^{n+1}_+}(K):\ \text{ compact }K\subset O\right\}.$$
Conversely, any open set $O$ in $\mathbb R^{n}$ is the union of an increasing sequence of compact sets denoted by $\{K_{j}\}_{j=1}^{\infty}$. Then it follows from Proposition \ref{prop-2.2} that
\begin{eqnarray*}
C^{\alpha,p}_{\mathbb{R}^{n+1}_+}(O)&=&C^{\alpha,p}_{\mathbb{R}^{n+1}_+}(\cup^{\infty}_{j=1}K_{j})=\lim_{j\rightarrow\infty}C^{\alpha,p}_{\mathbb{R}^{n+1}_+}(K_{j})\\
&\leq&\sup\left\{C^{\alpha,p}_{\mathbb{R}^{n+1}_+}(K):\ \text{ compact }K\subset O\right\}.
\end{eqnarray*}

\end{proof}

For  $t_{0}\geq 0$, $x_{0}\in \mathbb R^{n}$ and $r_{0}>0$, define the ball in $\mathbb R^{n+1}_{+}$ as
$$B_{r_{0}}(x_{0},t_{0})=\left\{(x,t)\in \mathbb R^{n+1}_{+}:\ |x-x_{0}|<r_{0}/2,\ r_{0}<t-t_{0}<2r_{0} \right\}.$$
Let $t=r_{0}s$, $x=r_{0}y$ and $f_{r_{0}}(x)=f(r_{0}x)$. We can get $(x,t)\in B_{r_{0}}(0,0)$ is equivalent to $(y,s)\in B_{1}(0,0)$. A direct computation, together with a change of variable $z=r_{0}u$, gives
\begin{eqnarray*}
P_{\alpha}f(x,t)&=&\int_{\mathbb R^{n}}\frac{t^{\alpha}}{(t^{2}+|x-z|^{2})^{(n+\alpha)/2}}f(z)dz\\
&=&\int_{\mathbb R^{n}}\frac{s^{\alpha}}{(|y-u|^{2}+s^{2})^{(n+\alpha)/2}}f_{r_{0}}(u)du= P_{\alpha}f_{r_{0}}(y,r).
\end{eqnarray*}
Then
$$P_{\alpha}f(x,t)\geq 1 \quad \forall (x,t)\in B_{r_{0}}(0,0)\Longleftrightarrow P_{\alpha}f_{r_{0}}(s,y)\geq 1 \quad \forall (y,s)\in B_{1}(0,0).$$
This means that
$$C^{\alpha,p}_{\mathbb{R}^{n+1}_+}(B_{r_{0}}(0,0))=r^{n}_{0}C^{\alpha,p}_{\mathbb{R}^{n+1}_+}(B_{1}(0,0)).$$

Now we investigate the $L^{p}$-capacity of general balls $B_{r_{0}}(x_{0}, t_{0})$. We first give a space-time estimate for $P_{\alpha} f$.
\begin{lemma}\label{le-4.1}
Let $1\leq r\leq  p\leq\infty.$ Then
$\|P_{\alpha}f(\cdot,t)\|_{L^{p}(\mathbb{R}^n)}\leq t^{n(1/p-1/r)}\|f\|_{L^{r}(\mathbb{R}^n)}.$
\end{lemma}

\begin{proof}
Let $q$ obey $1/p+1=1/r+1/q$.
It is easy to see that
\begin{eqnarray*}
\|p^{\alpha}_{t}(\cdot)\|_{L^{q}(\mathbb{R}^n)}=\left(\int_{\mathbb R^{n}}\left|\frac{t^{\alpha}}{(t^{2}+|x|^{2})^{(n+\alpha)/2}}\right|^{q}dx\right)^{1/q} \lesssim t^{n(1/q-1)}.
\end{eqnarray*}
It follows from  Young's inequality that
\begin{eqnarray*}
\|P_{\alpha}f(\cdot,t)\|_{L^{p}(\mathbb{R}^n)}\leq\|p^{\alpha}_{t}\ast f\|_{L^{p}(\mathbb{R}^n)}\leq\|p^{\alpha}_{t}(\cdot)\|_{L^{q}(\mathbb{R}^n)}\|f\|_{L^r(\mathbb{R}^n)} \leq  t^{n(1/p-1/r)}\|f\|_{L^{r}(\mathbb{R}^n)}.
\end{eqnarray*}
\end{proof}

\begin{theorem}
Let $(q,p,r)$  be a triple satisfying  $1/q=n(1/r-1/p)$, where
$$1<r\leq p<
\begin{cases}
nr/(n-1), n>2r;\\
\infty, n\leq 2r.
\end{cases}$$
Given $f\in L^{r}(\mathbb R^{n})$. Then for $I=(0,T)$ with $0<T\leq \infty$,
$P_{\alpha}f(\cdot,\cdot)\in L^{q}(I; L^{p}(\mathbb{R}^n))\cap
C_{b}(I; L^{r}(\mathbb{R}^n))$ with the estimate
$$\|P_{\alpha} f(\cdot,\cdot)\|_{L^{q}(I; L^{p}(\mathbb{R}^n))}\lesssim \|f\|_{L^{r}(\mathbb{R}^n)}.$$
Here $C_{b}(I; L^{r}(\mathbb{R}^n))$ denotes the space  of bounded continuous  functions from $I$ to $L^{r}(\mathbb{R}^n).$
\end{theorem}

\begin{proof}
{\it Case 1: $p=r\ \&\ q=\infty$}.  By Lemma \ref{le-4.1}, we can get
\begin{eqnarray*}
\|P_{\alpha}f\|_{L^{\infty}(I; L^{r}(\mathbb{R}^n))}&=&\sup_{t>0}\|p^{\alpha}_{t}\ast f\|_{L^{r}(\mathbb{R}^n)}
\leq\sup_{t>0}t^{-n(1/r-1/r)}\|f\|_{L^{r}(\mathbb{R}^n)}
\leq \|f\|_{L^{r}(\mathbb{R}^n)}.
\end{eqnarray*}

{\it Case 2: $p\neq r$}. Denote by $F_{t}(f)=\|p^{\alpha}_{t}\ast f\|_{L^{p}(\mathbb{R}^n)}$. Applying Lemma \ref{le-4.1} again, we also obtain
\begin{eqnarray*}
F_{t}(f)&=&\|p^{\alpha}_{t}\ast f\|_{L^{p}(\mathbb{R}^n)}\leq t^{-n(1/r-1/p)}\|f\|_{L^{r}(\mathbb{R}^n)}\leq t^{-1/q}\|f\|_{L^{r}(\mathbb{R}^n)}.
\end{eqnarray*}
On the other hand,
\begin{eqnarray*}
\left|\left\{t:\ |F_{t}(f)|>\tau\right\}\right|&\leq& \left|\left\{t:\ t^{-1/q}\|f\|_{L^{r}(\mathbb{R}^n)}>\tau\right\}\right|
\leq\left|\left\{t:\ \|f\|^{q}_{L^{r}(\mathbb{R}^n)}/\tau^{q}>t\right\}\right|
\leq \|f\|^{q}_{L^{r}(\mathbb{R}^n)}/\tau^{q}.
\end{eqnarray*}
The above estimate implies that $F_{t}$ is a weak $(r,q)$ type operator. Noticing that
$$|p^{\alpha}_{t}\ast f(x)|\leq\int_{\mathbb R^{n}}\frac{t^{\alpha}}{(t^{2}+|x-y|^{2})^{(n+\alpha)/2}}|f(y)|dy,$$
we have
$$|F_{t}(f)|=\|P_{\alpha}f(\cdot,t)\|_{L^{p}(\mathbb{R}^n)}\leq t^{n(1/p-1/p)}\|f\|_{L^{p}(\mathbb{R}^n)},$$
i.e., $F_{t}$ is also a $(p,\infty)$ type operator. For any triplet $(q,p,r)$, we can choose another  triplet $(q_{1}, p_{1}, r_{1})$ such that $q_{1}<q<\infty$, $r_{1}<r<p$ and
$$\begin{cases}
1/q=\theta/q_{1}+(1-\theta)/\infty,\\
1/r=\theta/r_{1}+(1-\theta)/p.
\end{cases}$$
The Marcinkiewicz interpolation theorem implies that $F_{t}$ is a strong $(r,q)$ type operator and satisfies
\begin{equation}\label{eq-4.3}
\|P_{\alpha}f(\cdot,\cdot)\|_{L^{q}(I; L^{p}(\mathbb{R}^n))}\leq \|f\|_{L^{r}(\mathbb{R}^n)}.
\end{equation}

Below we prove that $P_{\alpha}f(\cdot,\cdot)\in C_{b}(I; L^{r}(\mathbb{R}^n))$. Since $u(\cdot,\cdot):=P_{\alpha}f(\cdot,\cdot)$ satisfies the equations (\ref{1}). Then by a direct computation, we can  verify that $u(\cdot,\cdot)$ also satisfies:
\begin{equation}\label{eq-2.4}
  \begin{cases}
  \Delta_{x}u+\frac{1-\alpha}{t}u_{t}+u_{tt}=0,&\ \ (x,t)\in\ \mathbb{R}^{n}\times(0,\infty);\\
  u(x,0)=f(x),&\ \ x\in\ \mathbb{R}^{n},
  \end{cases}
\end{equation}
see \cite{Caffarelli}. Then by the Poisson type formula obtained by P. Stinga and J. Torrea (cf. \cite[(1.9)]{sti}),
\begin{eqnarray*}
 u(x,t)\equiv P_{\alpha}f(x) = C_{\alpha}\int^{\infty}_{0}e^{-s}e^{-t^{2}(-\Delta)/4s}f(x)s^{\alpha/2-1}ds,\quad x\in\mathbb{R}^{n}\ \&\ t>0.
\end{eqnarray*}
By Minkowski's inequality, we get
\begin{eqnarray*}
\Big|\|u(\cdot,t)\|_{L^{r}(\mathbb{R}^n)}-\|u(\cdot,t_{0})\|_{L^{r}(\mathbb{R}^n)}\Big|&\leq&\|u(\cdot,t)-u(\cdot,t_{0})\|_{L^{r}(\mathbb{R}^n)}\\
&\leq&C_{\alpha}\int^{\infty}_{0}e^{-s}\|e^{-t^{2}(-\Delta)/4s}f(\cdot)-e^{-t_{0}^{2}(-\Delta)/4s}f(\cdot)\|_{L^{r}(\mathbb{R}^n)}s^{\alpha/2-1}ds.
\end{eqnarray*}
Notice that
$$\|e^{-t^{2}(-\Delta)/4s}f(\cdot)-e^{-t_{0}^{2}(-\Delta)/4s}f(\cdot)\|_{L^{r}(\mathbb{R}^n)}\leq C\|f\|_{L^{r}(\mathbb{R}^n)}.$$
Since $\{e^{-t(-\Delta)}\}_{t>0}$ is a strongly continuous semigroups in $L^{r}(\mathbb R^{n})$, then by the dominate convergence theorem, we can deduce that
$$\lim_{t\rightarrow t_{0}}\|u(\cdot,t)\|_{L^{r}(\mathbb{R}^n)}=\|u(\cdot,t_{0})\|_{L^{r}(\mathbb{R}^n)}.$$

\end{proof}

For $p\in (1,\infty)$, choose $(\tilde{p}, \tilde{q})$ such that
$$\begin{cases}
1\leq p\leq \tilde{p}<{np}/{(n-1)};\\
1/\tilde{q}=n(1/p-1/\tilde{p}).
\end{cases}$$

\begin{theorem}\label{th-4.1}
If $ 1\leq p<\infty,$ then for $(x_0,r_0)\in\mathbb R^{n+1}_+$,
$$ r_0^{n}\lesssim C^{\alpha,p}_{\mathbb{R}^{n+1}_+}\left(B_{r_0}(x_0, t_0)\right)\lesssim
(t_0+r_0)^{pn}r_{0}^{n(1-p)}.$$
 Particularly, if $t_0 \lesssim r_0$, then
$$r_0^{n}\lesssim C^{\alpha,p}_{\mathbb{R}^{n+1}_+}\left(B_{r_0}(x_0,t_0)\right)\lesssim  r_0^{n}.$$
\end{theorem}
\begin{proof}
Let $B_{r_{0}}(x_{0},t_{0})$ be a ball in $\mathbb R^{n+1}_{+}$. For any $f\geq 0$ satisfying $P_{\alpha}(f)\geq 1_{B_{r_{0}}(x_{0}, t_{0})}$,
by (\ref{eq-4.3}), we can get
\begin{eqnarray*}
r_{0}^{n/\tilde{p}+1/\tilde{q}}
&\lesssim&\left(\int_{r_{0}<t-t_{0}<2r_{0}}\left(\int_{|x-x_{0}|<r_{0}/2}|P_{\alpha}f(x,t)|^{\tilde{p}}dx\right)^{\tilde{q}/\tilde{p}}dt\right)^{1/\tilde{q}}
\lesssim\|f\|_{L^{p}(\mathbb{R}^n)},
\end{eqnarray*}
which gives $r_{0}^{(n/\tilde{p}+1/\tilde{q})p}\leq \|f\|_{L^p(\mathbb{R}^n)}^{p}$ and $r_{0}^{n}\leq C^{\alpha,p}_{\mathbb{R}^{n+1}_+}(B_{r_{0}}(x_{0}, t_{0}))$.

For the converse, choose $f=1_{\{x\in\mathbb R^{n}:\ |x-x_{0}|<r_{0}/2\}}$. If $(x,t)\in B_{r_{0}}(x_{0}, t_{0})$, then $|x-x_{0}|<r_{0}/2$ and $r_{0}<t-t_{0}<2r_{0}$, i.e., $r_{0}+t_{0}\leq t<t_{0}+2r_{0}$. We can get there exists a constant $c_{\alpha}$ such that
$$\frac{t^{\alpha}}{(t^{2}+|x-y|^{2})^{(n+\alpha)/2}}\geq \frac{(t_{0}+r_{0})^{\alpha}}{(t_{0}+2r_{0}+3r_{0}/2)^{(n+\alpha)/2}}\geq \frac{c_{\alpha}}{(t_{0}+r_{0})^{n}},$$
which gives
\begin{eqnarray*}
P_{\alpha}f(x,t)&=&\int_{|y-x_{0}|<r_{0}/2}\frac{t^{\alpha}}{(t^{2}+|x-y|^{2})^{(n+\alpha)/2}}dy\geq\frac{c_{\alpha}r^{n}_{0}}{(t_{0}+r_{0})^{n}},
\end{eqnarray*}
equivalently,
$P_{\alpha}\left(\left(1+{t_{0}}/{r_{0}}\right)^{n}c^{-1}_{\alpha}f\right)\geq 1.$
By the definition of $C^{\alpha,p}_{\mathbb{R}^{n+1}_+}(\cdot)$, we obtain
\begin{eqnarray*}
C^{\alpha,p}_{\mathbb{R}^{n+1}_+}\left(B_{r_{0}}(x_{0}, t_{0})\right)&\leq&\Big\|\frac{(t_{0}+r_{0})^{n}}{r^{n}_{0}}f\Big\|^{p}_{L^{p}(\mathbb{R}^n)}\\
&=&\frac{(t_{0}+r_{0})^{pn}}{r^{pn}_{0}}\int_{\mathbb R^{n}}\left(1_{\{x\in\mathbb R^{n}:\ |x-x_{0}|<r_{0}/2\}}(y)\right)^{p}dy\\
&=&(t_{0}+r_{0})^{pn}r_{0}^{(1-p)n},
\end{eqnarray*}
which completes the proof of Theorem \ref{th-4.1}.
\end{proof}
\subsection{Capacitary strong type inequality}\label{sec-2.2}
In order to establish  the embeddings  of $L^{p}(\mathbb R^{n})$ to  $L^{q}(\mathbb R^{n+1}_{+})$  via $P_{\alpha}$,  we need to  prove the capacitary strong and weak inequalities. Let $L^{p}_{+}(\mathbb R^{n})$ denote the class of all nonnegative functions in $L^{p}(\mathbb R^{n})$.
\begin{lemma}\label{le-4.2}
Let $p\in (1,\infty)$. Then
\begin{equation}\label{eq-2.3}
\int^{\infty}_{0}C^{\alpha,p}_{\mathbb{R}^{n+1}_+}\left(\left\{(x,t)\in\mathbb R^{n+1}_{+}:\ P_{\alpha}f(x,t)\geq\lambda\right\}\right)d\lambda^{p}\lesssim \|f\|_{L^{p}(\mathbb{R}^n)}^{p}\quad  \forall\ f\in L^{p}_{+}(\mathbb R^{n}).
\end{equation}
\end{lemma}
\begin{proof}
Firstly,  we  verify the inequality for any nonnegative function in $C^{\infty}_{0}(\mathbb R^{n})$. Let $f\in C^{\infty}_{0}(\mathbb R^{n})$ and for $r>0$, set
$$E_{j}:=\left\{(x,t)\in \overline{B_{r}(0,0)}:\ P_{\alpha}f(x,t)\geq 2^{j}\right\}.$$
Note that $E_j$  depends on $r.$  Since $f\in C^{\infty}_{0}(\mathbb R^{n}),$  $P_{\alpha}f(x,t)$ is continuous. Thus, $E_j$ is the intersection of $\overline{B_r(0,0)}$ and a close set and so $E_j$ is compact.  

Let $\mu_{j}$ stand for the measure corresponding to $E_{j}$ such that
$$\mu_{j}(E_{j})=\int_{\mathbb R^{n}}(P_{\alpha}^{\ast}\mu_{j}(x))^{p'}dx=\int_{\mathbb R^{n+1}_{+}}P_{\alpha}(P^{\ast}_{\alpha}\mu_{j})^{p'-1}d\mu_{j}=C^{\alpha,p}_{\mathbb{R}^{n+1}_+}(E_{j}).$$

Let $S:=\sum\limits^{\infty}_{j=-\infty}2^{jp}\mu_{j}(E_{j})$ and $T:=\big\|\sum\limits^{\infty}_{j=-\infty}2^{j(p-1)}(P^{\ast}_{\alpha}\mu_{j})\big\|^{p'}_{L^{p'}}$. 
Because $f\in C^{\infty}_{0}(\mathbb R^{n})$, there exists a  positive integer $j_{0}$ such that $E_{j}$ are empty sets for $j>j_{0}$, i.e., $\mu_{j}(E_{j})=C^{\alpha,p}_{\mathbb R^{n+1}_{+}}(E_{j})=0, j>j_{0}$. Hence $\sum\limits^{\infty}_{j=1}2^{jp}\mu_{j}(E_{j})<\infty$. On the other hand, 
\begin{eqnarray*}
\sum^{-1}_{j=\infty}2^{jp}\mu_{j}(E_{j})&=&\sum^{-1}_{j=\infty}2^{jp}C^{\alpha,p}_{\mathbb R^{n+1}_{+}}(E_{j})\\
&\leq&\sum^{-1}_{j=\infty}2^{jp}C^{\alpha,p}_{\mathbb R^{n+1}_{+}}(B_{r}(0,0))\\
&\lesssim&r^{n}\sum^{-1}_{j=\infty}2^{jp},
\end{eqnarray*}
which means that $S<\infty$. Then by H\"older's inequality,
\begin{equation}\label{eq-2.5}
S\leq\sum^{\infty}_{j=-\infty}2^{j(p-1)}\int_{\mathbb R^{n+1}_{+}}P_{\alpha}fd\mu_{j}
\leq\int_{\mathbb R^{n}}f\left(\sum^{\infty}_{j=-\infty}2^{j(p-1)}(P^{\ast}_{\alpha}\mu_{j})\right)dx
\leq T^{1/p'}\|f\|_{L^{p}(\mathbb{R}^n)}.
\end{equation}

Below we claim that 
\begin{equation}\label{eq-2.6}
T\lesssim S.
\end{equation}
We divide  the proof of this claim into two cases.

{\it Case 1: $2\leq p<\infty$.} For $k=0, \pm1, \pm2,\ldots,$ let
$$\begin{cases}
\sigma_{k}(x)=\sum\limits^{\infty}_{j=k}2^{j(p-1)}P^{\ast}_{\alpha}\mu_{j}(x);\\
\sigma(x)=\sum\limits^{\infty}_{j=-\infty}2^{j(p-1)}P^{\ast}_{\alpha}\mu_{j}(x).
\end{cases}$$
Since $f\in C^{\infty}_{0}(\mathbb R^{n})$, the sets $E_{j}$ are empty for sufficiently large $j$, i.e., there exists a positive integer $j_{0}$ such that
$C^{\alpha,p}_{\mathbb{R}^{n+1}_+}(E_{j})=0,\ j> j_{0}$. Then by (ii) of Proposition \ref{prop-4.1},
\begin{eqnarray*}
\|\sigma_{k}\|_{L^{p'}(\mathbb{R}^n)}&=&\Big\|\sum\limits^{\infty}_{j=k}2^{j(p-1)}P^{\ast}_{\alpha}\mu_{j}\Big\|_{L^{p'}(\mathbb{R}^n)}\\
&\leq&\sum\limits^{\infty}_{j=k}2^{j(p-1)}\left\|P^{\ast}_{\alpha}\mu_{j}\right\|_{L^{p'}(\mathbb{R}^n)}\\
&\simeq&\sum\limits^{j_{0}}_{j=k}2^{j(p-1)}(C^{\alpha,p}_{\mathbb{R}^{n+1}_+}(E_{j}))^{1/p'}<\infty,
\end{eqnarray*}
which implies that for any $k$, $\sigma_{k}\in L^{p'}(\mathbb R^{n})$. For $x\in\mathbb R^{n}$,  the sequence $\{\sigma_{k}(x)\}$ is increasing as $k\rightarrow-\infty$ since $P^{\ast}_{\alpha}\mu_{j}(x)\geq0$. If there exists an upper bound for $\{\sigma_{k}(x)\}$, we know that $\sigma_{k}(x)$ convergences to the series $\sum\limits^{\infty}_{j=-\infty}2^{j(p-1)}P^{\ast}_{\alpha}\mu_{j}(x)$. If the sequence  $\{\sigma_{k}(x)\}$ is unbounded, then $\sigma_{k}(x)$ tends to $\infty$ as $k\rightarrow-\infty$. Without loss of generality, formally, we write $\lim\limits_{k\rightarrow -\infty}\sigma_{k}(x)=\sigma(x)$. We can get
\begin{eqnarray*}
\sum^{\infty}_{k=n}\Big(\sigma_{k}(x)^{p'}-\sigma_{k+1}(x)^{p'}\Big)&=&\sum^{\infty}_{k=n}\sigma_{k}(x)^{p'}-\sum^{\infty}_{k=n+1}\sigma_{k}(x)^{p'}
=\sigma_{n}(x)^{p'}.
\end{eqnarray*}
By the mean value theorem, noticing that $\sigma_{k}(x)\geq\sigma_{k+1}(x)$, we have
\begin{eqnarray}\label{eq-2.7}
\sigma(x)^{p'}&=&\left(\sum^{\infty}_{j=-\infty}2^{j(p-1)}P^{\ast}_{\alpha}\mu_{j}(x)\right)^{p'}=p'\sum^{\infty}_{k=-\infty}(\sigma_{k}(x))^{p'-1}\\
&=&\lim_{n\rightarrow-\infty}\sigma_{n}(x)^{p'}\nonumber\\
&=&\lim_{n\rightarrow-\infty}\sum^{\infty}_{k=n}\Big(\sigma_{k}(x)^{p'}-\sigma_{k+1}(x)^{p'}\Big)\nonumber\\
&\leq&p'(\sigma_{k}(x))^{p'-1}\Big(\sigma_{k}(x)-\sigma_{k+1}(x)\Big)\nonumber\\
&=&p'\sum^{\infty}_{k=-\infty}(\sigma_{k}(x))^{p'-1}
2^{k(p-1)}P^{\ast}_{\alpha}\mu_{k}(x).\nonumber
\end{eqnarray}
Using Cauchy-Schwartz's inequality, we obtain
\begin{eqnarray*}
T&=&p'\int_{\mathbb R^{n}}\sum^{\infty}_{k=-\infty}\sigma_{k}^{p'-1}(x)
2^{k(p-1)}P^{\ast}_{\alpha}\mu_{k}(x)dx\\
&\lesssim&p'\int_{\mathbb R^{n}}\left(\sum^{\infty}_{k=-\infty}(\sigma_{k}(x))2^{k\frac{1}{(p-1)}\frac{1}{p'-1}}(P^{\ast}_{\alpha}\mu_{k})^{p'-1}\right)^{p'-1}\\
&&\times\left(\sum^{\infty}_{k=-\infty}2^{k(p-1)\frac{p(p-2)}{(p-1)^{2}}\frac{1}{2-p'}}(P^{\ast}_{\alpha}\mu_{k}(x))^{p'}\right)^{2-p'}dx\\
&=&p'\int_{\mathbb R^{n}}\left(\sum^{\infty}_{k=-\infty}\sigma_{k}(x)2^{k}(P^{\ast}_{\alpha}\mu_{k})^{p'-1}\right)^{p'-1}
\left(\sum^{\infty}_{k=-\infty}2^{kp}(P^{\ast}_{\alpha}\mu_{k}(x))^{p'}\right)^{2-p'}dx,
\end{eqnarray*}
which, together with H\"older's inequality, indicates that $T\lesssim p' T_{1}^{2-p'}T^{p'-1}_{2}$, where
\begin{equation}
\left\{ \begin{aligned}
&T_{1}:=\int_{\mathbb R^{n}}\left(\sum^{\infty}_{k=-\infty}2^{kp}(P^{\ast}_{\alpha}\mu_{k}(x))^{p'}\right)dx;\nonumber\\
&T_{2}:=\int_{\mathbb R^{n}}\left(\sum^{\infty}_{k=-\infty}\sigma_{k}(x)2^{k}(P^{\ast}_{\alpha}\mu_{k})^{p'-1}\right)dx.\nonumber
\end{aligned} \right.
\end{equation}

For $T_{1}$, we have
\begin{eqnarray*}
T_{1}&=&\sum^{\infty}_{k=-\infty}2^{kp}\int_{\mathbb R^{n}}\left(P^{\ast}_{\alpha}\mu_{k}(x)\right)^{p'}dx
=\sum^{\infty}_{k=-\infty}2^{kp}C^{\alpha,p}_{\mathbb{R}^{n+1}_+}(E_{k})\\
&\lesssim&\sum^{\infty}_{k=-\infty}\int^{2^{k}}_{2^{k-1}}C^{\alpha,p}_{\mathbb{R}^{n+1}_+}\left(\left\{(x,t)\in \overline{B_{r}(0,0)}
:\ P_{\alpha}f(x,t)\geq 2^{k}\right\}\right)d\lambda^{p}\\
&\lesssim&\sum^{\infty}_{k=-\infty}\int^{2^{k}}_{2^{k-1}}C^{\alpha,p}_{\mathbb{R}^{n+1}_+}\left(\left\{(x,t)\in \overline{B_{r}(0,0)}:\ P_{\alpha}f(x,t)\geq \lambda\right\}\right)d\lambda^{p}\\
&\lesssim&\int^{\infty}_{0}C^{\alpha,p}_{\mathbb{R}^{n+1}_+}\left(\left\{(x,t)\in \overline{B_{r}(0,0)}:\ P_{\alpha}f(x,t)\geq \lambda\right\}\right)d\lambda^{p}
\lesssim S.
\end{eqnarray*}

For $T_{2}$, we can get
\begin{eqnarray*}
T_{2}&=&\sum^{\infty}_{k=-\infty}2^{k}\int_{\mathbb R^{n}}\left(\sum^{\infty}_{j\geq k}2^{j(p-1)}P^{\ast}_{\alpha}\mu_{j}(x)\right)(P^{\ast}_{\alpha}\mu_{k})^{p'-1}dx\\
&=&\sum^{\infty}_{k=-\infty}\sum^{\infty}_{j\geq k}2^{k+j(p-1)}\int_{\mathbb R^{n}}P^{\ast}_{\alpha}\mu_{j}(x)(P^{\ast}_{\alpha}\mu_{k})^{p'-1}dx\\
&=&\sum^{\infty}_{k=-\infty}\sum^{\infty}_{j\geq k}2^{k+j(p-1)}\int_{\mathbb R^{n+1}_{+}}P_{\alpha}(P^{\ast}_{\alpha}\mu_{k})^{p'-1}d\mu_{j}(x).
\end{eqnarray*}
Note that $E_{k}$ is compact subsets. Let $f_{E_{k}}$ be the function satisfying (\ref{eq-2.1}). Suppose that $P_{\alpha}f_{E_{k}}(x_{0},t_{0})>1$, by the lower-semicontinuity, $P_{\alpha}f_{E_{j}}\geq 1+\delta>1$ on some neighborhood $U$ of $(x_{0},t_{0})$. On the other hand, denote by $\Omega_{E_{k}}$ the set $\{(x,t)\in E_{k}:\ P_{\alpha}f_{E_{k}}(x,t)< 1\}$ and let $F$ be any compact subset of $\Omega_{E_{k}}$. For any $0\leq f\in L^{p}(\mathbb R^{n})$ such that $P_{\alpha}f\geq 1$ on $F$
\begin{eqnarray*}
\mu_{k}(F)&\leq&\int_{\mathbb R^{n+1}_{+}}P_{\alpha}fd\mu_{k}\\
&=&\int_{\mathbb R^{n}}f P^{\ast}_{\alpha}\mu_{k}dx\\
&\leq&\|P^{\ast}_{\alpha}\mu_{k}\|_{L^{p'}(\mathbb{R}^n)}\|f\|_{L^{p}(\mathbb{R}^n)},
\end{eqnarray*}
which, together with (\ref{eq-2.1}), gives
$$\mu_{k}(F)\leq\|P^{\ast}_{\alpha}\mu_{k}\|_{L^{p'}(\mathbb{R}^n)}C^{\alpha,p}_{\mathbb R^{n+1}_{+}}(F)\leq \|P^{\ast}_{\alpha}\mu_{k}\|_{L^{p'}(\mathbb{R}^n)}C^{\alpha,p}_{\mathbb R^{n+1}_{+}}(\Omega_{E_{k}})=0.$$
Hence, $\mu_{k}(\Omega_{E_{k}})=0$, i.e., $P_{\alpha}f_{E_{k}}\geq 1$ on $E_{k}$ for $\mu_{k}$ a.e.  We can deduce that
\begin{eqnarray*}
C^{\alpha,p}_{\mathbb R^{n+1}_{+}}({E_{k}})&=&\int_{\mathbb R^{n+1}_{+}}P_{\alpha}(P^{\ast}_{\alpha}\mu_{k})^{p'-1}d\mu_{k}\\
&=&\int_{\mathbb R^{n+1}_{+}}P_{\alpha}f_{E_{k}}d\mu_{k}\\
&\geq&(1+\delta)\mu_{k}(U)+\mu_{k}(E_{k}\setminus U)\\
&\geq&\delta\mu_{k}(U)+\mu_{k}(E_{k}),
\end{eqnarray*}
which gives $\mu_{k}(U)=0$, i.e., $(x_{0},t_{0})\notin \text{ supp }\mu_{k}$. Equivalently, for all $(x,t)\in \text{ supp }\mu_{k}$, $$P_{\alpha}(P^{\ast}_{\alpha}\mu_{k})^{p'-1}=P_{\alpha}f_{E_{k}}\leq 1.$$
Since  $E_{j}\subset E_{k}$ for $j\geq k$,  we obtain
\begin{eqnarray*}
T_{2}&\lesssim&\sum^{\infty}_{k=-\infty}\sum^{\infty}_{j\geq k}2^{k+j(p-1)}C^{\alpha,p}_{\mathbb{R}^{n+1}_+}(E_{j})\\
&=&\sum^{\infty}_{j=-\infty}2^{j(p-1)}C^{\alpha,p}_{\mathbb{R}^{n+1}_+}(E_{j})\sum^{j}_{k=-\infty}2^{k}\\
&\lesssim&\sum^{\infty}_{j=-\infty}2^{jp}C^{\alpha,p}_{\mathbb{R}^{n+1}_+}(E_{j})\lesssim S.
\end{eqnarray*}
The estimates for $T_{1}$ and $T_{2}$ imply (\ref{eq-2.6}). It can be deduced from (\ref{eq-2.5})\ \&\ (\ref{eq-2.6}) that
$$\Big\|\sum\limits^{\infty}_{j=-\infty}2^{j(p-1)}(P^{\ast}_{\alpha}\mu_{j})\Big\|^{p'}_{L^{p'}(\mathbb{R}^n)}\lesssim \|f\|_{L^{p}(\mathbb{R}^n)}^{p},$$
i.e., $\sigma\in L^{p'}(\mathbb R^{n})$ and the series $\sum\limits^{\infty}_{j=-\infty}2^{j(p-1)}P^{\ast}_{\alpha}\mu_{j}(x)<\infty \text{ a.e. }x\in\mathbb R^{n}$.

{\it Case 2: $1<p<2$.} For $k=0,\pm1,\pm2,\ldots,$ let
$$\begin{cases}
\sigma_{k}(x)=\sum\limits^{k}_{j=-\infty}2^{j(p-1)}P^{\ast}_{\alpha}\mu_{j}(x);\\
\sigma(x)=\sum\limits^{\infty}_{j=-\infty}2^{j(p-1)}P^{\ast}_{\alpha}\mu_{j}(x).
\end{cases}$$
Similarly, because $\sigma_{k}(x)\geq 0$ for $x\in \mathbb R^{n}$, we can also write $\lim\limits_{k\rightarrow\infty}\sigma(x)=\sigma(x)$ formally.
Below, similar to {\it Case 1}, we will prove (\ref{eq-2.6}) and $\sigma\in L^{p'}(\mathbb R^{n})$. Following the procedure of (\ref{eq-2.7}), we can deduce that $$(\sigma(x))^{p'}=\left(\sum^{\infty}_{j=-\infty}2^{j(p-1)}P^{\ast}_{\alpha}\mu_{j}(x)\right)^{p'}\leq p'\sum^{\infty}_{k=-\infty}(\sigma_{k}(x))^{p'-1}
2^{k(p-1)}P^{\ast}_{\alpha}\mu_{k}(x).$$
We get
\begin{eqnarray*}
T
&=&p'\sum^{\infty}_{k=-\infty}2^{k(p-1)}\int_{\mathbb R^{n}}P^{\ast}_{\alpha}\mu_{k}(x)\left(\sum\limits^{k}_{j=-\infty}2^{j(p-1)}P^{\ast}_{\alpha}\mu_{j}(x)\right)^{p'-1}dx\\
&=&p'\sum^{\infty}_{k=-\infty}2^{k(p-1)}\int_{\mathbb R^{n}}\left(\sum\limits^{k}_{j=-\infty}2^{j(p-1)}\left(P^{\ast}_{\alpha}\mu_{k}(x)\right)^{1/(p'-1)}P^{\ast}_{\alpha}\mu_{j}(x)\right)^{p'-1}dx\\
&=&p'\sum^{\infty}_{k=-\infty}2^{k(p-1)}\left\|\sum\limits^{k}_{j=-\infty}2^{j(p-1)}\left(P^{\ast}_{\alpha}\mu_{k}(x)\right)^{1/(p'-1)}
P^{\ast}_{\alpha}\mu_{j}(x)\right\|_{L^{p'-1}(\mathbb{R}^n)}^{p'-1}\\
&\lesssim&p'\sum^{\infty}_{k=-\infty}2^{k(p-1)}\left\{\sum\limits^{k}_{j=-\infty}2^{j(p-1)}\left\|\left(P^{\ast}_{\alpha}\mu_{k}(x)\right)^{1/(p'-1)}
P^{\ast}_{\alpha}\mu_{j}(x)\right\|_{L^{p'-1}(\mathbb{R}^n)}\right\}^{p'-1}.
\end{eqnarray*}
Also, for $j\leq k$, if $x\in E_{k}$ then $x\in E_{j}$. Similarly, we can deduce that
$$P_{\alpha}\left(P^{\ast}_{\alpha}\mu_{j}(x)\right)^{p'-1}(x)\leq1,\ x\in E_{k}.$$
This indicates that
\begin{eqnarray*}
\int_{\mathbb R^{n}}\left(P^{\ast}_{\alpha}\mu_{k}(x)\right)\left(P^{\ast}_{\alpha}\mu_{j}(x)\right)^{p'-1}dx
&=& \int_{\mathbb R^{n+1}_{+}}P_{\alpha}\left(P^{\ast}_{\alpha}\mu_{j}(x)\right)^{p'-1}d\mu_{k}\\
&\leq&\mu_{k}(E_{k})=C^{\alpha,p}_{\mathbb{R}^{n+1}_+}(E_{k}).
\end{eqnarray*}
We obtain
\begin{eqnarray*}
T&\lesssim& \sum^{\infty}_{k=-\infty}2^{k(p-1)}\left\{\sum\limits^{k}_{j=-\infty}2^{j(p-1)}\left(\int_{\mathbb R^{n}}P^{\ast}_{\alpha}\mu_{k}(x)
\left(P^{\ast}_{\alpha}\mu_{j}(x)\right)^{p'-1}dx\right)^{1/(p'-1)}\right\}^{p'-1}\\
&\lesssim& \sum^{\infty}_{k=-\infty}2^{k(p-1)}C^{\alpha,p}_{\mathbb{R}^{n+1}_+}(E_{k})\left(\sum^{k}_{j=-\infty}2^{j(p-1)}\right)^{p'-1}\\
&\lesssim&\sum^{\infty}_{k=-\infty}2^{kp}C^{\alpha,p}_{\mathbb{R}^{n+1}_+}(E_{k})\lesssim S,
\end{eqnarray*}
which gives (\ref{eq-2.6}) for $1<p<2$.

It follows from (\ref{eq-2.5})\ \&\ (\ref{eq-2.6}) that
$$S\lesssim \|f\|_{L^{p}(\mathbb{R}^n)}T^{1/p'}\lesssim \|f\|_{L^{p}(\mathbb{R}^n)}S^{1/p'}.$$
We can get $S\lesssim \|f\|_{L^{p}(\mathbb{R}^n)}^{p}$ and consequently, $T<\infty$. Then  $\sigma\in L^{p'}(\mathbb R^{n}),$ for $1<p<2$.
By the fact that $f\in C^{\infty}_{0}(\mathbb R^{n})$, there exists an integer $j_{0}$ such that
\begin{eqnarray*}
\Big\|\sum\limits^{\infty}_{j=0}2^{j(p-1)}P^{\ast}_{\alpha}\mu_{j}\Big\|_{L^{p'}(\mathbb{R}^n)}
&\leq&\sum\limits^{\infty}_{j=0}2^{j(p-1)}\left\|P^{\ast}_{\alpha}\mu_{j}\right\|_{L^{p'}(\mathbb{R}^n)}\\
&\simeq&\sum\limits^{j_{0}}_{j=0}2^{j(p-1)}(C^{\alpha,p}_{\mathbb{R}^{n+1}_+}(E_{j}))^{1/p'}<\infty.
\end{eqnarray*}
This indicates that $\sum\limits^{-1}_{j=-\infty}2^{j(p-1)}P^{\ast}_{\alpha}\mu_{j}\in L^{p'}(\mathbb R^{n})$. For $\sigma_{k}$, if $k\leq -1$, because $P^{\ast}_{\alpha}\mu_{j}(x)\geq 0$, then
$$\|\sigma_{k}\|_{L^{p'}(\mathbb{R}^n)}\leq \Big\|\sum\limits^{-1}_{j=-\infty}2^{j(p-1)}P^{\ast}_{\alpha}\mu_{j}\Big\|_{L^{p'}(\mathbb{R}^n)}$$
and $\sigma_{k}\in L^{p'}(\mathbb R^{n})$. If $k\geq 0$, noticing that
$$\sigma_{k}=\sum\limits^{-1}_{j=-\infty}2^{j(p-1)}P^{\ast}_{\alpha}\mu_{j}+\sum\limits^{j_{0}}_{j=0}2^{j(p-1)}P^{\ast}_{\alpha}\mu_{j},$$
we can also get $\sigma_{k}\in L^{p}(\mathbb R^{n})$.

Further, we obtain that
\begin{eqnarray*}
&&\int^{\infty}_{0}C^{\alpha,p}_{\mathbb{R}^{n+1}_+}\left(\left\{(x,t)\in \overline{B_{r}(0,0)}:\ P_{\alpha}f(x,t)\geq \lambda\right\}\right)d\lambda^{p}\\
&&\quad =\sum^{\infty}_{j=-\infty}\int^{2^{j+1}}_{2^{j}}C^{\alpha,p}_{\mathbb{R}^{n+1}_+}\left(\left\{(x,t)\in \overline{B_{r}(0,0)}:\ P_{\alpha}f(x,t)\geq \lambda\right\}\right)d\lambda^{p}\\
&&\quad\lesssim \sum^{\infty}_{j=-\infty}C^{\alpha,p}_{\mathbb{R}^{n+1}_+}(E_{j})\int^{2^{j+1}}_{2^{j}}d\lambda^{p}\\
&&\quad\lesssim \sum^{\infty}_{j=-\infty}2^{pj}\mu_{j}(E_{j})\lesssim \|f\|^{p}_{L^{p}(\mathbb{R}^n)}.
\end{eqnarray*}
Hence we can deduce from Proposition \ref{prop-2.2} and Fauto's lemma that
\begin{eqnarray*}
&&\int^{\infty}_{0}C^{\alpha,p}_{\mathbb{R}^{n+1}_+}\left(\left\{(x,t)\in \mathbb R^{n+1}_{+}:\ P_{\alpha}f(x,t)>\lambda\right\}\right)d\lambda^{p}\\
&&\quad=\int^{\infty}_{0}C^{\alpha,p}_{\mathbb{R}^{n+1}_+}\left(\cup^{\infty}_{r=1}\left\{(x,t)\in \overline{B_{r}(0,0)}:\ P_{\alpha}f(x,t)>\lambda\right\}\right)d\lambda^{p}\\
&&\quad=\int^{\infty}_{0}\lim_{r\rightarrow\infty}C^{\alpha,p}_{\mathbb{R}^{n+1}_+}\left(\left\{(x,t)\in \overline{  B_{r}(0,0)}:\ P_{\alpha}f(x,t)>\lambda\right\}\right)d\lambda^{p}\\
&&\quad\leq \underline{\lim}_{r\rightarrow\infty}\int^{\infty}_{0}C^{\alpha,p}_{\mathbb{R}^{n+1}_+}\left(\left\{(x,t)\in  \overline{ B_{r}(0,0)}:\ P_{\alpha}f(x,t)>\lambda\right\}\right)d\lambda^{p}\\
&&\quad\leq \underline{\lim}_{r\rightarrow\infty}\int^{\infty}_{0}C^{\alpha,p}_{\mathbb{R}^{n+1}_+}\left(\left\{(x,t)\in   \overline{B_{r}(0,0)}
:\ P_{\alpha}f(x,t)\geq \lambda\right\}\right)d\lambda^{p}\\
&&\quad\lesssim\|f\|^{p}_{L^{p}(\mathbb{R}^n)},
\end{eqnarray*}
which proves (\ref{eq-2.3}) for the functions in $C^{\infty}_{0}(\mathbb R^{n})$.

At last, we prove that (\ref{eq-2.3}) holds for the functions in $L^{p}_{+}(\mathbb R^{n})$. Let $f\in L^{p}_{+}(\mathbb R^{n})$. Since $C^{\infty}_{0}(\mathbb R^{n})$ is dense in $L^{p}_{+}(\mathbb R^{n})$, there exists a sequence of $\{f_{j}\}_{j=1}^{\infty}$ such that $f_{n}\rightarrow f$ in $L^{p}(\mathbb R^{n})$ as $n\rightarrow \infty$. Then it is obvious that $\|f_{j}\|_{L^{p}(\mathbb{R}^n)}\rightarrow\|f\|_{L^{p}(\mathbb{R}^n)}$ as $n\rightarrow \infty$. We can select a subsequence denoted by $\{f_{i_{j}}\}$ such that $\|f_{i_{j}}-f\|_{L^{p}(\mathbb{R}^n)}<4^{-j}$. For any $\lambda>0$, there exists an $m_{1}\in\mathbb N$ such that for $j\geq m_{1}$,
\begin{eqnarray*}
C^{\alpha,p}_{\mathbb R^{n+1}_{+}}\Big(\Big\{(x,t):\ P_{\alpha}(|f_{i_{j}}-f|)(x,t)>\lambda/2\Big\}\Big)
&&\leq C^{\alpha,p}_{\mathbb R^{n+1}_{+}}\Big(\Big\{(x,t):\ P_{\alpha}(|f_{i_{j}}-f|)(x,t)>2^{-j}\Big\}\Big)\\
&&\lesssim 2^{jp}\|f_{i_{j}}-f\|^{p}_{L^{p}(\mathbb{R}^n)}\\
&&\lesssim 2^{-jp}.
\end{eqnarray*}
On the other hand, there exists an $m_{2}\in\mathbb N$ such that
$$2^{-jp}\leq\frac{1}{2} C^{\alpha,p}_{\mathbb R^{n+1}_{+}}\Big(\Big\{(x,t):\ P_{\alpha}(|f_{i_{j}}-f|)(x,t)>\lambda\Big\}\Big).$$
Take $m:=\max\{m_{1}, m_{2}\}$. Then
\begin{eqnarray*}
C^{\alpha,p}_{\mathbb R^{n+1}_{+}}\Big(\Big\{(x,t):\ P_{\alpha}(f)(x,t)>\lambda\Big\}\Big)&\leq&C^{\alpha,p}_{\mathbb R^{n+1}_{+}}\Big(\Big\{(x,t):\ P_{\alpha}(f_{i_j})(x,t)>\lambda/2\Big\}\Big)\\
&&+C^{\alpha,p}_{\mathbb R^{n+1}_{+}}\Big(\Big\{(x,t):\ P_{\alpha}(|f_{i_{j}}-f|)(x,t)>\lambda/2\Big\}\Big)\\
&\leq&C^{\alpha,p}_{\mathbb R^{n+1}_{+}}\Big(\Big\{(x,t):\ P_{\alpha}(f_{i_j})(x,t)>\lambda/2\Big\}\Big)\\
&&+\frac{1}{2} C^{\alpha,p}_{\mathbb R^{n+1}_{+}}\Big(\Big\{(x,t):\ P_{\alpha}(|f_{i_{j}}-f|)(x,t)>\lambda\Big\}\Big),
\end{eqnarray*}
which means that $\forall j\geq m$,
$$C^{\alpha,p}_{\mathbb R^{n+1}_{+}}\Big(\Big\{(x,t):\ |P_{\alpha}(f)|(x,t)>\lambda\Big\}\Big)\leq 2C^{\alpha,p}_{\mathbb R^{n+1}_{+}}\Big(\Big\{(x,t):\ |P_{\alpha}(f_{i_j})|(x,t)>\lambda/2\Big\}\Big).$$
Then
$$C^{\alpha,p}_{\mathbb R^{n+1}_{+}}\Big(\Big\{(x,t):\ |P_{\alpha}(f)|(x,t)>\lambda\Big\}\Big)\leq 2\liminf_{j\rightarrow\infty}C^{\alpha,p}_{\mathbb R^{n+1}_{+}}\Big(\Big\{(x,t):\ |P_{\alpha}(f_{i_j})|(x,t)>\lambda/2\Big\}\Big).$$
By Fauto's lemma, we can get
\begin{eqnarray*}
&&\int^{\infty}_{0}C^{\alpha,p}_{\mathbb{R}^{n+1}_+}\left(\left\{(x,t)\in\mathbb R^{n+1}_{+}:\ P_{\alpha}f(x,t)\geq\lambda\right\}\right)d\lambda^{p}\\
&&\lesssim \int^{\infty}_{0}\liminf_{j\rightarrow\infty}C^{\alpha,p}_{\mathbb R^{n+1}_{+}}\Big(\Big\{(x,t):\ |P_{\alpha}(f_{j})|(x,t)>\lambda/2\Big\}\Big)d\lambda^{p}\\
&&\lesssim \liminf_{j\rightarrow\infty}\int^{\infty}_{0}C^{\alpha,p}_{\mathbb R^{n+1}_{+}}\Big(\Big\{(x,t):\ P_{\alpha}(f_{i_j})(x,t)>\lambda/2\Big\}\Big)d\lambda^{p}\\
&&\lesssim  \liminf_{j\rightarrow\infty}\|f_{i_{j}}\|^{p}_{L^{p}(\mathbb{R}^n)}\\
&&\lesssim \|f\|_{L^{p}(\mathbb{R}^n)}^{p}.
\end{eqnarray*}

\end{proof}

\subsection{Preliminary  lemmas on the fractional capacity}\label{sec-2.3}
Let $\mathcal M_{+}(\mathbb R^{n+1}_{+})$ represent the class of all nonnegative Radon measures on $\mathbb R^{n+1}_{+}$.
\begin{lemma}\label{lemma 2}
Let $\alpha \in(0,2) $ and $\beta\in (0,n).$ Given $f\in \dot{W}^{\beta,p}(\mathbb{R}^n), s>0,$ and $\mu\in\mathcal M_{+}(\mathbb R^{n+1}_{+}),$ let
$$L^{\alpha,\beta}_{s}(f)=\left\{(x,t)\in \mathbb{R}^{n+1}_+:\ |P_{\alpha} f(x,t)|>s\right\}$$
and
$$R^{\alpha,\beta}_{s}(f)=\Big\{y\in \mathbb{R}^{n}:\ \sup_{|y-x|<t}|P_{\alpha} f(x,t)|>s\Big\}.$$
Then the following four statements are true.
\item{\rm(i)}  For any natural number $k$
    $$\mu\left(L^{\alpha,\beta}_{s}(f)\cap T(B(0,k))\right)\leq \mu\left(T\left(R^{\alpha,\beta}_{s}(f)\cap B(0,k)\right)\right).$$
\item{\rm (ii)}  For any natural number $k,$
    $$Cap_{\mathbb{R}^n}^{\beta,p}\left(R^{\alpha,\beta}_{s}(f)\cap B(0,k)\right)\geq c^\beta_p(\mu, \mu\left(T\left(R^{\alpha,\beta}_{s}(f)\cap B(0,k)\right)\right).$$
\item{\rm (iii)} There exists a  constant $\theta_\alpha>0$ such that
    $$\sup_{|y-x|<t}|P_{\alpha} f (y,t)|\leq \theta_\alpha \mathcal{M} f(x),\quad  x\in \mathbb{R}^n,$$
   where $\mathcal{M}$ denotes  the Hardy-Littlewood maximal operator:
$$\mathcal{M}f(x)=\sup_{r>0}r^{-n}\int_{B(x,r)}|f(y)|dy,\  x\in \mathbb{R}^n.$$

\item{\rm (iv)} There exists a  constant $\eta_{n,\alpha}>0$ such that
$$(x,t)\in T(O)\Rightarrow (p^{\alpha}_t\ast
|f|)(x,t) \geq
\eta_{n,\alpha},$$
where $O$ is a bounded open set contained in $\hbox{Int}(\{x\in \mathbb{R}^n: f(x)\geq 1\}).$

\end{lemma}
\begin{proof}
    (i) Since $\sup\limits_{|y-x|<t}|P_{\alpha} f(x,t)|$ is lower semicontinuous on $\mathbb{R}^n$, we can see that $R^{\alpha,\beta}_s(f)$ is an open subset of $\mathbb{R}^n$ and
    $$\left\{\begin{aligned}
    &L^{\alpha,\beta}_s(f)\subseteq T(R^{\alpha,\beta}_s(f));\\
     &\mu(L^{\alpha,\beta}_s(f))\leq \mu(T(R^{\alpha,\beta}_s(f))).
    \end{aligned}\right.$$
    Then
    \begin{eqnarray*}
    \mu(L^{\alpha,\beta}_s(f)\cap T(B(0,k)))&\leq& \mu(T(R^{\alpha,\beta}_s(f)\cap T(B(0,k))))=\mu(T(R^{\alpha,\beta}_s(f)\cap B(0,k))).
    \end{eqnarray*}

    (ii) It follows from the definition of $c^\beta_p(\mu;t).$

    (iii) Since
    $p^{\alpha}_{t}(x)=\frac{1}{t^{n}}\psi(\frac{x}{t})$, where $\psi(x):=(1+|x|^{2})^{-(n+\alpha)/2}$ is radial bounded and integrable on $\mathbb{R}^n,$ it follows from \cite[p.57, Proposition]{Stein} that   $$\sup_{|y-x|<t}|P_{\alpha}f(y,t)|\leq \theta_\alpha\mathcal{M}f(x).$$

    (iv) For any $(x,t)\in T(O),$ we have
    $$B(x,t)\subseteq O\subseteq\hbox{Int }(\{x: f(x)>1\}).$$
    There exist $\sigma$ and $C$ which   depend only on $n$ and $\alpha$ such that  $
    \inf\{p^\alpha_t(x): |x|<\sigma t\}\geq Ct^{-n}.$
    Then $$p^{\alpha}_t\ast |f|(x,t)\geq Ct^{-n}\int_{B(x,\sigma t)\cap \text{ Int }(\{x: f(x)\geq 1\})} |f|(y)dy.
    $$
If $\sigma>1,$ then
$$B(x,\sigma t)\cap \text{ Int }(\{x: f(x)\geq 1\})\supseteq B(x, t)\cap \text{ Int }(\{x: f(x)\geq 1\})=B(x,t).$$
If $\sigma\leq 1,$ then
$$B(x,\sigma t)\cap \text{ Int }(\{x: f(x)\geq 1\})=B(x,\sigma t).$$
Thus $p^\alpha_t\ast|f|(x,t)\geq \eta_{n,\alpha}$ for a constant $\eta_{n,\alpha}>0.$

    \end{proof}

The following result provides  the capacitary strong estimates for $Cap_{\mathbb{R}^n}^{\beta,p}(\cdot)$. For the proof, we refer the reader to
\cite[Lemma 2.1]{Zhai} and the references therein.
\begin{lemma}\label{lemma 4}
Let $\beta\in (0,n)$ and $p\in [1, {n}/{\beta}]$.
    \item{\rm (i)} For $f\in C_0^\infty(\mathbb{R}^n)$,
    $$\int_0^\infty Cap_{\mathbb{R}^n}^{\beta,p}\left(\left\{x\in \mathbb{R}^n: |f(x)|\geq s  \right\}\right)ds^p\lesssim \|f\|^p_{\dot{W}^{\beta,p}(\mathbb{R}^n)}.$$
    \item{\rm (ii)} For $f\in C_0^\infty(\mathbb{R}^n)$,
$$\int_0^\infty Cap_{\mathbb{R}^n}^{\beta,p}\left(\left\{x\in \mathbb{R}^n: |\mathcal{M}f(x)|\geq s  \right\}\right)ds^p\lesssim \|f\|^p_{\dot{W}^{\beta,p}(\mathbb{R}^n)}.$$

\end{lemma}

        For handling the endpoint case $p=n/\beta$ when $p>q,$ we need the following Riesz potentials  on $\mathbb{R}^{2n}$, see Adam-Xiao \cite{Adams Xiao} and Adam \cite{Adams} . For $\gamma\in (0,2n),$
    $$I^{(2n)}_\gamma\ast f(z)=\int_{\mathbb{R}^{2n}}|z-y|^{\gamma-2n}f(y)dy, \quad z\in \mathbb{R}^{2n}.$$
    For $\gamma\in (0,2n),$ $\dot{\mathcal{L}}^p_{\gamma}(\mathbb{R}^{2n})$ is defined as the completion of $I^{2n}_\gamma\ast C^\infty_0(\mathbb{R}^{2n})$ with  $\|I^{2n}_\gamma\ast f\|_{\dot{\mathcal{L}}^p_{\gamma}(\mathbb{R}^{2n})}=\|f\|_{L^{p}(\mathbb{R}^{2n})}.$

The following result is a special case of \cite[Theorem 5.2]{Adams} or \cite[Theorem A]{Adams Xiao}.

\begin{lemma}\label{lemma 3}
Let $\beta\in (0,n).$ Then there are a linear extension operator
$$\mathcal{E}:\ \dot{W}^{\beta,n/\beta}(\mathbb{R}^n)\longrightarrow\dot{\mathcal{L}}^{n/\beta}_{2\beta}(\mathbb{R}^{2n}),$$
and a linear restriction operator
$$\mathcal{R}:\ \dot{\mathcal{L}}^{n/\beta}_{2\beta}(\mathbb{R}^{2n})\longrightarrow\dot{W}^{\beta,n/\beta}(\mathbb{R}^n)$$
such that $\mathcal{R}\mathcal{E}$ is the identity. Moreover,
    \item{\rm (i)} For $f\in \dot{W}^{\beta,n/\beta}(\mathbb{R}^n)$,
$    \|\mathcal{E} f\|_{\dot{\mathcal{L}}^{n\beta}_{2\beta}(\mathbb{R}^{2n})}\lesssim\|f\|_{\dot{W}^{\beta,n/\beta}(\mathbb{R}^n)}$.
\item{\rm (ii)} For $g\in \dot{\mathcal{L}}^{n/\beta}_{2\beta}(\mathbb{R}^{2n})$,
$    \|\mathcal{R} g\|_{\dot{W}^{\beta,n/\beta}(\mathbb{R}^n)}\lesssim\|g\|_{\dot{\mathcal{L}}^{n/\beta}_{2\beta}(\mathbb{R}^{2n})}.$

\end{lemma}

\begin{lemma}
\label{lemma 1}
If $\alpha\in(0,2), \beta\in(0,n)$ and $(x,t)\in \mathbb{R}^{n+1}_+,$ then
$$\int_{\mathbb{R}^n}p^{\alpha}_t(y) |y-x|^{\beta-n}dy\lesssim (t^2+|x|^2)^{{(\beta-n)}/{2}}.$$
\end{lemma}
\begin{proof}
Define $$J(x,t)=\int_{\mathbb{R}^n}p^{\alpha}_t(y) |y-x|^{\beta-n}dy
=c(n,\alpha)\int_{\mathbb{R}^n}\frac{t^\alpha|y-x|^{\beta-n}}{(|y|^2+t^2)^{{(n+\alpha)}/{2}}} dy.
$$
Via the change of variables: $x\longrightarrow tx\ \&\ y\longrightarrow ty,$ it is sufficient to show that
$$J(x,1)\lesssim (1+|x|^2)^{{(\beta-n)}/{2}}.$$
Since $J(0,1)\lesssim 1,$ we may assume that $|x|>0.$ Write $J(x,1)\lesssim I_1(x)+I_2(x)$, where
$$\left\{\begin{aligned}
&I_{1}(x):= \int_{B(x,|x|/2)}\frac{|y-x|^{\beta-n}}{(|y|^2+1)^{{(n+\alpha)}/{2}}}dy;\\
&I_{2}(x):=\int_{\mathbb{R}^n\backslash B(x,|x|/2)}\frac{|y-x|^{\beta-n}}{(|y|^2+1)^{{(n+\alpha)}/{2}}}dy.
\end{aligned}\right.$$

Since $|x-y|\leq |x|/2$ implies that $|y|\approx |x|,$ we have
\begin{eqnarray*}
I_1(x)
&\lesssim& (1+|x|^2)^{-{(n+\alpha)}/{2}}\int_{B(x,|x|/2)}|y-x|^{\beta-n}dy\\
%&\lesssim& (1+|x|^2)^{-{(n+\alpha)}/{2}}\int_0^{|x|/2}s^{\beta-1}ds\\
%&\lesssim& (1+|x|^2)^{-{(n+\alpha)}/{2}}|x|^\beta\\
&\lesssim& (1+|x|^2)^{{(\beta-n)}/{2}}.
\end{eqnarray*}
If $|x-y|>|x|/2,$ then $|y-x|^{\beta-n}\lesssim |x|^{\beta-n},$ thus
\begin{eqnarray*}
I_2(x)&\lesssim&|x|^{\beta-n}\int_{\mathbb{R}^n\backslash B(x,|x|/2)}\frac{1}{(|y|^2+1)^{{(n+\alpha)}/{2}}}dy\  \lesssim\ |x|^{\beta-n}.
\end{eqnarray*}
For $|x-y|>|x|/2$, it holds $|y|<3|x-y|$ and
\begin{eqnarray*}
I_2(x)=\int_{\mathbb{R}^n\backslash B(x,|x|/2)}\frac{1}{(|y|^2+1)^{{(n+\alpha)}/{2}}|y|^{n-\beta}}dy\lesssim1.
\end{eqnarray*}
So, $I_2(x)\lesssim (1+|x|^2)^{{(\beta-n)}/{2}}$ and $J(x,1)\lesssim (1+|x|^2)^{{(\beta-n)}/{2}}.$
\end{proof}

\section{Embeddings  of $L^p(\mathbb{R}^{n})$ to $L^q(\mathbb{R}^{n+1}_+,\mu)$}\label{sec-3}
In this section,  we focus on the embedding \eqref{5}:
$$\|P_{\alpha} f(\cdot,\cdot)\|_{L^{q}(\mathbb{R}^{n+1}_+,\mu)}\lesssim \|f\|_{L^p(\mathbb R^n)}
.$$

\subsection{Embeddings  of $L^p(\mathbb{R}^{n})$ to $L^q(\mathbb{R}^{n+1}_+,\mu)$   when $p\leq q$}
Given $1<p\leq q<\infty$.  For $\lambda>0$, define
$$c_{\alpha,p}(\mu; \lambda):=\inf\left\{C^{\alpha,p}_{\mathbb{R}^{n+1}_+}(K):\ \text{ compact }K\subset \mathbb R^{n+1}_{+},\quad  \mu(K)\geq \lambda\right\}.$$
\begin{theorem}\label{thm-3.1}
Let $1<p\leq q<\infty$ and $\mu\in \mathcal{M}_{+}(\mathbb R^{n+1}_{+})$.
\item{\rm (i)} The embedding \eqref{5} holds if and only if $$\sup_{\lambda\in(0,\infty)}\lambda^{p/q}/c_{\alpha,p}(\mu;\lambda)<\infty.$$
\item{\rm (ii)} If $1<p<q<\infty$, then $\sup_{\lambda\in(0,\infty)}\lambda^{p/q}/c_{\alpha,p}(\mu;\lambda)<\infty$ can be replaced by
$$\sup_{ (r, x_{0}, t_{0})\in(0,\infty)\times\mathbb R^{n+1}_+,  t_{0}\lesssim r}\frac{\mu(B_{r}(x_{0}, t_{0})) }{r^{qn/p}}<\infty.$$
\end{theorem}

\begin{proof}
(i) Suppose that \eqref{5} is true. Let $K$ be a compact set in $\mathbb R^{n+1}_{+}$. Denote by $\mu\mid_{K}$ the restriction of $\mu$ to the set $K$.
\begin{eqnarray*}
\int_{\mathbb R^{n}}f(x)P^{\ast}_{\alpha}\mu\mid_{K}(x)dx &=&\int_{\mathbb R^{n+1}_{+}}P_{\alpha}f(x,t)d\mu\mid_{K}(x,t)\\
&\lesssim& \|P_{\alpha}f\|_{L^{q}(\mathbb R^{n+1}_{+},\mu)}(\mu(K))^{1/q'}
\lesssim \|f\|_{L^{p}(\mathbb R^{n})}(\mu(K))^{1/q'},
\end{eqnarray*}
which means that
$\|P^{\ast}_{\alpha}\mu\mid_{K}\|_{L^{p'}(\mathbb R^{n})}\lesssim (\mu(K))^{1/q'}$. Define
$$E_{\lambda}(f)=\left\{(x,t)\in\mathbb R^{n+1}_{+}:\ |P_{\alpha}f(x,t)|\geq\lambda\right\}.$$
We can get
\begin{eqnarray*}
\lambda\mu(E_{\lambda}(f))&\leq&\int_{E_{\lambda}(f)}|P_{\alpha}f(x,t)|d\mu\mid_{E_{\lambda}(f)}(x,t)\\
&\leq&\|f\|_{L^{p}(\mathbb R^{n})}\|P^{\ast}_{\alpha}\mu\mid_{E_{\lambda}(f)}\|_{L^{p'}}\\
&\leq&\|f\|_{L^{p}(\mathbb R^{n})}(\mu(E_{\lambda}(f)))^{1/q'}.
\end{eqnarray*}
This implies $$\sup_{\lambda\in(0,\infty)}\lambda^{q}(\mu(E_{\lambda}(f)))\lesssim \|f\|_{L^{p}(\mathbb R^{n})}^{q}.$$

Choose a function $f\in L^{p}(\mathbb R^{n})$ such that $P_{\alpha}f\geq1$ on a given compact set $K\subset\mathbb R^{n+1}_{+}$, i.e., $K\subset E_{1}(f)$. We have
\begin{equation}\label{eq-4.4}
(\mu(K))^{1/q}\leq (\mu(E_{1}(f)))^{1/q}\lesssim \|f\|_{L^{p}(\mathbb R^{n})}.
\end{equation}
Take the infimum on both sides of (\ref{eq-4.4}), we obtain $(\mu(K))^{1/q}\lesssim (C^{\alpha,p}_{\mathbb{R}^{n+1}_+}(K))^{1/p}$. If $K$ is compact and $\mu(K)\geq \lambda$, then $\lambda^{p/q}\lesssim C^{\alpha,p}_{\mathbb{R}^{n+1}_+}(K)$ and hence, $\lambda^{p/q}\lesssim c(\mu;\lambda)$ due to
$$\lambda^{p/q}\lesssim\inf\left\{C^{\alpha,p}_{\mathbb{R}^{n+1}_+}(K),\ K \text{ is compact and }\mu(K)\geq\lambda \right\}.$$

Conversely, if
$\sup\limits_{\lambda\in(0,\infty)}\lambda^{p/q}/c_{\alpha,p}(\mu;\lambda)<\infty,$
then for any $\lambda>0$, $\lambda^{p/q}\lesssim C^{\alpha,p}_{\mathbb{R}^{n+1}_+}(K)$.
%We have proved that
%$$\int^{\infty}_{0}C^{\alpha,p}_{\mathbb{R}^{n+1}_+}\left(\left\{(x,t)\in\mathbb R^{n+1}_{+}:\ P_{\alpha}f(x,t)\geq\lambda\right\}\right)d\lambda^{p}\leq \|f\|_{L^{p}}^{p},\quad \forall\ f\in L^{p}_{+}(\mathbb R^{n}).$$
For any $\tau$, according to Lemma \ref{le-4.2}, one has
\begin{eqnarray*}
\tau^{p}C^{\alpha,p}_{\mathbb{R}^{n+1}_+}(E_{\tau})&=&\int^{\tau}_{0}C^{\alpha,p}_{\mathbb{R}^{n+1}_+}\left(\left\{(x,t)\in\mathbb R^{n+1}_{+}:\ P_{\alpha}f(x,t)\geq\tau\right\}\right)d\lambda^{p}\\
&\lesssim& \int^{\tau}_{0}C^{\alpha,p}_{\mathbb{R}^{n+1}_+}\left(\left\{(x,t)\in\mathbb R^{n+1}_{+}:\ P_{\alpha}f(x,t)\geq\lambda\right\}\right)d\lambda^{p}\\
&\lesssim& \int^{\infty}_{0}C^{\alpha,p}_{\mathbb{R}^{n+1}_+}\left(\left\{(x,t)\in\mathbb R^{n+1}_{+}:\ P_{\alpha}f(x,t)\geq\lambda\right\}\right)d\lambda^{p}\lesssim \|f\|_{L^{p}(\mathbb R^{n})}^{p}.
\end{eqnarray*}
By the layer cake representation, this implies
\begin{eqnarray*}
\int_{\mathbb R^{n+1}_{+}}|P_{\alpha}f(x,t)|^{q}d\mu(x,t)&=&\int^{\infty}_{0}\mu(E_{\lambda})d\lambda^{q}\\
&\lesssim&\int^{\infty}_{0}\left(C^{\alpha,p}_{\mathbb{R}^{n+1}_+}(E_{\lambda})\right)^{q/p}d\lambda^{q}\\
&\lesssim&\int^{\infty}_{0}\left(\lambda^{-p}\|f\|_{L^{p}(\mathbb R^{n})}^{p}\right)^{q/p-1}\left(C^{\alpha,p}_{\mathbb{R}^{n+1}_+}(E_{\lambda})\right)\lambda^{q-1}d\lambda\\
&\lesssim&\|f\|^{q-p}_{L^{p}(\mathbb R^{n})}\int^{\infty}_{0}C^{\alpha,p}_{\mathbb{R}^{n+1}_+}(E_{\lambda})d\lambda^{p}\\
&\lesssim&\|f\|^{q}_{L^{p}(\mathbb R^{n})}.
\end{eqnarray*}

(ii) If $\lambda^{p/q}\lesssim c(\mu;\lambda)$, then $\mu^{1/q}(K)\lesssim (C^{\alpha,p}_{\mathbb{R}^{n+1}_+}(K))^{1/p}$ for any compact set $K\subset \mathbb R^{n+1}_{+}$. Let $K=B_{r}(x_{0}, t_{0})$. By Theorem \ref{th-4.1},  for $t_{0}\lesssim r$,
$$(\mu(B_{r}(x_{0}, t_{0})))^{1/q}\lesssim (C^{\alpha,p}_{\mathbb{R}^{n+1}_+}(B_{r}(x_{0}, t_{0})))^{1/p}\lesssim r^{n/p}.$$

For the reverse, take a compact set $K$ such that $\mu(K) \geq \lambda$. If $(x,t)\in B_{r}(x_{0}, t_{0})$, $|x-x_{0}|<r/2$ and $r+t_{0}\leq t\leq t_{0}+2r$.  If $|x-x_{0}|<r<t$, we can get $|x-x_{0}|/t<1$ and
\begin{eqnarray*}
p^{\alpha}_{t}(x-x_{0})&=&\frac{t^{\alpha}}{(t^{2}+|x-x_{0}|^{2})^{(n+\alpha)/2}}\\
&=&\frac{1}{t^{n}}\frac{1}{(1+|x-x_{0}|^{2}/t^{2})^{(n+\alpha)/2}}\\
&\geq&\frac{1}{2^{(n+\alpha)/2}}\frac{1}{t^{n}}\\
&\gtrsim&\frac{1}{3^{n}r^{n}}\frac{1}{2^{(n+\alpha)/2}},
\end{eqnarray*}
equivalently, there exists a constant $c$ such that $r\geq c(p_{t}^{\alpha}(x-x_{0}))^{-1/n}$. Hence if $(x,t)\in B_{r}(t_{0}, x_{0})$, then
$r\in (c(p_{t}^{\alpha}(x-x_{0}))^{-1/n}, \infty)$. For this case, the characteristic functions
$$1_{B_{r}(x_{0},t_{0})}(x,t)=1_{(c(p_{t}^{\alpha}(x-x_{0}))^{-1/n}, \infty)}(r)=1.$$
By Fubini's theorem, we can get
\begin{eqnarray*}
\int^{\infty}_{0}\mu\mid_{K}(B_{r}(x_{0},t_{0}))\frac{dr}{r^{n+1}}
&=&\int^{\infty}_{0}\left\{\int_{\mathbb R^{n+1}_{+}}1_{B_{r}(x_{0},t_{0})}(x,t)d\mu\mid_{K}(x,t)\right\}\frac{dr}{r^{n+1}}\\
&=&\int_{\mathbb R^{n+1}_{+}}\left\{\int^{\infty}_{0}1_{(c(p_{t}^{\alpha}(x-x_{0}))^{-1/n}, \infty)}(r)\frac{dr}{r^{n+1}}\right\}d\mu\mid_{K}(x,t)\\
&=&\int_{\mathbb R^{n+1}_{+}}\left\{\int^{\infty}_{(c(p_{t}^{\alpha}(x-x_{0}))^{-1/n}}\frac{dr}{r^{n+1}}\right\}d\mu\mid_{K}(x,t)\\
&\gtrsim&\int_{\mathbb R^{n+1}_{+}}p^{\alpha}_{t}(x_{0}-x)d\mu\mid_{K}(x,t)\\
&=&P^{\ast}_{\alpha}\mu\mid_{K}(x_{0}).
\end{eqnarray*}
By Minkowski's inequality, we have
$$\|P^{\ast}_{\alpha}\mu\mid_{K}\|_{L^{p'}(\mathbb R^{n})}\lesssim \int^{\infty}_{0}\|\mu\mid_{K}(B_{r_{0}}(\cdot,t_{0}))\|_{L^{p'}(\mathbb R^{n})}\frac{dr}{r^{n+1}}.$$

Take $\delta=(\mu(K))^{p/nq}$. On the one hand,
\begin{eqnarray*}
\|\mu\mid_{K}(B_{r}(\cdot,t_{0}))\|_{L^{p'}(\mathbb R^{n})}^{p'}
&=&\int_{\mathbb R^{n}}\left(\mu(K\cap B_{r}(x_{0}, t_{0}))\right)^{p'}dx_{0}\\
&\lesssim&(\mu(K))^{p'-1}\int_{\mathbb R^{n}}\mu(K\cap B_{r}(x_{0}, t_{0}))dx_{0}\\
&=&(\mu(K))^{p'-1}\int_{\mathbb R^{n}}\left(\int_{K\cap B_{r}(x_{0}, t_{0})}1d\mu(t,y)\right)dx_{0}\\
&\lesssim&(\mu(K))^{p'-1}\int_{K}\left(\int_{|y-x_{0}|<r}dx_{0}\right)d\mu(t,y)\\
&\lesssim&(\mu(K))^{p'}r^{n}.
\end{eqnarray*}
The above estimate gives
\begin{eqnarray*}
\int^{\infty}_{\delta}\|\mu\mid_{K}(B_{r}(\cdot,t_{0}))\|_{L^{p'}(\mathbb R^{n})}\frac{dr}{r^{n+1}}&\lesssim&\int^{\infty}_{\delta}\mu(K)r^{n/p'-n-1}dr
\lesssim \mu(K)\delta^{-n/p}.
\end{eqnarray*}
Meanwhile, on the other hand, since $\mu(B_{r}(x_{0}, t_{0}))\lesssim r^{nq/p}$, we can obtain
\begin{eqnarray*}
\|\mu\mid_{K}(B_{r}(\cdot,t_{0}))\|_{L^{p'}(\mathbb R^{n})}^{p'} &\lesssim& r^{nq(p'-1)/p}\int_{\mathbb R^{n}}\mu\mid_{K}(B_{r}(t_{0},x_{0}))dx_{0}\\
&=&r^{nq(p'-1)/p}\int_{\mathbb R^{n}}\left(\int_{K\cap B_{r}(x_{0}, t_{0})}1d\mu(t,y)\right)dx_{0}\\
&\lesssim&r^{nq(p'-1)/p}\int_{K}\left(\int_{|y-x_{0}|<r}dx_{0}\right)d\mu(t,y)\\
&\lesssim& \mu(K)r^{n+nq/(p^{2}-p)}.
\end{eqnarray*}
Hence, noting that $q>p$, we obtain
$$\int^{\delta}_{0}\|\mu\mid_{K}(B_{r}(\cdot,t_{0}))\|_{L^{p'}(\mathbb R^{n})}\frac{dr}{r^{1+n}}\lesssim (\mu(K))^{1/p'}\delta^{n(q-p)/p^{2}}.$$
Finally, we get
\begin{eqnarray}\label{eq-3.3}
\|P^{\ast}_{\alpha}\mu\mid_{K}\|_{L^{p'}(\mathbb R^{n})}&\lesssim&\int^{\infty}_{0}\|\mu\mid_{K}(B_{r}(\cdot,t_{0}))\|_{L^{p'}(\mathbb R^{n})}\frac{dr}{r^{n+1}}\\
&=&\left(\int^{\delta}_{0}+\int_{\delta}^{\infty}\right)\|\mu\mid_{K}(B_{r}(\cdot,t_{0}))\|_{L^{p'}(\mathbb R^{n})}\frac{dr}{r^{n+1}}\nonumber\\
&\lesssim&(\mu(K))^{1/p'}\delta^{n(q-p)/p^{2}}+(\mu(K))\delta^{-n/p}\nonumber\\
&\lesssim&(\mu(K))^{1/q'}.\nonumber
\end{eqnarray}
Below we prove that
$(\mu(K))^{1/q}\leq (C^{\alpha,p}_{\mathbb{R}^{n+1}_+}(K))^{1/p}$ for any compact set $K\subset \mathbb R^{n+1}_{+}$.
In fact, let $f\in C^{\infty}_{0}(\mathbb R^{n})$ and $E_{\lambda,K}=\{(x,t)\in K:\ |P_{\alpha}\ast f(x,t)|\geq \lambda\}$. Then the set $E_{\lambda,K}$ is compact.
It follows from (\ref{eq-3.3}) that
\begin{eqnarray}\label{eq-3.4}
\lambda\mu(E_{\lambda,K})&\leq&\int_{\mathbb R^{n}}|P_{\alpha}f(x,t)|d\mu\mid_{E_{\lambda,K}}\\
&\leq&\|f\|_{L^{p}(\mathbb{R}^n)}\|P^{\ast}_{\alpha}\mu\mid_{E_{\lambda,K}}\|_{L^{p'}}\nonumber\\
&\lesssim&\|f\|_{L^{p}(\mathbb{R}^n)}(\mu(E_{\lambda,K}))^{1/q'}.\nonumber
\end{eqnarray}
For an arbitrary $f\in L^{p}(\mathbb R^{n})$, via approximating $f$ by a sequence from  $C^{\infty}_{0}(\mathbb R^{n})$ in the $L^{p}$-norm, we can prove that (\ref{eq-3.4}) holds for $f$. For $f\in L^{p}(\mathbb R^{n})$ such that $P_{\alpha}f\geq 1$ on $K,$ we have $E_{\lambda,K}=K.$ Then  (\ref{eq-3.4}) implies $(\mu(K))^{1/q}\lesssim \|f\|_{L^{p}(\mathbb{R}^n)}$, which gives $(\mu(K))^{1/q}\leq (C^{\alpha,p}_{\mathbb{R}^{n+1}_+}(K))^{1/p}.$ Recall that $\mu(K)\geq \lambda$. Then taking the infimum over the compact sets $K$ such that $\mu(K) \geq \lambda$, we get
$\lambda^{p/q}\lesssim C^{\alpha,p}_{\mathbb{R}^{n+1}_+}(K)$, i.e.,
$$\lambda^{p/q}\lesssim \inf\left\{C^{\alpha,p}_{\mathbb{R}^{n+1}_+}(K):\ \text{ compact }K\subset \mathbb R^{n+1}_{+},\quad  \mu(K)\geq \lambda\right\}=c_{\alpha,p}(\mu; \lambda).$$
\end{proof}

\subsection{Embeddings  of $L^p(\mathbb{R}^{n})$ to $L^q(\mathbb{R}^{n+1}_+,\mu)$  when  $p>q$ }
Let $M\mu(x)=\sup_{r>0}r^{-n}\mu\left({B_{r}(x, r)}\right)$ be the parabolic maximal function of a nonnegative Radon measure $\mu$ on $\mathbb R^{n+1}_+$. We show the embedding for $p>q$ inspired by some  ideas from  \cite{Shi Xiao}, which needs the following $L^p$-boundedness of $M\mu$.

 \begin{lemma}\label{l21}  Let $1<p<\infty$. Then
 $\|M\mu\|_{L^{p}(\mathbb{R}^{n})}\approx \|P^\ast_\alpha\mu\|_{L^p(\mathbb R^n)}.$
 \end{lemma}
 \begin{proof}
 It is easy to check that for any $r>0$,
 $$
 P^\ast_\alpha\mu(x)\gtrsim \int_{B_{r}(x, r)}\frac{t^{\alpha}}{(t^{2}+|x-y|^{2})^{{(n+\alpha)}/{2}}}d\mu(y,t)\gtrsim r^{-n}\mu(B_{r}(x, r)),
 $$
 and so that
 $ \|M\mu\|_{L^{p}(\mathbb{R}^{n})}\lesssim \|P^\ast_{\alpha}\mu\|_{L^{p}(\mathbb{R}^{n})}.$
Next, we are about to prove the converse inequality by a slight modification of \cite[Theorem 3.6.1]{AH}.
Denote by $E_{\mu}(T,\lambda,r)=\{x\in \mathbb{R}^{n}:T\mu(x)>\lambda r\}$ for an operator $T$ and $(\lambda,r) \in (\mathbb{R}, \mathbb{R})$. Then
by the famous good $\lambda$ inequality, there exist two constants $c_{1}>1$ and $c_{2}>0$ such that
 $$
\left |E_{\mu}(P^\ast_{\alpha},c_{1},\rho)\right|\leq c_{2}\theta^{{(n+\alpha)}/{n}}\left|E_{\mu}(P^\ast_{\alpha},1,\rho)\right|
 +\left|E_{\mu}(M,\theta,\rho)\right|\,\,\,\,\hbox{for any}\,\, \rho>0\,\,\hbox{and}\,\, 0<\theta\leq 1.
$$
 Hence, for any $\tau>0$, we get
 $$
 \int_{0}^{\tau}\left|E_{\mu}(P^\ast_{\alpha},c_{1},\rho)\right|\rho^{p-1}d\rho \leq c_{2}\theta^{{(n+\alpha)}/{n}}\int_{0}^{\tau}\left|E_{\mu}(P^\ast_{\alpha},1,\rho)\right|\rho^{p-1}d\rho+\int_{0}^{\tau}\left|E_{\mu}(M,\theta,\rho)\right|\rho^{p-1}d\rho.
$$
 Namely,
 $$
 c_{1}^{-p}\int_{0}^{c_{1}\tau}\left|E_{\mu}(P^\ast_{\alpha},1,\rho)\right|\rho^{p-1}d\rho\leq c_{2}\theta^{{(n+\alpha)}/{n}}\int_{0}^{\tau}\left|E_{\mu}(P^\ast_{\alpha},1,\rho)\right|\rho^{p-1}d\rho+\theta^{-p}\int_{0}^{\theta \tau}\left|E_{\mu}(M,1,\rho)\right|\rho^{p-1}d\rho.
$$
  Accordingly,
 $$
 c_{1}^{-p}\int_{\mathbb{R}^{n}}(P^\ast_{\alpha}\mu(x))^{p}dx\leq 2\theta^{-p}\int_{\mathbb{R}^{n}}(M\mu(x))^{p}dx
 $$
 by letting $\theta$ be so small such that $c_{2}\theta^{{(n+\alpha)}/{n}}\leq \frac{1}{2}c_{1}^{-p}$ and $\tau\to\infty$.
 The foregoing analysis yields
 $$
 \|M\mu\|_{L^{p}(\mathbb{R}^{n})}\gtrsim \|P^\ast_{\alpha}\mu\|_{L^{p}(\mathbb{R}^{n})},
 $$
 which is the desired one.
 \end{proof}

\begin{lemma}\label{l22}
Denote by
$$ H_{p}\mu(x,t):=\int_{0}^{\infty}\left(r^{-n}\mu(B_{r}(x, t))\right)^{p'-1}r^{-1}dr$$
the Hedberg-Wolff  potential of $\mu.$ Let $1<p<\infty$ and $\mu\in \mathcal M_{+}(\mathbb R^{n+1}_{+}).$ Then one has
$$\|P_{\alpha}^{*}\mu\|_{L^{p'}(\mathbb{R}^{n})}^{p'}\approx \int_{\mathbb{R}_{+}^{n+1}}H_{p}\mu\, d\mu.$$
\end{lemma}

\begin{proof}
We conclude from the fact
$$
\frac{\mu(B_{r}(x, r))}{r^{n}}\lesssim \left(\int_{r}^{2r}\left(\frac{\mu(B_{s}(x, s))}{s^{n}}\right)^{p'}\frac{ds}{s}\right)^{{1}/{p'}}
$$
that
$$
M\mu(x)\lesssim \left(\int_{0}^{\infty}\left(\frac{\mu(B_{s}(x, s))}{s^{n}}\right)^{p'}\frac{ds}{s}\right)^{{1}/{p'}}.
$$

Since
$$
\int_{\mathbb{R}^{n}}\int_{0}^{\infty}\left(\frac{\mu(B_{r}(x, r))}{r^{n}}\right)^{p'}\frac{drdx}{r}=\int_{0}^{\infty}\int_{\mathbb{R}^{n}}\frac{\mu(B_{r}(x, r))^{p'}}{r^{np'+1}}dxdr,$$
by the Fubini theorem,
$$
\int_{\mathbb{R}^{n}}\mu(B_{r}(x, r))^{p'}dx \lesssim\int_{B_{r}(y, r)}\int_{\mathbb{R}^{n}}\mu(B_{r}(x, r))^{p'-1}dxd\mu\lesssim r^{n}\int_{B_{\frac{3r}{2}}(y, \frac{r}{2})}\mu(B_{r}(y, r))^{p'-1}d\mu.
$$
Accordingly,
$$
\int_{0}^{\infty}\int_{\mathbb{R}^{n}}\frac{\mu(B_{r}(x, r))^{p'}}{r^{np'+1}}dxdr
\lesssim \int_{\mathbb{R}_{+}^{1+n}}\left(\int_{0}^{\infty}\left(\frac{\mu(B_{r}(x, r))}{r^{n}}\right)^{p'-1}\frac{dr}{r}\right)d\mu,
$$
which, together with Lemma \ref{l21}, gives
$$\|P^\ast_{\alpha}\mu\|_{L^{p'}(\mathbb{R}^{n})}^{p'}\lesssim \int_{\mathbb{R}_{+}^{n+1}}H_{p}\mu\,d\mu.$$

Write
$$B(y,2^{-m})=\left\{x\in\mathbb R^n:\ |x-y|<2^{-m}\ \&\ 2^{-m}<t<2^{1-m}\right\}\quad \forall\,\, m\in\mathbb Z\equiv\left\{0,\pm 1,\pm 2,\ldots\right\}.$$
The converse inequality is a by-product of the following estimate
\begin{align*}
\int_{\mathbb{R}^{n}}(P^\ast_{\alpha}\mu(x))^{p'-1}p_{t}^{\alpha}(x-y)dx
&\quad\approx\int_{\mathbb{R}^{n}}\frac{t^{\alpha}}{(t^{2}+|x-y|^{2})^{{(n+\alpha)}/ {2}}}
\left(\int_{\mathbb{R}_{+}^{n+1}}\frac{s^{\alpha}}{(s^{2}+|x-z|^{2})^{{(n+\alpha)}/{2}}}d\mu\right)^{p'-1}dx\\
&\quad\gtrsim \sum_{m\in \mathbb{Z}}\int_{B(y,2^{-m})}t^{-n}
\left(\int_{B_{2^{-m}}(y, t)}s^{-n}d\mu\right)^{p'-1}dx\\
&\quad\gtrsim \sum_{m\in \mathbb{Z}}\int_{B(y,2^{-m})}{2^{mn}}
\left(\frac{\mu(B_{2^{-m}}(y, t))}{2^{-mn}}\right)^{p'-1}dx\\
&\quad\gtrsim \int_{0}^{\infty}
\left(\frac{\mu(B_{r}(y, t))}{r^{n}}\right)^{p'-1}\frac{dr}{r}
\end{align*}
since
$$
\|P^\ast_{\alpha}\mu\|_{L^{p'}(\mathbb{R}^{n})}^{p'}=\int_{\mathbb{R}^{n}}\left(P^\ast_{\alpha}\mu(x)\right)^{p'-1}(P^\ast_{\alpha}\mu(x))dx=\int_{\mathbb{R}_{+}^{n+1}}\int_{\mathbb{R}^{n}}(P^\ast_{\alpha}\mu(x))^{p'-1}p_{t}^{\alpha}(x-y)dx\,d\mu(y,t).
$$

Therefore,
 $$\|P^\ast_{\alpha}\mu\|_{L^{p'}(\mathbb{R}^{n})}^{p'}\gtrsim \int_{\mathbb{R}_{+}^{n+1}}H_{p}\mu\,d\mu
 $$
as desired.
\end{proof}

%Let $\mu$ be a nonnegative Radon measure and $\lambda>0$.  For abbreviation, we write
%$$
%c^{\alpha}(\mu;\lambda)=\inf\left\{C^{\alpha,p}_{\mathbb{R}^{n+1}_+}(K): \ %\hbox{compact}\,\, K\subset \mathbb{R}_{+}^{1+n}\,\,\&\,\,\mu(K)\geq %\lambda\right\}.
%$$
Now, we are ready to characterize the embedding \eqref{5} for $1<q<p<\infty$ as follows.
\begin{theorem}\label{Them 3.4}
Let $1<q<p<\infty$ and $\mu\in\mathcal M_{+}(\mathbb R^{n+1}_{+}).$ The following statements are true.
    \item{\rm (i)} The embedding \eqref{5} holds if and only if
$$\int^{\infty}_{0}\left(\frac{\lambda^{p/q}}{c(\mu; \lambda)}\right)^{q/(p-q)}\frac{d\lambda}{\lambda}<\infty.$$

\item{\rm (ii)} The embedding \eqref{5} holds if and only if
$$\int_{\mathbb R^{n+1}_+}\left\{\int_0^\infty \left(\frac{\mu(B_{r}(x_{0}, t_{0}))}{C^{\alpha,p}_{\mathbb{R}^{n+1}_+}(B_{r}(x_{0}, t_{0}))}\right)^{1/(p-1)}\,\frac{dr}{r}\right\}^{{q(p-1)}/{(p-q)}}\,d\mu(x_0,t_0)<\infty.
$$

\item{\rm (iii)}  The embedding \eqref{5} holds if and only if $H_{p}\mu \in L^{q(p-1)/(p-q)}(\mathbb{R}_{+}^{n+1}, \mu).$

\end{theorem}

\begin{proof}
It is easy to see that the  statement (ii) is a consequence of the statement (iii). We only need to show statements (i)\ $\&$\ (iii). The rest of the proof is divided into two parts.

{\it Part I: Proof of statement (i)}.
  Let $P_{\alpha}: L^{p}(\mathbb R^{n})\rightarrow L^{q}(\mathbb R^{n+1}_{+},\mu)$ be bounded. Then
$$\left(\int_{\mathbb R^{n+1}_{+}}|P_{\alpha}f(x,t)|^{q}d\mu\right)^{1/q}\lesssim \|f\|_{L^{p}(\mathbb R^{n})}.$$
By Markov's inequality, we can get
$$\sup_{\lambda>0}\lambda\left(\mu(E_{\lambda}(f))\right)^{1/q}\lesssim \|f\|_{L^{p}(\mathbb R^{n})},$$
where $E_\lambda(f)=\{(x,t)\in 
\mathbb{R}^{n+1}_+: |P_{\alpha}f(x,t)|\geq \lambda\}.$
By the layer cake representation, we get
$$\left(\int^{\lambda}_{0}\mu(E_{\tau}(f))\tau^{q-1}d\tau\right)^{1/q}\lesssim\|f\|_{L^{p}(\mathbb R^{n})}\quad \forall\ \lambda>0.$$
Then we have
$$(\mu(E_{\lambda}(f)))^{1/q}\left(\int^{\lambda}_{0}\tau^{q-1}d\tau\right)^{1/q}\lesssim\|f\|_{L^{p}(\mathbb R^{n})}.$$
 For each integer $j$, there is a compact set $K_{j}\subset \mathbb R^{n+1}_{+}$ and a nonnegative function $f_{j}\in L^{p}(\mathbb R^{n})$ such that
$$\begin{cases}
C^{\alpha,p}_{\mathbb{R}^{n+1}_+}(K_{j})\leq 2c_{\alpha,p}(\mu; 2^{j});\\
\mu(K_{j})>2^{j};\\
P_{\alpha}f_{j}\geq 1_{K_{j}};\\
\|f_{j}\|^{p}_{L^{p}(\mathbb R^{n})}\leq 2C^{\alpha,p}_{\mathbb{R}^{n+1}_+}(K_{j}).
\end{cases}$$
For the integers $i,k$ with $i<k$, let
$f_{i,k}:=\sup\limits_{i\leq j\leq k}\left(\frac{2^{j}}{c_{\alpha,p}(\mu; 2^{j})}\right)^{1/(p-q)}f_{j}.$
Then
\begin{eqnarray}\label{eq-3.5}
\|f_{i,k}\|^{p}_{L^{p}(\mathbb R^{n})}&\leq&\sum^{k}_{j=i}\left(\frac{2^{j}}{c_{\alpha,p}(\mu; 2^{j})}\right)^{{p}/{(p-q)}}\|f_{j}\|^{p}_{L^{p}(\mathbb R^{n})}\\
&\lesssim&\sum^{k}_{j=i}\left(\frac{2^{j}}{c_{\alpha,p}(\mu; 2^{j})}\right)^{{p}/{(p-q)}}C^{\alpha,p}_{\mathbb{R}^{n+1}_+}(K_{j})\nonumber\\
&\lesssim&\sum^{k}_{j=i}\left(\frac{2^{j}}{c_{\alpha,p}(\mu; 2^{j})}\right)^{{p}/{(p-q)}}c_{\alpha,p}(\mu; 2^{j}).\nonumber
\end{eqnarray}
Note that if $(x,t)\in K_{j}$, then
\begin{eqnarray}\label{eq-3.2}
P_{\alpha}f_{i,k}(x,t)&=&P_{\alpha}\left(\sup_{i\leq l\leq k}\left(\frac{2^{j}}{c_{\alpha,p}(\mu; 2^{l})}\right)^{1/(p-q)}f_{l}(x,t)\right)\\
&\geq&\left(\frac{2^{j}}{c_{\alpha,p}(\mu; 2^{j})}\right)^{1/(p-q)}P_{\alpha}f_{j}(x,t)\nonumber\\
&\geq&\left(\frac{2^{j}}{c_{\alpha,p}(\mu; 2^{j})}\right)^{1/(p-q)}.\nonumber
\end{eqnarray}
It can be seen from (\ref{eq-3.2}) that
$$K_{j}\subset \left\{(x,t)\in\mathbb R^{n+1}_{+}:\ P_{\alpha}f_{i,k}(x,t)\geq\left(\frac{2^{j}}{c_{\alpha,p}(\mu; 2^{j})}\right)^{1/(p-q)}\right\}.$$
This means that
\begin{equation}\label{eq-3.6}
2^{j}<\mu(K_{j})\leq\mu\left(E_{({2^{j}}/{c_{\alpha,p}(\mu; 2^{j})})^{1/(p-q)}}(f_{i,k})\right).
\end{equation}
In view of the embedding (\ref{5}), we can use the (nonsymmetric) decreasing rearrangement (\cite[Theorem 1.8]{Bennett}), and then a simple computation to obtain
\begin{eqnarray*}
\|f_{i,k}\|_{L^{p}(\mathbb R^{n})}^{q}&\gtrsim&\int_{\mathbb R^{n+1}_{+}}|P_{\alpha}(f_{i,k})(x,t)|^{q}d\mu(x,t)\\
&\approx&\int^{\infty}_{0}\left(\inf\left\{\lambda:\ \mu(E_{\lambda}(f_{i,k}))\leq s\right\}\right)^{q}ds.
\end{eqnarray*}
Notice that (\ref{eq-3.6}) indicates that
$$\inf \Big\{ \lambda:\ \mu(E_\lambda(f_{i,k}) \leq 2^j \Big\} > \Big(\frac{2^j}{c_{\alpha,p}(\mu;2^j)}\Big)^{1/(p-q)}.$$
Then we split the integral to get
\begin{eqnarray*}
\|f_{i,k}\|_{L^{p}(\mathbb R^{n})}^{q}&\geq&\sum^{k}_{j=i}2^{j}\left(\inf\left\{\lambda:\ \mu(E_{\lambda}(f_{i,k}))\leq 2^{j}\right\}\right)^{q}\geq\sum^{k}_{j=i}2^{j}\left(\frac{2^{j}}{c_{\alpha,p}(\mu; 2^{j})}\right)^{q/(p-q)}.
\end{eqnarray*}
Finally, applying (\ref{eq-3.5}) and a direct computation, we have
\begin{eqnarray*}
\|f_{i,k}\|_{L^{p}(\mathbb R^{n})}^{q}&\geq&\left(\frac{\sum\limits_{j=i}^k\left({2^j}/{c_{\alpha,p}(\mu;2^j)}\right)^{{q}/{(p-q)}}2^j}{\left(\sum\limits_{j=i}^k
\left({2^j}/{c_{\alpha,p}(\mu;2^j)}\right)
 ^{{p}/{(p-q)}}c_{\alpha,p}(\mu;2^j)\right)^{{q}/{p}}}\right)\|f_{i,k}\|_{L^p(\mathbb{R}^{n})}^q\\
 &\approx& \left(\sum_{j=i}^k\frac{2^{{jp}/{(p-q)}}}{\left(c_{\alpha,p}(\mu;2^j)\right)^{{q}/{(p-q)}}}\right)^{{(p-q)}/{p}}
 \|f_{i,k}\|_{L^p(\mathbb{R}^{n})}^q,
\end{eqnarray*}
which implies
$$\sum_{j=i}^k\frac{2^{{jp}/{(p-q)}}}{\left(c_{\alpha,p}(\mu;2^j)\right)^{{q}/{(p-q)}}}\lesssim 1,$$
and hence
 $$
\int_0^\infty\left(\lambda^{p/q}/c_{\alpha,p}(\mu;\lambda)\right)^{{q}/{(p-q)}}\lambda^{-1}\,d\lambda\lesssim
\sum_{j=-\infty}^\infty\frac{2^{{jp}/{(p-q)}}}{\left(c_{\alpha,p}(\mu;2^j)\right)^{{q}/{(p-q)}}}\lesssim
1.
 $$

 Conversely, let
 $$
 I_{p,q}(\mu)=\int_0^\infty\left(\frac{\lambda^{p/q}}{c_{\alpha,p}(\mu;\lambda)}\right)^{{q}/{(p-q)}}\frac{d\lambda}{\lambda}<\infty.
 $$
 Now for  $f\in C^{\infty}_0(\mathbb{R}^{n})$, let
 $$
 S_{p,q}(\mu;f)=\sum_{j=-\infty}^\infty
 \frac{\left(\mu\left(E_{2^j}(f)\right)-\mu\left(E_{2^{j+1}}(f)\right)\right)^{{p}/{(p-q)}}}{\left(C^{\alpha,p}_{\mathbb R^{n+1}_{+}}\left(E_{2^j}(f)\right)\right)^{{q}/{(p-q)}}}.
 $$
 Using layer-cake representation, H\"older's inequality and Lemma \ref{le-4.2}, we obtain
 \begin{align*}
 &\int_{ \mathbb{R}^{n+1}_+}|P_{\alpha}f(x,t)|^q\, d\mu(x,t)\\
 &\ \ =q\int_{0}^{\infty}\lambda^{q-1}\mu(E_\lambda(f))\, d\lambda\\
 &\ \ \lesssim\sum_{j=-\infty}^\infty\left(\mu\left(E_{2^j}(f)\right)-\mu(E_{2^{j+1}}(f)\right)2^{jq}\\
 &\ \ \lesssim(S_{p,q}(\mu;f))^{{(p-q)}/{p}}
 \left(\sum_{j=-\infty}^\infty 2^{jp}C^{\alpha,p}_{\mathbb R^{n+1}_{+}}\left(E_{2^j}(f)\right)\right)^{q/p}\\
 &\ \ \lesssim(S_{p,q}(\mu;f))^{{(p-q)}/{p}}\left(\int_0^\infty C^{\alpha,p}_{\mathbb R^{n+1}_{+}}\left(\left\{(x,t)\in \mathbb{R}^{n+1}_+:\  |P_{\alpha}f(x,t)|>\lambda\right\}\right)\,d\lambda^p\right)^{{q}/{p}}\\
 &\ \ \lesssim(S_{p,q}(\mu;f))^{{(p-q)}/{p}}\|f\|_{L^p(\mathbb{R}^{n})}^q.
 \end{align*}
 Note also that
 \begin{align*}
 \left(S_{p,q}(\mu;f)\right)^{{(p-q)}/{p}}
 &\ \ =\left\{\sum_{j=-\infty}^\infty\frac{\left(\mu(E_{2^j}(f))-\mu(E_{2^{j+1}}(f))\right)^{{p}/{(p-q)}}}
 {(C^{\alpha,p}_{\mathbb R^{n+1}_{+}}\left(E_{2^j}(f)\right))^{{q}/{(p-q)}}}\right\}^{{(p-q)}/{p}}\\
  &\ \ \lesssim\left\{\sum_{j=-\infty}^\infty\frac{\left(\mu(E_{2^j}(f))-\mu(E_{2^{j+1}}(f))\right)^{{p}/{(p-q)}}}
  {\left(c_{\alpha,p}(\mu;\mu\left(E_{2^j}(f)\right))\right)^{{q}/{(p-q)}}}\right\}^{{(p-q)}/{p}}\\
  &\ \ \lesssim\left\{\sum_{j=-\infty}^\infty\frac{\left(\mu(E_{2^j}(f))\right)^{{p}/{(p-q)}}-\left(\mu(E_{2^{j+1}}(f))\right)^{{p}/{(p-q)}}}
  {\left(c_{\alpha,p}(\mu;\mu\left(E_{2^j}(f)\right))\right)^{{q}/{(p-q)}}}\right\}^{{(p-q)}/{p}}\\
    &\ \ \lesssim\left(\int_0^\infty\frac{d s^{{p}/{(p-q)}}}{\left(c_{\alpha,p}(\mu; s)\right)^{{q}/{(p-q)}}}\right)^{{(p-q)}/{p}}\\
 &\ \ \simeq \left(I_{p,q}(\mu)\right)^{{(p-q)}/{p}}.
 \end{align*}
 Therefore,
 $$
 \left(\int_{ \mathbb{R}^{n+1}_+}|P_{\alpha}f(x,t)|^q\,d\mu(x,t)\right)^{1/q}\lesssim \left(I_{p,q}(\mu)\right)^{{(p-q)}/{pq}}\|f\|_{L^p(\mathbb{R}^{n})}.
 $$

{\it Part II: Proof of statements (iii)}.
Similar to \cite[Theorem 2.1]{COV}, 
this part consists of two steps.

{\it Step 1} - proving that
$
\label{right} \eqref{5}\Rightarrow H_{p}\mu \in L_\mu^{q(p-1)/(p-q)}(\mathbb{R}_{+}^{n+1}).
$

For $m_{0}\in \mathbb Z_{+}={0,1,2,...},$ and $m_{k}\in \mathbb Z$, $k=1,2,...$, denote by $\Theta_{l}$ the following dyadic cube  with side length $l\equiv l(\Theta_{l})$:
$$\Theta_{l}\equiv \left(m_{1}l, (m_{1}+1)l\right)\times\cdots\times\left(m_{n}l, (m_{n}+1)l\right)\times \left(m_{0}l, (m_{0}+1)l\right).$$
Let $\mathbf{\Theta}=\{\Theta_{l}\}$ be the family of all the above-defined-dyadic cubes in $\mathbb R^{n+1}_+$. Setting
$$H_{p}^{d}\mu:=H_{p}^{d}\mu(x,t)=\sum\limits_{\Theta_{l}\in\mathbf{\Theta}}\left({\mu(\Theta_{l})}/{l^{n}}\right)^{p'-1}\textbf{1}_{\Theta_{l}}(x,t),$$
we first  show that
\begin{equation}\label{4.23}
\eqref{5}\Rightarrow\int_{\mathbb{R}_{+}^{n+1}}(H_{p}^{d}\mu)^{{q(p-1)}/{(p-q)}}d\mu<\infty.
\end{equation}
In fact,
$(\ref{5})$ is equivalent to the following inequality by the  duality
$$
\|P_{\alpha}^{*}(fd\mu)\|_{L^{p'}(\mathbb{R}^{n})}^{p'}\lesssim \|f\|_{L^{q'}(\mathbb{R}_{+}^{n+1}, \mu)}^{p'} \quad\forall f\in L^{q'}(\mathbb{R}_{+}^{n+1}, \mu).
$$
Applying Lemma \ref{l22} to $H_{p}^{d}\mu$ and $f d\mu$, and using Jensen's inequality, we thus get
$$
\|P_{\alpha}^{*}(f d\mu)\|_{L^{p'}(\mathbb{R}^{n})}^{p'}\gtrsim \int_{\mathbb{R}_{+}^{n+1}}H_{p}^{d}(fd\mu)(x,t)f(x,t)d\mu
\gtrsim \sum_{\Theta_{l}}\left\{{\int_{\Theta_{l}}f(x,t)d\mu}\right\}^{p'}l^{n(1-p')},
$$
and hence
$$
\sum_{\Theta_{l}}\left\{{\int_{\Theta_{l}}f(x,t)d\mu}\right\}^{p'}l^{n(1-p')}\lesssim \|f\|_{L_\mu^{q'}(\mathbb{R}_{+}^{n+1})}^{p'},
$$
which is equivalent to
\begin{equation}\label{4.25''}
\sum_{\Theta_{l}}g_{\Theta_{l}}\left({\int_{\Theta_{l}}f(x,t)\,d\mu(x,t)}\right)^{p'}\left(\mu(\Theta_{l})\right)^{-p'}\lesssim \|f\|_{L_\mu^{q'}(\mathbb{R}_{+}^{n+1})}^{p'} \quad\hbox{with}\quad
g_{\Theta_{l}}=\left(\mu(\Theta_{l})\right)^{p'}l^{n(1-p')}.
\end{equation}

For $0\le \widetilde{f}\in L_\mu^{{q'}/{p'}}(\mathbb{R}_{+}^{n+1})$, set
$M(x,t)=(M_{\mu}^{d}\widetilde{f})^{{1}/{p'}}(x,t).$
Here $M_{\mu}^{d}$ is the dyadic Hardy-Littlewood maximal function defined as
$$
M_{\mu}^{d}f(x,t)=\sup_{(x,t)\in \Theta_{l}\in \mathbf{\Theta}}\frac{1}{\mu(\Theta_{l})}\int_{\Theta_{l}}|f(y,s)|d\mu.
$$
Applying $\eqref{4.25''}$ to $M$ in place of $f$, and using $
\|M\|_{L^{q'}(\mathbb{R}_{+}^{n+1}, \mu)}^{p'}\lesssim \|\widetilde{f}\|_{L^{{q'}/{p'}}(\mathbb{R}_{+}^{n+1},\mu)},
$ we find
$$
\sum_{\Theta_{l}}g_{\Theta_{l}} \left(\frac{1}{\mu(\Theta_{l})}\int_{\Theta_{l}}f(x,t)d\mu(x,y)\right)^{p'}\lesssim \|M\|_{L^{q'}(\mathbb{R}_{+}^{n+1}, \mu)}^{p'}\lesssim \|\widetilde{f}\|_{L^{{q'}/{p'}}(\mathbb{R}_{+}^{n+1},\mu)}.
$$
Then using
$$
\left(\frac{1}{\mu(\Theta_{l})}\int_{\Theta_{l}}M(x,t)d\mu\right)^{p'}\gtrsim \frac{1}{\mu(\Theta_{l})}\int_{\Theta_{l}}\widetilde{f}(x,t)d\mu,
$$
one has
$$
\sum_{\Theta_{l}}\frac{g_{\Theta_{l}}}{\mu(\Theta_{l})}\int_{\Theta_{l}}\widetilde{f}(x,t)d\mu\lesssim \|\widetilde{f}\|_{L^{{q'}/{p'}}(\mathbb{R}_{+}^{n+1},\mu)}.
$$
Thereby getting by the duality that
$$
\int_{\mathbb{R}_{+}^{n+1}}\sum_{\Theta_{l}}\frac{g_{\Theta_{l}}}{\mu(\Theta_{l})}\textbf{1}_{\Theta_{l}}\widetilde{f}(x,t)d\mu=\sum_{\Theta_{l}}\frac{g_{\Theta_{l}}}{\mu(\Theta_{l})}\textbf{1}_{\Theta_{l}}\in L^{{q'}/{(q'-p')}}(\mathbb{R}_{+}^{n+1},\mu)$$
i.e.$$
\sum_{\Theta_{l}}\left(\frac{\mu(\Theta_{l})}{l^{n}}\right)^{p'-1}\textbf{1}_{\Theta_{l}}\in L^{{q(p-1)}/{(p-q)}}(\mathbb{R}_{+}^{n+1},\mu),
$$
which shows \eqref{4.23} as desired.

We proceed the proof by letting
$$\begin{cases}
H_{p}^{d,\lambda}\mu(x,t)=\sum_{\Theta_{l}\in\mathbf{\Theta}_{\lambda}}\left(\frac{\mu(\Theta_{l})}{l^{n}}\right)^{p'-1}\textbf{1}_{\Theta_{l}}(x,t);\\
\mathbf{\Theta}_{\lambda}=\mathbf{\Theta}+\lambda=\{\Theta_{l}+\lambda\}_{\Theta_{l}\in \mathbf{\Theta}};\\
\Theta_{l}+\lambda=\{(x,t)+\lambda:\ (x,t)\in \Theta_{l}\}.
\end{cases}$$
 Then
\begin{equation}\label{4.25}
\sup_{\lambda\in \mathbb{R}_{+}^{n+1}}\int_{\mathbb{R}_{+}^{n+1}}\left(H_{p}^{d,\lambda}\mu(x,t)\right)^{{q(p-1)}/{(p-q)}}d\mu<\infty
\end{equation}
by the arguments leading to \eqref{4.23}(apply just the same for $\Theta_{l}$ instead of $\Theta$). Hence Step 1 is completed by showing that
\begin{equation}\label{4.25'}
H_{p}\mu \in L_\mu^{q(p-1)/(p-q)}(\mathbb{R}_{+}^{n+1}),
\end{equation}
which will be considered in two cases.

{\it Case 1. $\mu$ is a doubling measure.} $(\ref{4.25'})$ can be obtained by $(\ref{4.23})$ and the following estimate
$$
H_{p}\mu(x,t)\lesssim \sum_{\Theta_{l}}\left(\frac{\mu(\Theta_{l}^{*})}{l^{n}}\right)^{p'-1}\textbf{1}_{\Theta_{l}}(x,t),
$$
where $\Theta^{*}_{l}$ is the cube with the same center as $\Theta_{l}$ and side length two times as $\Theta_{l}$.

{\it Case 2. $\mu$ is a possibly non-doubling measure.} In this case, we first claim that for any $\delta>0$,
\begin{equation}\label{4.26}
H_{p,\delta}\mu(x,t):=\int_{0}^{\delta}\left(\frac{\mu(B_{r}(x,t))}{r^{n}}\right)^{p'-1}\frac{dr}{r}\lesssim \delta^{-(n+1)}\int_{|\lambda|\lesssim\delta}H_{p}^{d,\lambda}\mu(x,t)d\lambda.
\end{equation}
Indeed, for fixed $x\in \mathbb{R}^{n}$ and $\delta>0$ with $2^{i-1}\xi\leq \delta< 2^{i}\xi$, where $i\in\mathbb{Z}$ and $\xi>0$ will be determined later, we have
$$
H_{p,\delta}\mu(x,t)\lesssim \sum_{j=-\infty}^{i}\left(\frac{\mu(B_{2^{j}\xi}(x,t))}{(2^{j}\xi)^{n}}\right)^{p'-1}.
$$
 Assume that $\Theta_{h,j}$ is a cube centered at $x$, of length $h$, with $2^{j-1}<h\leq2^{j}$ for $j\leq i$, then $B_{2^{j}\xi}(x,t)\subseteq \Theta_{h,j}$ for sufficiently small $\xi$ and $\xi$ does not depend on $x$. Suppose that $F=\{\lambda: \lambda\in \mathbb{R}_{+}^{n+1}, \,\,|\lambda|\lesssim \delta\}$ and there exists $\Theta_{l}^{\lambda}\in \Theta_{l}$ satisfying $l=2^{j+1}$ and $\Theta_{l,j}\subseteq \Theta^{\lambda}_{l}$. Then  it is immediate that there is a constant $c(n)>0$ such that $|F|\ge c(n)\delta^{n+1}$ by a geometric consideration. Therefore, one has
\begin{align*}
\mu(B_{2^{j}\xi}(x,t))^{p'-1}&\lesssim |F|^{-1}\int_{F}\sum_{\Theta^{\lambda}_{l}\in \Theta^{\lambda}}\mu(\Theta^{\lambda}_{l})^{p'-1}\textbf{1}_{\Theta^{\lambda}_{l}}(x,t)d\lambda\\
&\lesssim \delta^{-(n+1)}\int_{|\lambda|\lesssim \delta}\sum_{\Theta^{\lambda}_{l}\in \Theta^{\lambda}}\mu(\Theta^{\lambda}_{l})^{p'-1}\textbf{1}_{\Theta^{\lambda}_{l}}(x,t)d\lambda,
\end{align*}
with $l=2^{j+1}$ and $B_{2^{j}\xi}(x,t)\subset \Theta^{\lambda}_{l}$, which clearly implies
\begin{align*}
H_{p,\delta}\mu(x,t)&\lesssim \delta^{-(n+1)}\int_{|\lambda|\lesssim\delta} \sum_{j=-\infty}^{i}\sum_{l=2^{j+1}}\left(\frac{\mu(\Theta^{\lambda}_{l})}{(2^{j}\xi)^{n}}\right)^{p'-1}\textbf{1}_{\Theta^{\lambda}_{l}}(x,t)ds\lesssim \delta^{-(n+1)}\int_{|\lambda|\lesssim\delta}H_{p}^{d,\lambda}\mu(x,t)d\lambda
\end{align*}
as the desired inequality $(\ref{4.26})$. This, along with H\"{o}lder's inequality and the Fubini theorem,  implies
\begin{align*}
&\int_{\mathbb{R}_{+}^{n+1}}\left(H_{p,\delta}\mu(x,t)\right)^{{q(p-1)}/{(p-q)}}d\mu(x,t)\\
& \lesssim \int_{\mathbb{R}_{+}^{n+1}} \left\{\frac{1}{\delta^{(n+1)}}\left(\int_{|\lambda|\lesssim C\delta} \left(H_{p}^{d,\lambda}\mu(x,t)\right)^{{q(p-1)}/{(p-q)}}d\lambda\right)^{{(p-q)}/{[q(p-1)]}}
\left(\int_{|\lambda|\lesssim\delta}d\lambda\right)^{{p(q-1)}/{[q(p-1)]}}\right\}^{{q(p-1)}/{(p-q)}}d\mu\\
& \lesssim\delta^{-(n+1)}\int_{|\lambda|\lesssim\delta}\left( \int_{\mathbb{R}_{+}^{n+1}}\left(H_{p}^{d,\lambda}\mu(x,t)\right)^{{q(p-1)}/{(p-q)}}d\mu\right) d\lambda\\
& \le C(n).
\end{align*}
In the last step, we used $(\ref{4.25}).$ 
The constant $C(n)$ is independent of $\delta$.
$(\ref{4.25'})$ follows readily from the monotone convergence theorem and the above inequality by letting $\delta\rightarrow \infty$.

{\it Step 2} - showing that
$
H_{p}\mu \in L_\mu^{q(p-1)/(p-q)}(\mathbb{R}_{+}^{n+1})\Rightarrow\eqref{5}.
$

Utilizing Lemma \ref{l22}, it is sufficient to show that  $$H_{p}\mu \in L_\mu^{q(p-1)/(p-q)}(\mathbb{R}_{+}^{n+1})
\Rightarrow\int_{\mathbb{R}_{+}^{n+1}}H_{p}(fd\mu)(x,t)f(x,t)d\mu(x,t)\lesssim \|f\|_{L^{q'}_{\mu}(\mathbb{R}_{+}^{n+1})}^{p'} \quad\forall  f\in L_\mu^{q'}(\mathbb{R}_{+}^{n+1}),
$$
since \eqref{5} is equivalent to
$$
\|P_{\alpha}^{*}(fd\mu)\|_{L^{p'}(\mathbb{R}^{n})}\lesssim \|f\|_{L_\mu^{q'}(\mathbb{R}_{+}^{n+1})}.
$$
Without loss of generality, we assume $f\geq 0$ in the following analysis. Let
$$
M_{\mu}f(x,t)=\sup_{r>0}\frac{1}{\mu(B_{r}(x,t))}\int_{B_{r}(x,t)}f(y,s)d\mu(y,s)
$$
denote the centered Hardy-Littlewood maximal function of $f$ with respect to $\mu$. Then the H\"{o}lder inequality shows
\begin{align*}
&\int_{\mathbb{R}_{+}^{n+1}}H_{p}(fd\mu)(x,t)f(x,t)d\mu(x,t)\\
&\ \ \lesssim \int_{\mathbb{R}_{+}^{n+1}}\left(M_{\mu}f(x,t)\right)^{p'-1}H_{p}\mu(x,t)f(x,t)d\mu(x,t)\\
&\ \ \lesssim \left(\int_{\mathbb{R}_{+}^{n+1}}\left(M_{\mu}f(x,t)\right)^{q'}d\mu(x,t)\right)^{{(p'-1)}/{q'}} \left(\int_{\mathbb{R}_{+}^{n+1}}\left(f(x,t)H_{p}\mu(x,t)\right)^{{q'}/{(q'-p'+1)}}d\mu(x,t)\right)^{{(q'-p'+1)}/{q'}}
\end{align*}
via the following observation
\begin{align*}
H_{p}(fd\mu)(x,t)
&\approx\int_{0}^{\infty} \left(\frac{\mu(B_{r}(x,t))}{r^{n}}\right)^{p'-1}\left(\frac{1}{\mu(B_{r}(x,t))}\int_{B_{r}(x,t)}f(x,t)d\mu(x,t)\right)^{p'-1}\frac{dr}{r}\\
&\lesssim \left(M_{\mu}f(x,t)\right)^{p'-1}H_{p}\mu(x,t).
\end{align*}

We conclude from the $L_\mu^{q'}(\mathbb{R}_{+}^{n+1})$-boundedness of $M_{\mu}$ (cf.\cite{F}) and a further use of the H\"{o}lder inequality that
$$
\int_{\mathbb{R}_{+}^{n+1}}H_{p}(fd\mu)(x,t)f(x,t)d\mu(x,t)\lesssim \|f\|_{L_\mu^{q'}(\mathbb{R}_{+}^{n+1})}^{p'}\left(\int_{\mathbb{R}_{+}^{n+1}}\left(H_{p}\mu\right)^{{q(p-1)}/{(p-q)}}d\mu(x,t)\right)^{{(p-q)}/{[q(p-1)]}}
$$
as desired.
\end{proof}

\section{Embeddings  of $\dot{W}^{\beta,p}(\mathbb{R}^n)$ to $L^q(\mathbb{R}^{n+1}_+,\mu)$}\label{sec-4}

In this section, we will characterize the embedding (\ref{4})
$$  \|P_{\alpha} f(\cdot,\cdot)\|_{L^{q}(\mathbb{R}^{n+1}_+,\mu)}\lesssim \|f\|_{\dot{W}^{\beta,p}(\mathbb{R}^{n})}
$$
in terms of the capacity and the fractional perimeter of open balls.
\subsection{Embeddings  of $\dot{W}^{\beta,p}(\mathbb{R}^n)$ to $L^q(\mathbb{R}^{n+1}_+,\mu)$ when $1\leq p\leq \min\{n/\beta,q\}$}
\begin{proposition}\label{proposition 1}
Let $\beta\in (0,n), 1\leq p\leq \min\{n/\beta,q\},$  and $\mu\in\mathcal M_{+}(\mathbb R^{n+1}_{+}).$ Then  the following   statements  are equivalent:
    \item{\rm (i)}  $\|P_{\alpha} f(\cdot,\cdot)\|_{L^{q,p}(\mathbb{R}^{n+1}_+,\mu)}\lesssim \|f\|_{\dot{W}^{\beta,p}(\mathbb{R}^{n})} \quad \forall\ f\in C_0^{\infty}(\mathbb{R}^{n});$

    \item{\rm (ii)}  $\|P_{\alpha} f(\cdot,\cdot)\|_{L^{q}(\mathbb{R}^{n+1}_+,\mu)}\lesssim \|f\|_{\dot{W}^{\beta,p}(\mathbb{R}^{n})} \quad \forall\ f\in C_0^{\infty}(\mathbb{R}^{n});$
\item{\rm (iii)} $\|P_{\alpha} f(\cdot,\cdot)\|_{L^{q,\infty}(\mathbb{R}^{n+1}_+,\mu)}\lesssim \|f\|_{\dot{W}^{\beta,p}(\mathbb{R}^{n})} \quad \forall\ f\in C_0^{\infty}(\mathbb{R}^{n});$
\item{\rm (iv)} $\sup\limits_{t>0}\frac{t^{p/q}}{c^\beta_p(\mu;t)}<\infty;$

\item{\rm (v)}  $(\mu(T(O)))^{p/q}\lesssim Cap_{\mathbb{R}^n}^{\beta,p}(O)$  holds for any bounded open set $ O\subseteq \mathbb{R}^n.$

\end{proposition}

\begin{proof}
The implications (i) $\Longrightarrow $ (ii) $ \Longrightarrow$ (iii) can be deduced from
$$\left(s^q\mu(L^{\alpha,\beta}_{s}(f))\right)^{p/q}
\leq \left(q\int_0^\infty \mu(L^{\alpha,\beta}_{s}(f))s^{q-1}
ds \right)^{p/q}\leq \int_0^\infty \left(\mu(L^{\alpha,\beta}_{s}(f))\right)^{p/q}ds^p$$
since
$$q\mu(L^{\alpha,\beta}_{s}(f))s^{q-1}
\leq
\frac{d}{ds}
\left(\int_0^s (\mu(L^{\alpha,\beta}_{t}(f)))^{p/q}dt^p
\right)^{q/p}.$$
Now, we prove (iii) $\Longrightarrow$ (v) $\Longrightarrow$ (i). If (iii) is true,
$$K_{p,q}(\mu)=\sup_{f\in C^\infty_0(\mathbb{R}^n)\ \& \ \|f\|_{\dot{W}^{\beta, p}(\mathbb{R}^n)}>0 }\frac{\sup_{s>0}s\left(\mu\left(\left\{(x,t)\in \mathbb{R}^{n+1}_+:\ |P_{\alpha} f(x,t)|>s\right\}\right)\right)^{1/q}}{\|f\|_{\dot{W}^{\beta, p}(\mathbb{R}^n)}}<\infty.$$
 Then, for any $f \in C^\infty_0(\mathbb{R}^n)$ and any open set $O\subseteq \hbox{ Int }(\{x\in \mathbb{R}^n: f(x)\geq 1\}),$  (iv) of  Lemma \ref{lemma 2} implies
$T(O)\subset L^{\alpha,\beta}_{\eta_{n,\alpha}}(f), $ thus
\begin{equation}\label{51}
\mu(T(O))\leq \mu(L^{\alpha,\beta}_{\eta_{n,\alpha}}(f))\lesssim (K_{p,q}(\mu)\|f\|_{\dot{W}^{\beta, p}(\mathbb{R}^n)})^q.\end{equation}
Thus (v) holds.

For (v) $\Longrightarrow$ (i), denote
$$Q_{p,q}(\mu):=\sup\left\{\frac{(\mu(T(O)))^{p/q}}{Cap_{\mathbb{R}^n}^{\beta,p}(O)}:\ \text{ bounded open }  O\subseteq \mathbb{R}^n\right\}<\infty.$$ Lemmas \ref{lemma 2}\ \&\  \ref{lemma 4} imply
\begin{eqnarray*}
&&\int_0^\infty(\mu(L^{\alpha,\beta}_s)(f)\cap T(B(0,k)))^{p/q}ds^p\\
&&\quad\leq \int_0^\infty(\mu(T(R^{\alpha,\beta}_s(f)\cap B(0,k)))^{p/q}ds^p \\
&&\quad\leq \int_0^\infty(\mu(T(\{x\in \mathbb{R}^n: \theta_\alpha\mathcal{M}(f)(x)>s\}\cap B(0,k)))^{p/q}ds^p \\
%&\quad\leq Q_{p,q}(\mu)\int_0^\infty(Cap_{\mathbb{R}^n}^{\beta,p}(\{x\in \mathbb{R}^n: \theta_\alpha\mathcal{M}(f)(x)>s\}\cap B(0,k)))ds^p \\
&&\quad\leq Q_{p,q}(\mu)\int_0^\infty(Cap_{\mathbb{R}^n}^{\beta,p}(\{x\in \mathbb{R}^n: \theta_\alpha\mathcal{M}(f)(x)>s\}))ds^p \\
&&\quad\leq Q_{p,q}(\mu)\|f\|^p_{\dot{W}^{\beta,p}(\mathbb{R}^n)}
\end{eqnarray*}
for any $f\in C^\infty_0(\mathbb{R}^n).$  Letting $k\longrightarrow \infty$ reaches (i).

Now, we will prove  (iii) $\Longrightarrow$ (iv) $\Longrightarrow$ (i).
If (iii) is true, then (\ref{51}) implies
$t^{p/q}\lesssim (K_{p,q}(\mu))^pCap_{\mathbb{R}^n}^{\beta,p}(O)$ whenever $t\in (0,\mu(T(O))).$ So, $t^{p/q}\lesssim (K_{p,q}(\mu))^pc^\beta_p(\mu;t).$
Thus, (iv) is true.

Assume (iv) is true.
By Lemma \ref{lemma 2} and Lemma \ref{lemma 4}, for any $f\in C^\infty_0(\mathbb{R}^n)$,
\begin{eqnarray*}
&&\int_0^\infty(\mu(L^{\alpha,\beta}_s)(f)\cap T(B(0,k)))^{p/q}ds^p\\
&&\quad \leq \int_0^\infty\frac{(\mu(L^{\alpha,\beta}_s(f)\cap B(0,k)))^{p/q}}{c^\beta_p(\mu;\mu(L^{\alpha,\beta}_s(f)\cap B(0,k)) )} Cap_{\mathbb{R}^n}^{\beta,p}\left(R^{\alpha,\beta}_s(f)\cap B(0,k)\right)ds^p \\
&&\quad\lesssim \sup_{t>0}\frac{t^{p/q}}{c^\beta_p(\mu;t)} \int_0^\infty Cap_{\mathbb{R}^n}^{\beta,p}\left(\left\{x\in \mathbb{R}^n: \theta_\alpha\mathcal{M}(f)(x)>s\right\}\cap B(0,k)\right)ds^p \\
&&\quad\lesssim \sup_{t>0}\frac{t^{p/q}}{c^\beta_p(\mu;t)} \|f\|^p_{\dot{W}^{\beta,p}(\mathbb{R}^n)},
\end{eqnarray*}
which gives (i) via letting $k\longrightarrow \infty$.
\end{proof}

\begin{remark}\label{4.2}
Given $\beta\in (0,n),$ $1=p\leq \min\{n/\beta,q\}<\infty$ or $1<p<\min\{n/\beta,q\}.$  Following  the idea of Xiao in \cite[Theorem 4.2]{Xiao 1} or  that of  Zhai in  \cite[Theorem 1.4]{Zhai}, we can deduce  from Lemma \ref{lemma 1}  that  the condition (v) in the equivalence of  (i)$\equiv$(v) of Proposition
\ref{proposition 1} can be replaced by   \begin{equation}\label{4.111}
    \sup_{x\in\mathbb{R}^n,r>0}\frac{(\mu(T(B(x,r))))^{1/q}}{Cap_{\mathbb{R}^n}^{\beta,p}(B(x,r))}<\infty.
\end{equation}
Specially, when $1=p\leq\min\{n/\beta, q\}<\infty$ and $\beta\in (0,n),$ the equivalence  $Cap_{\dot{W}^{\beta,1}}\approx H ^{n-\beta}_\infty$   implies that  $Cap_{\mathbb{R}^n}^{\beta,p}(B(x,r))$  in (\ref{4.111}) can be replaced by   $H ^{n-\beta}_\infty(B(x,r)).$ Here $H ^{d}_\infty(\cdot)$ denotes the  $d-$dimensional Hausdorff capacity.

\end{remark}

For the endpoint cases $p=1$ and $\beta\in(0,1),$ based on \cite[Theorem 2]{Xiao 2},  we can replace condition (v) of Proposition \ref{proposition 1} by a condition in terms of the fractional perimeter of bounded open sets. 

\begin{theorem}\cite[Theorem 2]{Xiao 2}\label{Them 2 Jie Xiao}
If $K$ is a compact subset of $\mathbb{R}^n,$ then
$$Cap_{\mathbb{R}^n}^{\beta,1}(K)=2\inf_{O\in O^\infty(K)}Per_\beta(O),$$ where $O^\infty(K)$ is the class of all open sets with $C^\infty$ boundary that contain $K.$
\end{theorem}

\begin{theorem}\label{them 4}
Let $\beta\in (0,1), 1=p\leq q< \infty$ and $\mu\in\mathcal M_{+}(\mathbb R^{n+1}_{+}).$ Then
 (\ref{4}) is equivalent to
\begin{equation}\label{periOpen}
\mu(T(O))^{1/q}\lesssim Per
_{\beta}(O)
\end{equation}
for all bounded open sets $ O\subseteq \mathbb{R}^n.$
%\begin{enumerate}((a))
    %\item $$ \|P_{\alpha} f(x,t)\|_{L^{q}(\mathbb{R}^{n+1}_+,\mu)}\lesssim \|f\|_{\dot{W}^{\beta,1}(\mathbb{R}^{n})}, \quad \forall f\in C_0^{\infty}(\mathbb{R}^{n})$$
%\item
%\end{enumerate}
\end{theorem}
\begin{proof}

It follows from  Proposition \ref{proposition 1} that  (\ref{4}) is equivalent to
\begin{equation}\label{8}
(\mu(T(O)))^{1/q}\lesssim Cap_{\mathbb{R}^n}^{\beta,1}(O)
\end{equation}
for all bounded open sets $O\subseteq \mathbb{R}^n.$ Thus it suffices to show that $(\ref{8})\Longrightarrow (\ref{periOpen})\Longrightarrow (\ref{4}).$

Firstly, we show that $(\ref{8})\Longrightarrow (\ref{periOpen}).$ Ponce-Spector in  \cite{Ponce Spector} proved that
$$H^{n-\beta}_\infty(O)\lesssim Per_\beta(O)$$
holds for every bounded open sets $O\subseteq \mathbb{R}^n.$
%Here $H^{d}_\infty(\cdot) $ is the $d-$dimensional Hausdorff capacity.
According to \cite{Adams 1989, Adams 1998}, we have 
$Cap_{\mathbb{R}^n}^{\beta,1}(\cdot)\approx H^{n-\beta}_\infty(\cdot).$
 Then (\ref{periOpen}) follows from
$$(\mu(T(O)))^{1//q}\lesssim Cap_{\mathbb{R}^n}^{\beta,1}(O)\lesssim H^{n-\beta}_\infty(O)\lesssim Per_\beta(O).$$

For $(\ref{periOpen})\Longrightarrow(\ref{4}),$ 
denote
$$Q_{q}(\mu)\coloneqq\sup\frac{(\mu(T(O)))^{1/q}}{Per_\beta(O)}<\infty,$$
where the supermum is taken over all bounded open sets $O\subseteq \mathbb{R}^n.$
For any $f\in C^\infty_0(\mathbb{R}^n)$ and any integer $k,$ denote $$V=\{x\in \mathbb{R}^n: \theta_\alpha\mathcal{M}(f)(x)>s\}\cap B(0,k).$$
It follows from Theorem \ref{Them 2 Jie Xiao} that $$Cap_{\mathbb{R}^n}^{\beta,1}(\overline{V})=2\inf_{O\in O^\infty(K)}Per_\beta(\overline{V}),$$ where $O^\infty(K)$ is the class of all open sets with $C^\infty$ boundary that contain $\overline{V}.$
Thus, we have \begin{eqnarray*}
\mu(T(V))^{1/q}&\leq&\mu(T(\overline{V}))^{1/q}\\
&\leq&
\inf_{O\in O^\infty(\overline{V})}\mu(T(O))^{1/q}\\
&\leq& Q_{q}(\mu)\inf_{O\in O^\infty(\overline{V})}Per_\beta(O) \\ &=&\frac{Q_{q}(\mu)}{2}Cap_{\mathbb{R}^n}^{\beta,1}(\overline{V})\\
&\lesssim& Cap_{\mathbb{R}^n}^{\beta,1}(\{x\in \mathbb{R}^n: \theta_\alpha\mathcal{M}(f)(x)\geq s\}).
\end{eqnarray*}

It follows from   Lemma \ref{lemma 2} and Lemma \ref{lemma 4} that, for any $f\in C^\infty_0(\mathbb{R}^n)$,
\begin{eqnarray*}
\int_0^\infty\left(\mu(L^{\alpha,\beta}_s)(f)\cap T(B(0,k))\right)^{1/q}ds
& \leq&\int_0^\infty(\mu(T(R^{\alpha,\beta}_s(f)\cap B(0,k)))^{1/q}ds\\
& \leq & \int_0^\infty(\mu(T(\{x\in \mathbb{R}^n: \theta_\alpha\mathcal{M}(f)(x)>s\}\cap B(0,k)))^{1/q}ds \\
&  \lesssim& \int_0^\infty Cap_{\mathbb{R}^n}^{\beta,1}(\{x\in \mathbb{R}^n: \theta_\alpha\mathcal{M}(f)(x)\geq s\})ds \\
& \lesssim&  \|f\|_{\dot{W}^{\beta,1}(\mathbb{R}^n)},
\end{eqnarray*}
which   reaches (\ref{4}) by letting $k\longrightarrow \infty.$

\end{proof}

\begin{remark}
It follows from Remark \ref{4.2} and $Cap_{\mathbb{R}^n}^{\beta,p}(B(x,r))= r^{n-\beta}Per_\beta(B(0,1)),$
see Xiao \cite[Themrem 4]{Xiao 2}, that  $Per_\beta(O)$  in (\ref{periOpen}) of Theorem   \ref{them 4}
 can be  also  replaced by $r^{n-\beta}Per_\beta(B(0,1))$  when  $\beta\in (0,1).$
\end{remark}

\subsection{Embeddings  of $\dot{W}^{\beta,p}(\mathbb{R}^n)$ to $L^q(\mathbb{R}^{n+1}_+,\mu)$ when $\max\{q,1\}<p\leq n/ \beta$}

\begin{theorem}\label{them 1}
Let $\beta\in (0,n), 0<q<p,1<p\leq n/\beta$ and $\mu\in\mathcal M_{+}(\mathbb R^{n+1}_{+}). $ Then (\ref{4}) is equivalent to
\begin{equation}\label{integralCondition}
    \int_0^\infty\left(\frac{t^{p/q}}{c^\beta_{p}(\mu;t)}\right)^{q/(p-q)}\frac{dt}{t}<\infty.
\end{equation}
\end{theorem}

\begin{proof}
$(\ref{integralCondition})\Rightarrow (\ref{4}).$ If
$$I_{p,q}(\mu)\coloneqq\int_0^\infty\left(\frac{t^{p/q}}{c^\beta_{p}(\mu;t)}\right)^{q/(p-q)}\frac{dt}{t}<\infty,$$
then for each $f\in C^\infty_0(\mathbb{R}^n),$ each $j=0,\pm 1, \pm 2,\ldots$, and each natural number $k,$  Lemma \ref{lemma 2} (iii) implies
$$Cap_{\mathbb{R}^n}^{\beta,p}\left(R^{\alpha,\beta}_{2^j}(f)\cap(B(0,k))\right)\leq Cap_{\mathbb{R}^n}^{\beta,p}
\left(\left\{x\in \mathbb{R}^n: \theta_\alpha\mathcal{M}f(x)>2^j \right\}\cap B(0,k)\right).$$
Define
$$\left\{\begin{aligned}
&\mu_{j,k}(f)=\mu\left(T(R^{\alpha,\beta}_{2^j}(f)\cap(B(0,k)))\right),\\
&S_{p,q,k}(\mu;f)=\sum_{j=-\infty}^\infty\frac{(\mu_{j,k}(f)-\mu_{j+1,k}(f))^{{p}/{(p-q)}}}
{(Cap_{\mathbb{R}^n}^{\beta,p}(R^{\alpha,\beta}_{2^j}(f)\cap(B(0,k))))^{{q}/{(p-q)}}}.
\end{aligned}\right.$$
It follows from (ii) of Lemma \ref{lemma 2} that
\begin{eqnarray*}
\left(S_{p,q,k}(\mu;f)\right)^{{(p-q)}/{p}}
&\lesssim &
\left(\sum_{j=-\infty}^\infty\frac{\mu^{{q}/{(p-q)}}_{j,k}(f)-\mu^{{p}/{(p-q)}}_{j+1,k}(f)}{\left(c^\beta_{p}(\mu;\mu_{j,k}(f))
\right)^{{q}/{(p-q)}}}
\right)^{{(p-q)}/{p}}\\
&\lesssim &
\left(\int_0^\infty\frac{1}{\left(c^\beta_{p}(\mu;s)\right)^{{q}/{(p-q)}}}ds^{p/(p-q)}\right)^{{(p-q)}/{p}}\\
&\simeq& (I_{p,q}(\mu))^{{(p-q)}/{p}}.
\end{eqnarray*}

On the other hand,  it follows from the H\"{o}lder inequality, (ii) of Lemma \ref{lemma 4} and (ii)-(iii) of Lemma \ref{lemma 2} that
\begin{eqnarray*}
&&\int_{T(B(0,k))}|P_{\alpha} f(x,t)|^{q}d\mu(x,t)\\
&&\quad= \int_0^\infty
\mu(L^{\alpha,\beta}_s(f)\cap T(B(0,k)))ds^q\\
&&\quad\lesssim \sum_{j=-\infty}^\infty (\mu_{j,k}(f)-\mu_{j+1,k}(f))2^{jq}\\
&&\quad\lesssim\left(S_{p,q,k}(\mu;f)\right)^{{(p-q)}/{p}}\left(\sum_{j=-\infty}^{\infty} 2^{jp}Cap_{\mathbb{R}^n}^{\beta,p}(R^{\alpha,\beta}_{2j}(f)\cap(B(0,k))\right)^{q/p} \\
&&\quad\lesssim\left(S_{p,q,k}(\mu;f)\right)^{{(p-q)}/{p}}\left(\int_0^\infty Cap_{\mathbb{R}^n}^{\beta,p}\left(\left\{x\in \mathbb{R}^N: \theta_\alpha\mathcal{M}f(x)>s\right\}\cap(B(0,k)\right)ds^p\right)^{q/p} \\
&&\quad\lesssim\left(S_{p,q,k}(\mu;f)\right)^{{(p-q)}/{p}}\|f\|^q_{\dot{W}^{\beta,p}(\mathbb{R}^n)}.
\end{eqnarray*}
So, we get
$$\left(\int_{T(B(0,k))}|P_{\alpha} f(x,t)|^{q}d\mu(x,t)\right)^{1/q}\lesssim (I_{p,q}(\mu))^{{(p-q)}/{pq}}\|f\|_{\dot{W}^{\beta,p}(\mathbb{R}^n)}.$$
Letting $k\longrightarrow \infty$ derives (\ref{4}).

$(\ref{4})\Rightarrow(\ref{integralCondition}). $ If (\ref{4}) holds, then
$$C_{p,q}(\mu)\coloneqq\sup_{f\in C_0^\infty(\mathbb{R}^n),\ \|f\|_{\dot{W}^{\beta,p}(\mathbb{R}^n)}>0}\frac{1}{\|f\|_{\dot{W}^{\beta,p}(\mathbb{R}^n)}}\left(\int_{\mathbb{R}^{n+1}_+}|P_{\alpha} f(x,t)|^{q} d\mu(x,t)\right)^{1/q}<\infty.$$
Then for each  $f\in C_0^\infty(\mathbb{R}^n)$ with $ \|f\|_{\dot{W}^{\beta,p}(\mathbb{R}^n)}>0,$ there holds
$$\left(\int_{\mathbb{R}^{n+1}_+}|P_{\alpha} f(x,t)|^{q}d\mu(x,t)\right)^{1/q}\leq C_{p,q}(\mu)\|f\|_{\dot{W}^{\beta,p}(\mathbb{R}^n)},$$
which implies by Markov's inequality
$$\sup_{s>0}s\left(\mu(L^{\alpha,\beta}_s(f))\right)^{1/q}\lesssim C_{p,q}(\mu)\|f\|_{\dot{W}^{\beta,p}(\mathbb{R}^n)}.$$
This, together with (iv) of Lemma \ref{lemma 2}, implies that for fixed  $f\in C_0^\infty(\mathbb{R}^n),$
$$\mu(T(O))\lesssim   C^q_{p,q}(\mu)\|f\|^q_{\dot{W}^{\beta,p}(\mathbb{R}^n)}$$
for any  bounded open set $O\subseteq \hbox{Int}(\{x\in \mathbb{R}^n: f(x)\geq 1\}).$
The definition of $c^\beta_p(\mu;t)$ implies that $ c^\beta_p(\mu;t)>0$, and for $t\in (0,\infty)$ and every $j$ there exists a bounded open set $O_j\subseteq \mathbb{R}^n$ such that
$Cap_{\mathbb{R}^n}^{\beta,p}(O_j)\leq 2  c^\beta_p(\mu;2^j)$
and $\mu(T(O_j))>2^j.$

When $p\in (1,n/\beta)$, according to \cite[Theorem 2, P. 347]{Mazya}, $f\in \dot{W}^{\beta,p}(\mathbb{R}^n)$ when $p\in (1,n/\beta)$ if and only if $f=c I_{\beta}\ast g(x)$ for some $f\in L^{p}(\mathbb{R}^n).$ Here $I_{\beta}\ast g(x)=\int_{\mathbb{R}^n}\frac{f(y)}{|x-y|^{n-\beta}}dy.$ Thus, $\|f\|_{\dot{W}^{\beta,p}(\mathbb{R}^n)}=\|(-\triangle)^{\beta/2}f\|_{L^p(\mathbb{R}^n)}\approx\|(-\triangle)^{\beta/2}(I_{\beta}\ast g)\|_{L^p(\mathbb{R}^n)}\approx\| g\|_{L^p(\mathbb{R}^n)}.$ So, we have $$Cap_{\mathbb{R}^n}^{\beta,p}(S)\approx \inf\left\{\|g\|^p_{L^p(\mathbb{R}^n)}: g\in C^\infty_0(\mathbb{R}^n), g\geq 0, S\subseteq \hbox{Int}(\{x\in \mathbb{R}^n: I_{\beta}\ast g(x)\geq 1\})\right\}.$$
Then, there exists $g_j\in C^\infty_0(\mathbb{R}^n)$ such that $ I_{\beta}\ast g_j(x)\geq 1$ on $ O_j,$ and $$\|g_{j}\|^p_{L^p(\mathbb{R}^n)}\leq 2Cap_{\mathbb{R}^n}^{\beta,p}(O_j)\leq 4 c^\beta_p(\mu;2^j).$$
For the integers $ i,k$ with $i<k,$ define
$$g_{i,k}=\sup_{i\leq j\leq k}\left(\frac{2^j}{c^\beta_p(\mu;2^j)}\right)^{\frac{1}{p-q}} g_j.$$
Then $g_{i,k}\in L^p(\mathbb{R}^n)\cap L^1_{loc}(\mathbb{R}^n)$
with 
\begin{equation}\label{normofgik}
\|g_{i,k}\|^p_{L^p(\mathbb{R}^n)}\lesssim \sum_{j=i}^k\left(\frac{2^j}{c^\beta_p(\mu;2^j)}\right)^{\frac{p}{p-q}} c^\beta_p(\mu;2^j).\end{equation}
On the other hand,  $I_{\beta}\ast g_{i,k}\in \dot{W}^{\beta,p}(\mathbb{R}^n)$
  since $I_\beta:  \dot{F}_{p,2}^0(\mathbb{R}^n)=L^p(\mathbb{R}^n) \to	 \dot{F}_{p,2}^{\beta}=\dot{W}^{\beta,p}(\mathbb{R}^n), $ see \cite[Theorem 1, page 242]{Hans} 

Since for $i\leq j\leq k,$
$$x\in O_j\Rightarrow  I_{\beta}\ast g_{i,k}(x)\geq \left(\frac{2^j}{c^\beta_p(\mu;2^j)}\right)^{\frac{1}{p-q}},$$
(iv) of Lemma \ref{lemma 2} implies
that there exists a constant $\eta_{n,\alpha}$ such that
$$(x,t)\in T(O_j)\Rightarrow |u_{i,k}(x,t)|\geq \left(\frac{2^j}{c^\beta_p(\mu;2^j)}\right)^{\frac{1}{p-q}}\eta_{n,\alpha}.$$
Here $u_{i,k}$ denotes the Caffarelli-Silvestre extension  of $I_{\beta}\ast g_{i,k}.$
Thus,  with $s=\left(\frac{2^j}{c^\beta_p(\mu;2^j)}\right)^{\frac{1}{p-q}}\frac{\eta_{n,\alpha}}{2},$
$$2^j\leq \mu(T(O_j))\leq \mu (L^{\alpha,\beta}_{s}(I_\beta\ast g_{i,k})).$$
Similar to that in the the proof of (i) in Theorem \ref{Them 3.4},  we can get
\begin{eqnarray*}
(C_{p,q}(\mu)\|g_{i,k}\|_{L^p(\mathbb{R}^n)})^q
&\gtrsim&
\int_{\mathbb{R}^{n+1}_+}|u_{i,k}(x,t)|d\mu\\
&\approx& \int_0^\infty\left(\inf\left\{s: \mu(L^{\alpha,\beta}_s)(I_\beta\ast g_{i,k})\leq s\right\}\right)^qds\\
&\gtrsim&\sum_{j=i}^k\left(\inf\left\{s: \mu(L^{\alpha,\beta}_s)(I_\beta\ast g_{i,k})\leq2^j \right\}\right)^q2^j\\
&\gtrsim&
\sum_{j=i}^k\left(\frac{2^j}{c^\beta_p(\mu;2^j)}\right)^{\frac{q}{p-q}}2^j\\
&\gtrsim&
\left(\sum_{j=i}^k\frac{2^{{jp}/{(p-q)}}}{c^\beta_p(\mu;2^j)^{{q}/{(p-q)}}}\right)^{\frac{p-q}{p}}\|g_{i,k}\|^q_{L^p(\mathbb{R}^n)},
\end{eqnarray*}
where in the last step, we used (\ref{normofgik}).
Thus, $$\sum_{j=i}^k\frac{2^{{jp}/{(p-q)}}}{c^\beta_p(\mu;2^j)^{{q}/{(p-q)}}}\lesssim(C_{p,q}(\mu))^{{pq}/{(p-q)}}.$$

When $p={n}/{\beta},$ the definition of $Cap_{\mathbb{R}^n}^{\beta,p}(\cdot)$ implies that there exists a positive $f_j\in C^\infty_0(\mathbb{R}^n)$ such that  $f_j\geq 1$ on $O_j$ and $$\|f_j\|^p_{\dot{W}^{\beta,p}(\mathbb{R}^n)}\leq 2 Cap_{\mathbb{R}^n}^{\beta,p}(O_j)\leq 4c^\beta_p(\mu;2^j).$$
Define $g_j(x,h)\coloneqq|h|^{-2\beta}\triangle^k_hf_j(x).$ Then $g_j\in L^p(\mathbb{R}^{2n})\cap L^1_{loc}(\mathbb{R}^{2n}). $ On the other hand, \cite[Theorem 5.1]{Adams} implies $f_j(x)=CI_{2\beta}^{(2n)}\ast g_j(x,0),$ for a.e. $x\in \mathbb{R}^n.$

%We can define $g_{i,k}$ similar to the previous case. It is easy to show that   $g_{i,k}\in L^p(\mathbb{R}^{2n})$ and $I^{(2n)}_{2\beta}\ast g_{i,k} \in \dot{\mathcal{L}}^{p}_{2\beta}(\mathbb{R}^n).$ 

For the integers $ i,k$ with $i<k,$ define
$$g_{i,k}(x,h)=\sup_{i\leq j\leq k}\left(\frac{2^j}{c^\beta_p(\mu;2^j)}\right)^{\frac{1}{p-q}}  g_j(x,h).$$
Then $g_{i,k}\in L^{p}(\mathbb{R}^{2n})\cap L^1_{loc}(\mathbb{R}^{2n}),$ 
and Lemma \ref{lemma 3} implies
\begin{eqnarray*}
\|\mathcal{R}I^{(2n)}_{2\beta}\ast g_{i,k}\|^p_{\dot{W}^{\beta,p}(\mathbb{R}^n)}
&\lesssim& \sum_{j=i}^k\left(\frac{2^j}{c^\beta_p(\mu;2^j)}\right)^{\frac{p}{p-q}} \|I^{(2n)}_{2\beta}\ast g_{i,k}\|^p_{\dot{\mathcal{L}}^p(\mathbb{R}^{2n})}\\
&\lesssim& \sum_{j=i}^k\left(\frac{2^j}{c^\beta_p(\mu;2^j)}\right)^{\frac{p}{p-q}} \|f_j\|^p_{\dot{W}^{\beta,p}(\mathbb{R}^n)}\\
&\lesssim& \sum_{j=i}^k\left(\frac{2^j}{c^\beta_p(\mu;2^j)}\right)^{\frac{p}{p-q}} c^\beta_p(\mu;2^j).
\end{eqnarray*}

Let $u_{i,k}(x,t)$ be the Caffarelli-Silvestre extension of $\mathcal{R}(I^{(2n)}_{2\beta}g_{i,k})(x)$. Then, similar to the previous case, 
for $i\leq j\leq k,$
$$x\in O_j\Rightarrow   \mathcal{R}(I^{(2n)}_{2\beta}g_{i,k})(x)\geq \left(\frac{2^j}{c^\beta_p(\mu;2^j)}\right)^{\frac{1}{p-q}}.$$
 (iv) of Lemma \ref{lemma 2} implies
that there exists a constant $\eta_{n,\alpha}$ such that
$$(x,t)\in T(O_j)\Rightarrow |u_{i,k}(x,t)|\geq \left(\frac{2^j}{c^\beta_p(\mu;2^j)}\right)^{\frac{1}{p-q}}\eta_{n,\alpha}$$
and with $s=\left(\frac{2^j}{c^\beta_p(\mu;2^j)}\right)^{\frac{1}{p-q}}\frac{\eta_{n,\alpha}}{2},$
$$2^j\leq \mu(T(O_j))\leq \mu (L^{\alpha,\beta}_{s}(\mathcal{R}(I^{(2n)}_{2\beta}g_{i,k}))).$$

We can get
 \begin{eqnarray*}
(C_{p,q}(\mu)\|\mathcal{R}(I^{(2n)}_{2\beta}g_{i,k})\|_{\dot{W}^{\beta,p}(\mathbb{R}^n)})^q
&\gtrsim&
\int_{\mathbb{R}^{n+1}_+}|u_{i,k}(x,t)|d\mu\\
&\approx& \int_0^\infty\left(\inf\left\{s: \mu(L^{\alpha,\beta}_s(\mathcal{R}(I^{(2n)}_{2\beta}g_{i,k})))\leq s\right\}\right)^qds\\
&\gtrsim&\sum_{j=i}^k\left(\inf\left\{s: \mu(L^{\alpha,\beta}_s(\mathcal{R}(I^{(2n)}_{2\beta}g_{i,k})))\leq2^j \right\}\right)^q2^j\\
&\gtrsim&
\sum_{j=i}^k\left(\frac{2^j}{c^\beta_p(\mu;2^j)}\right)^{\frac{q}{p-q}}2^j\\
&\gtrsim&
\left(\sum_{j=i}^k\frac{2^{{jp}/{(p-q)}}}{c^\beta_p(\mu;2^j)^{{q}/{(p-q)}}}\right)^{\frac{p-q}{p}}\|\mathcal{R}(I^{(2n)}_{2\beta}g_{i,k}\|^q_{\dot{W}^{\beta,p}(\mathbb{R}^n)}.
\end{eqnarray*}
Thus, we have
 $$\sum_{j=i}^k\frac{2^{\frac{jp}{(p-q)}}}{c^\beta_p(\mu;2^j)^{\frac{q}{p-q}}}\lesssim\left(C_{p,q}(\mu)\right)^{\frac{pq}{p-q}}.$$
Letting $i,k\longrightarrow \infty,$ we reach
$$\int_0^\infty\left(\frac{t^{p/q}}{c^\beta_p(\mu;t)}\right)^{\frac{q}{p-q}}\lesssim \sum_{-\infty}^\infty\frac{2^{\frac{jp}{p-q}}}{(c^\beta_p(\mu;2^j))^{\frac{q}{(p-q)}}}\lesssim (C_{p,q}(\mu))^{\frac{pq}{p-q}},$$ which implies (\ref{integralCondition}).
\end{proof}

When $q<p=1,$ we can establish the following necessary conditions for the embedding (\ref{4}).

\begin{proposition}\label{them 5}
Let $\beta\in (0,n), 0<q<p=1$ and $\mu\in\mathcal M_{+}(\mathbb R^{n+1}_{+}). $ Then
(\ref{4}) $\Rightarrow$ (i) $\Rightarrow$ (ii) $\Rightarrow$ (iii):

%    \item  $$\|P_{\alpha} f(x,t)\|_{L^{q}(\mathbb{R}^{n+1}_+,\mu)}\lesssim \|f\|_{\dot{W}^{\beta,p}(\mathbb{R}^{n})}, \quad \forall f\in C_0^{\infty}(\mathbb{R}^{n}),$$
\item{\rm (i)} $\|P_{\alpha} f(\cdot,\cdot)\|_{L^{q,\infty}(\mathbb{R}^{n+1}_+,\mu)}\lesssim \|f\|_{\dot{W}^{\beta,1}(\mathbb{R}^{n})} \quad \forall f\in C_0^{\infty}(\mathbb{R}^{n});$
\item{\rm (ii)}
$\sup\left\{\frac{(\mu(T(O)))^{1/q}}{Cap_{\mathbb{R}^n}^{\beta,1}(O)}:  \text{ bounded open set} \,  O\subseteq \mathbb{R}^n\right\}<\infty;$
\item{\rm (iii)}  $\|P_{\alpha} f(\cdot,\cdot)\|_{L^{q,1}(\mathbb{R}^{n+1}_+,\mu)}\lesssim \|f\|_{\dot{W}^{\beta,1}(\mathbb{R}^{n})} \quad \forall f\in C_0^{\infty}(\mathbb{R}^{n}).$

\end{proposition}

\begin{proof}
The proofs of (\ref{4}) $\Longrightarrow$ (i) $\Longrightarrow$ (ii) are similar to those of  (ii) $\Longrightarrow$ (iii) $\Longrightarrow$ (v)  in the proof of  Proposition \ref{proposition 1}.   For the implication (ii) $\Longrightarrow$ (iii), firstly, there holds
\begin{eqnarray*}
\mu(L^{\alpha,\beta}_s(f))&\leq& \mu(T(R^{\alpha,\beta}_s(f)))\\
&\leq& \mu\left(T\left(\left\{x\in \mathbb{R}^n:\ \theta_\alpha\mathcal{M}f(x)>s\right\}\right)\right)\\
&\lesssim & \left(Cap_{\mathbb{R}^n}^{\beta,1}\left(\left\{x\in \mathbb{R}^n: \theta_\alpha\mathcal{M}f(x)>s\right\}\right)\right)^q,
\end{eqnarray*}
due to  Lemma \ref{lemma 2}. 
Then, applying (ii) of Lemma \ref{lemma 4}  we can show  (ii)$\Longrightarrow$(iii) in a similar way as the proof of (v)$\Longrightarrow$(i) in Proposition \ref{proposition 1}.

\end{proof}

\section*{Acknowledgement}
The authors are very grateful to the anonymous reviewers, especially one of whom gave numerous useful suggestions that lead to a substantial improvement of this manuscript.

%\end{multicols}

\end{document}